\documentclass[12pt,reqno]{amsart}

\usepackage{amsmath,amsthm}
\usepackage[margin=3cm]{geometry}
\usepackage{amssymb,mathtools}
\usepackage{latexsym}
\usepackage{hyperref}
\usepackage[all,cmtip,arrow,2cell]{xy}
\usepackage{enumitem}
\usepackage{eucal}

\numberwithin{equation}{section}

\usepackage{todonotes}
\usepackage{tikz-cd}
\usepackage{tikz-3dplot}
\usepackage{pgfplots}
\pgfplotsset{compat=1.16}
\usetikzlibrary{intersections}
\usetikzlibrary{positioning}
\usepgfplotslibrary{fillbetween}
\usetikzlibrary{arrows}
\usepackage{subcaption}
\usepackage{cite}

\theoremstyle{plain}
\newtheorem{Theorem}{Theorem}[section]
\newtheorem*{Theorem*}{Theorem}
\newtheorem{Proposition}[Theorem]{Proposition}
\newtheorem{Lemma}[Theorem]{Lemma}
\newtheorem{Corollary}[Theorem]{Corollary}

\theoremstyle{definition}
\newtheorem{Remark}[Theorem]{Remark}
\newtheorem{Definition}[Theorem]{Definition}
\newtheorem*{Definition*}{Definition}

\newtheorem{Questions}[Theorem]{Questions}
\newtheorem{Example}[Theorem]{Example}

\newtheorem{Terminology}[Theorem]{Terminology}

\newtheorem{Notation}[Theorem]{Notation}
\newtheorem{Caution}[Theorem]{Caution}

\newtheorem*{Acknowledgement*}{Acknowledgement}
\newtheorem*{Outline*}{Outline of the paper}
\newtheorem*{Conventions*}{Conventions and Notation}

\theoremstyle{remark}

\DeclareMathOperator{\Diff}{Diff}
\DeclareMathOperator{\Bis}{Bis}

\DeclareMathOperator{\Der}{Der}

\DeclareMathOperator{\Hom}{Hom}

\DeclareMathOperator{\End}{End}
\DeclareMathOperator{\Aut}{Aut}

\DeclareMathOperator{\Empty}{{\_\!\_\,}}
\DeclareMathOperator{\Comma}{\downarrow}

\DeclareMathOperator*{\Lim}{lim}

\DeclareMathOperator{\Ran}{Ran}

\DeclareMathOperator{\intDflg}{\underline{\mathcal{D}\mathrm{fl}}g}

\DeclareMathOperator{\intAut}{\underline{\mathrm{Aut}}}

\DeclareMathOperator{\Sym}{Sym}
\DeclareMathOperator{\Spec}{Spec}

\newcommand{\bbA}{\mathbb{A}}

\newcommand{\bbR}{\mathbb{R}}
\newcommand{\bbZ}{\mathbb{Z}}
\newcommand{\sDelta}{\boldsymbol{\Delta}}

\newcommand{\calA}{\mathcal{A}}
\newcommand{\calB}{\mathcal{B}}
\newcommand{\calC}{\mathcal{C}}

\newcommand{\calW}{\mathcal{W}} 
\newcommand{\calX}{\mathcal{X}}

\newcommand{\catA}{{\mathcal{A}}}
\newcommand{\catB}{{\mathcal{B}}}
\newcommand{\catC}{{\mathcal{C}}}
\newcommand{\catD}{{\mathcal{D}}}

\newcommand{\catI}{{\mathcal{I}}}

\newcommand{\frakg}{\mathfrak{g}}
\newcommand{\op}{{\mathrm{op}}}

\newcommand{\id}{\mathrm{id}}
\newcommand{\Id}{\mathrm{1}}
\newcommand{\pr}{\mathrm{pr}}

\newcommand{\Ad}{\mathrm{Ad}}

\newcommand{\Set}{{\mathcal{S}\mathrm{et}}}

\newcommand{\Mfld}{{\mathcal{M}\mathrm{fld}}}

\newcommand{\Eucl}{{\mathcal{E}\mathrm{ucl}}}

\newcommand{\Mod}{{\mathcal{M}\mathrm{od}}}

\newcommand{\Aff}{{\mathcal{A}\mathrm{ff}}}

\newcommand{\CAlg}{{\mathrm{C}\mathcal{A}\mathrm{lg}}}

\newcommand{\BunG}{\mathcal{B}\mathrm{un}_G}
\newcommand{\BunGdiff}{\mathcal{B}\mathrm{un}_G^\mathrm{diff}}

\newcommand{\VTan}[1]{{V^{[#1]}\!}}
\newcommand{\Zero}[1]{0^{[#1]}}
\newcommand{\ZeroDiag}[2]{0_{#1,#2}}

\newcommand{\Val}{\eta}
\newcommand{\Rmult}{\hat{m}}
\newcommand{\Radd}{\hat{+}}
\newcommand{\Runit}{\hat{1}}
\newcommand{\Rzero}{\hat{0}}

\newcommand{\eFun}[1]{{#1}}

\newcommand{\Tpart}[1]{T_{(#1)}}

\renewcommand{\phi}{\varphi}
\renewcommand{\epsilon}{\varepsilon}

\tikzset{bolli/.style={
    very thick,
    fill=lightgray,
    scale=1.3}}

\tikzset{point/.style={
    shape=circle,
    draw,
    fill=black,
    inner sep=1.5pt]}}

\tikzset{whitepoint/.style={
    shape=circle,
    draw,
    fill=white,
    inner sep=1.5pt]}}

\begin{document}

\title[Differentiable groupoid objects]{Differentiable groupoid objects and their\\ abstract Lie algebroids}

\author[L.~Aintablian]{Lory Aintablian}
\address{Max-Planck-Institut f\"ur Mathematik, Vivatsgasse 7, 53111 Bonn, Germany}
\email{lory.aintablian@gmail.com}

\author[C.~Blohmann]{Christian Blohmann}
\address{Max-Planck-Institut f\"ur Mathematik, Vivatsgasse 7, 53111 Bonn, Germany}
\email{blohmann@mpim-bonn.mpg.de}


%\date{November 27, 2018}
%\date{\today}


\begin{abstract}
The infinitesimal counterpart of a Lie groupoid is its Lie algebroid. As a vector bundle, it is given by the source vertical tangent bundle restricted to the identity bisection. Its sections can be identified with the invariant vector fields on the groupoid, which are closed under the Lie bracket. We generalize this differentiation procedure to groupoid objects in any category with an abstract tangent structure in the sense of Rosick\'{y} and a scalar multiplication by a ring object that plays the role of the real numbers. We identify the categorical conditions that the groupoid object must satisfy to admit a natural notion of invariant vector fields. Then we show that invariant vector fields are closed under the Lie bracket defined by Rosick\'{y} and satisfy the Leibniz rule with respect to ring-valued morphisms on the base of the groupoid. The result is what we define axiomatically as an abstract Lie algebroid, by generalizing the underlying vector bundle to a module object in the slice category over its base. Examples include diffeomorphism groups, bisection groups of Lie groupoids, the diffeological symmetry groupoids of general relativity (Blohmann/Fernandes/Weinstein), symmetry groupoids in Lagrangian Field Theory, holonomy groupoids of singular foliations, elastic diffeological groupoids, groupoid objects in differentiable stacks, and affine groupoid schemes.
\end{abstract} 
\maketitle

%\newpage\input{ResponseToReferee2}\newpage

\tableofcontents

%\makeatletter\providecommand\@dotsep{5}\makeatother\listoftodos\relax

% \include{01_Intro}
% \include{02_Tangent}
% \include{03_Groupoids}
% \include{04_GBundles}
% \include{05_LieAlgebroid}
% \include{06_Examples}
% \appendix
% \include{A_Appendix1}
% \include{B_Appendix2}

%\newpage

\section{Introduction}

\subsection{Goal of the paper}

A Lie group is a group object in the category $\Mfld$ of smooth finite-dimensional manifolds. To every Lie group, we can associate its Lie algebra, given by the right-invariant vector fields. In this sense, every group object in $\Mfld$ is differentiable. This is not true for every group\emph{oid} object in $\Mfld$. To construct its Lie algebroid, we have to assume that the source (and therefore target) map is a submersion. The groupoid objects that are differentiable in this sense are the Lie groupoids.

The category of smooth manifolds is inconvenient. Other than having coproducts and finite products, it does not possess many good categorical properties, which puts it at odds with many modern developments of mathematics that are structured and guided by category theoretic concepts. This has motivated the development of a plethora of convenient settings for differential geometry, which can be vaguely defined as categories that contain smooth manifolds as a full subcategory, but have better properties, such as having all small limits, colimits, exponential objects, etc. The price to be paid for this convenience is that such a category is usually too large as to allow for strong geometric results that hold for all its objects. The task is then to identify the structure of a category and the properties of its objects that are needed to generalize a certain differential geometric construction, in our case: the differentiation of Lie groupoids. This can be done for a particular category, e.g.~the category of diffeological spaces, Chen spaces, Fr\"ohlicher spaces, differentiable stacks, polyfolds, bornological spaces, etc. \cite{iglesias2013, BehrendXu:2011, Stacey2011}. Better still, we can try to identify the structure that any category must have and the properties that any object must have for the construction to work. In this paper, we ask and answer the following:

\begin{Questions}
\label{questions:Main}
Given a groupoid object $G$ in a category $\calC$, what are the structures of $\calC$ and the properties of $G$ needed for its differentiation? What are the infinitesimal objects in $\calC$ that generalize Lie algebroids? 
\end{Questions}

\subsection{Lie groupoids and Lie algebroids}

A \textbf{groupoid} is a small category in which every morphism is invertible. Explicitly, a groupoid consists of a set of objects $G_0$ and a set of arrows $G_1$ together with the source $s: G_1 \to G_0$, target $t: G_1 \to G_0$, multiplication $m: G_1 \times_{G_0}^{s,t} G_1 \to G_1$, unit $\Id: G_0 \to G_1$, and inverse $i: G_1 \to G_1$, satisfying the usual axioms of a category. 

A \textbf{Lie groupoid} is a groupoid internal to the category of smooth manifolds such that the source and target maps are submersions. Lie groupoids are many-unit generalizations of Lie groups. They describe local and global symmetries of geometric structures in more general cases \cite{weinstein1996}. A large class of geometric structures can be described via Lie groupoids, such as representations, equivalence relations, group actions, foliations, orbifolds, and differentiable stacks \cite{mackenzie2005, models, moerdijk2003}. 

The theory of Lie groupoids comes together with its infinitesimal counterpart: Lie algebroids. Being a generalization of the notion of Lie algebras, a \textbf{Lie algebroid} is a smooth vector bundle $A \to M$ together with a vector bundle map $\rho: A \to TM$, called the \textbf{anchor}, and a Lie bracket $[\Empty , \Empty]$ on its space of smooth sections $\Gamma(M,A)$, such that the Leibniz rule
\begin{equation*}
[a,fb]=f[a,b]+\big((\rho \circ a) \cdot f\big)b
\end{equation*}
holds for all sections $a, b \in \Gamma(M,A)$ and smooth real-valued functions $f \in C^{\infty}(M)$.

\subsection{Differentiation of Lie groupoids}

Lie groupoids can be differentiated to their Lie algebroids. To see how their differentiation can be generalized, we have to analyze the procedure from a categorical point of view. Let $G$ be a Lie groupoid. The assumption that the source map $s$ is a submersion is sufficient for the existence of the domain of the groupoid multiplication
\begin{equation*}
m: G_1 \times_{G_0}^{s,t} G_1 \longrightarrow G_1
\,,
\end{equation*}
as well as for the existence of the iterated $k$-fold pullback 
\begin{equation*}
%\label{eq:nervepullback}
\underbrace{G_1 \times_{G_0} \ldots \times_{G_0} G_1}_{\text{$k$ factors}}
\end{equation*}
for all $k \geq 1$. While the assumption that the source is a submersion is sufficient for the existence of these pullbacks, it is not necessary.

In the first step of the differentiation of $G$, we take the derivative of the multiplication by applying the tangent functor,
\begin{equation*}
Tm: T( G_1 \times_{G_0} G_1) \longrightarrow TG_1
\,.
\end{equation*}
To obtain the multiplication of the tangent groupoid, we need an isomorphism
\begin{equation*}
  \nu: T( G_1 \times_{G_0} G_1)
  \xrightarrow{~\cong~}
  TG_1 \times_{TG_0} TG_1
  \,,
\end{equation*}
that is, we need the pullback of the codomain to exist and the natural map $\nu$ to have an inverse. This is the step where the assumption that $s$ (and therefore $t$) is a submersion is needed. The multiplication of the tangent groupoid is then defined by the composition
\begin{equation*}
m_{TG}: TG_1 \times_{TG_0} TG_1 \xrightarrow[\cong]{~\nu^{-1}~} T( G_1 \times_{G_0} G_1) \xrightarrow{~Tm~} TG_1 \,.
\end{equation*}

In the second step, we restrict the tangent multiplication to a right $G$-action on the $s$-vertical tangent bundle,
\begin{equation*}
  VG_1 := TG_1 \times_{TG_0}^{Ts, 0_{G_0}} G_0
  \,,
\end{equation*}
where $0_{G_0}: G_0 \to TG_0$ is the zero section. Since $s$ is a submersion, so is $Ts$, which ensures the existence of the pullback. The right $G$-action
\begin{equation*}
  VG_1 \times_{G_0}^{\pr_2, t} G_1
  \longrightarrow
  VG_1
\end{equation*}
is the restriction of $m_{TG}$ to the $s$-vertical tangent bundle. Since $t$ is a submersion, the pullback of the domain exists.

In the third step, we restrict the vertical tangent bundle $VG_1 \to G_1$ to the identity bisection of the groupoid,
\begin{equation*}
  A := G_0 \times_{G_1} VG_1
  \,.
\end{equation*}
Since the tangent bundle is a vector bundle, so is the vertical tangent bundle. This implies that $VG_1 \to G_1$ is a submersion, so that the pullback always exists. 

The Lie algebroid is constructed as the vector bundle $A \to G_0$ with anchor the restriction of $Tt$ to $A$. The sections of $A$ can be shown to be in bijection with the right-invariant vector fields on $G_1$, that is, the right $G$-equivariant sections of the vertical tangent bundle. The Lie bracket of the sections of $A$ is the Lie bracket of vector fields.

The upshot is the following: For differentiation of a Lie groupoid to its Lie algebroid, a number of pullbacks must exist and the tangent functor must commute with the nerve of the Lie groupoid. In the category of smooth manifolds, this is the case if the source map is a submersion.

\subsection{The Lie bracket of vector fields}
\label{sec:IntroBracket}

Any category $\catC$ in which we want to differentiate a groupoid must have a tangent functor $T$ and some additional structure that allows us to define a Lie bracket of vector fields. This construction is particularly difficult to generalize, because the map $X \mapsto \Gamma(X, TX)$ that sends a manifold to its space of vector fields is not even a functor.

Fortunately, this issue has been resolved by Rosick\'{y} in his seminal paper \cite{Rosicky:1984}. He introduces the concept of abstract tangent functor as an endofunctor $T: \calC \to \calC$ together with the natural transformations of the bundle projection $\pi: T \to \Id$, zero section $0: \Id \to T$, fiberwise addition of the tangent bundle $+: T \times_\Id T \to T$, the exchange of the order of differentiation $\tau: T^2 \to T^2$, and the vertical lift $\lambda: T \to T^2$, satisfying a number of natural axioms (Definition~\ref{def:TangentStructure}). A category with an abstract tangent functor is called a \textbf{tangent category} by Cockett and Cruttwell \cite{CockettCruttwell:2014} who have fully developed the concept into an active field of research (see Section~\ref{sec:IntroPrevious}).

Rosick\'y's axioms capture the minimal categorical structure  needed to define the Lie bracket of vector fields on an object of $\catC$. In order to understand the motivation and ingenuity of Rosick\'{y}'s definition, we consider the coordinate formula for the Lie bracket of two vector fields
\begin{equation}
\label{eq:LieBrackLoc}
\begin{split}
  [v,w]
  &=
  \Bigl[
  v^i \frac{\partial}{\partial x^i},
  w^j \frac{\partial}{\partial x^j}
  \Bigr]
  = \Bigl( v^j \frac{\partial w^i}{\partial x^j}
  - w^j \frac{\partial v^i}{\partial x^j} 
    \Bigr) \frac{\partial}{\partial x^i}
  \\
  &=
  [v,w]^i \frac{\partial}{\partial x^i}
  \,.    
\end{split}
\end{equation}
If we include the basepoint, the vector field $w: X \to TX$ maps $x=(x^i)$ to $\big(x^i, w^i(x)\big)$. Its tangent map $Tw: TX \to T^2 X$ is given in coordinates by
\begin{equation*}
\begin{split}
  (Tw)(x^i, u^i) 
  &= \Bigl(x^i, w^i, 
  \frac{\partial x^i}{\partial x^j} u^j, 
  \frac{\partial w^i}{\partial x^j} u^j
    \Bigr)
  \\
  &= \Bigl(x^i, w^i, u^i, u^j
  \frac{\partial w^i}{\partial x^j}
    \Bigr)
  \,.
\end{split}
\end{equation*}
We see that the two terms of the Lie bracket are part of
\begin{equation*}
\begin{aligned}
Tw \circ v 
&= \Bigl(x^i, w^i, v^i, v^j \frac{\partial w^i}{\partial x^j} \Bigr)
\\
Tv \circ w 
&= \Bigl(x^i, v^i, w^i, w^j \frac{\partial v^i}{\partial x^j} \Bigr) \,.
\end{aligned}
\end{equation*}
We cannot subtract the two terms yet, since the basepoints $(x^i, w^i)$ and $(x^i, v^i)$ are different. We first have to apply the symmetric structure
\begin{equation*}
\begin{aligned}
  \tau_X: T^2 X 
  &\longrightarrow T^2 X
  \\
  (x^i, u^i_1, u^i_2, u^i_{12})
  &\longmapsto
  (x^i, u^i_2, u^i_1, u^i_{12})
\end{aligned}
\end{equation*}
to the second term, which yields
\begin{equation*}
\tau_X \circ Tv \circ w = \Bigl(x^i, w^i, v^i, w^j \frac{\partial v^i}{\partial x^j} \Bigr) 
\,.
\end{equation*}
Now we can apply the fiberwise subtraction $-_{TX}:  T^2 X \times_{TX} T^2 X \to T^2 X$ and obtain
\begin{equation*}
\begin{split}
  \delta(v,w)
  :={}&
  -_{TX} \circ \, (Tw \circ v , \tau_X \circ Tv \circ w \bigr)
\\
  ={}& 
  \Bigl(x^i, \, w^i, \, 0, \, v^j \frac{\partial w^i}{\partial x^j} - w^j \frac{\partial v^i}{\partial x^j} \Bigr)
\,.    
\end{split}
\end{equation*}
The map $\delta(v,w): X \to T^2 X$ lies in the image of the vertical lift map
\begin{equation*}
\begin{aligned}
  \lambda_{2,X}: TX \times_X TX 
  &\longrightarrow T^2 X
  \\
  \bigl( (x^i, w^i), (x^i, u^i) \bigr)
  &\longmapsto
  (x^i, w^i, 0, u^i)
  \,.
\end{aligned}
\end{equation*}
In other words, it takes values in the kernel of $T\pi_X$. By projecting to the second factor we obtain the Lie bracket expression in~\eqref{eq:LieBrackLoc}.
We conclude that 
\begin{equation}
\label{eq:BracketCatDef}
\begin{split}
  \delta(v,w) 
  &= \lambda_{2,X} 
  \bigl( (x^i, w^i), (x^i, [v,w]^i) \bigr)
  \\
  &=
  \lambda_{2,X} \circ \big( w, [v,w] \big)
  \,.    
\end{split}
\end{equation}
The upshot is that~\eqref{eq:BracketCatDef} provides a construction of the Lie bracket in terms of the tangent functor and the natural transformations of a tangent category. This will be elaborated in Section~\ref{sec:LieBracket}.

\subsection{Bundles}

In $\Mfld$, a \emph{fiber} bundle is a smooth surjective map $p:A \to X$ which is locally the projection of a product manifold to one of its factors. This concept of bundles with local trivializations requires a generalization of open covers, which is provided by the notion of Grothendieck pretopology. However, already for our guiding example of diffeological spaces requiring local trivializations is too strong. Instead, we start with the usual categorical notion of bundle as a morphism without further conditions, in other words, an object in the overcategory $\calC \Comma X$. Then we add, step by step, the algebraic structure we need, using the following unified terminology:

\begin{Terminology}
\label{term:WibbleBundle}
Let ``Wibble'' be an algebraic theory modelled by a small category $\calW$. Let $X$ be an object in a category $\calC$ such that the overcategory $\calC \Comma X$ has all finite products (i.e.~pullbacks over $X$). An object in $\calC\Comma X$ will be called a \textbf{bundle over $X$}. A Wibble object in $\calC\Comma X$, that is, a product preserving functor $\calW \to \calC\Comma X$, will be called a \textbf{bundle of Wibbles over $X$}.
\end{Terminology}

In this paper, ``Wibble'' will be one of: group, abelian group, and $R$-module (for $R \in \catC$ a ring object). Let $W \to X$ be a bundle of Wibbles. If $\catC$ has finite products, then the pullback $W_x := * \times_X W$ over a point $x: * \to X$ is a Wibble object in $\calC$. Here, $*$ is the terminal object of $\catC$. In other words, every fiber of a bundle of Wibbles is a Wibble, which justifies the terminology. We emphasize once more, that the notion of bundle of Wibbles does not make any assumptions on local trivializations, whatsoever. 

In Sections~\ref{sec:BunAbelGp} and \ref{sec:BundleRMod}, we will elaborate Terminology~\ref{term:WibbleBundle} for the case of bundles of abelian groups and bundles of $R$-modules, and spell out some of their properties. In the axioms of a tangent structure, the tangent bundle $\pi_X: TX \to X$ is required to be a bundle of abelian groups for any object $X \in \catC$. 
%More generally, we will extend the notion of tangent categories to cartesian tangent categories with scalar $R$-multiplication (Definition~\ref{def:RScalar}). In this case, the tangent bundle will have the structure of a bundle of $R$-modules.

\subsection{Submersions and differentiability}

As we mentioned above, submersions are a class of morphisms in $\Mfld$ which play an important role in the definition of a Lie groupoid and in the differentiation to its Lie algebroid. Let us take a moment to delve deeper into their properties. Recall that a smooth map $f: M \to N$ of manifolds is called a \textbf{submersion} if:
\begin{itemize}
\item[(i)]
The tangent map $Tf|_x : T_x M \to T_{f(x)} N$ is surjective for every point $x \in M$.
\end{itemize}
Submersions have a number of good properties, such as:
\begin{itemize}
\item[(ii)] Every point $x \in M$ lies in the image of a smooth local section of $f$.

\item[(iii)] For every smooth map $M' \to N$, the pullback $M \times_N M'$ exists.

\item[(iv)] Every pullback via $f$ commutes with the tangent functor, that is, the natural map $T(M \times_N M') \to TM \times_{TN} TM'$ is an isomorphism.

\item[(v)] The tangent map $Tf: TM \to TN$ is a submersion.
\end{itemize}

%\item[(vi)] The fiber $f^{-1}(y)$ of every point $y \in Y$ is an embedded submanifold.
Properties (i), (ii), and (iv) are equivalent. It is tempting to state the conditions for a groupoid object in a tangent category to be \emph{differentiable} by requiring the source map to have the properties of a surjective submersion. However, Properties~(ii)--(v) make use of the fact that a submersion is locally the projection to one of the factors of a product manifold. This is proved with the implicit function theorem, a genuinely analytic method that cannot be easily generalized to convenient settings. In a general tangent category none of the Properties (i)--(iv) implies any other of the properties. So which properties should be the basis of our generalization? 

%The Property~(ii), that a morphism has local sections, is a popular choice \TodoChristian{Add refs}. 
For the existence of the pullbacks of the nerve of a groupoid we need Property~(iii), for the first step of the differentiation procedure we need Property~(iv), for the second step (v) and (iii). So we would need morphisms that have the Properties~(iii), (iv), and (v). We are not aware of any interesting tangent category other than the category of smooth manifolds, in which there is a known class of such morphisms.

Fortunately, the properties of a submersion are unnecessarily strong. We do not need the pullbacks via submersions to exist universally over \emph{all} other morphisms, but only the pullbacks that appear in our construction. Neither do we need the tangent functor to preserve all pullbacks, but only those of the nerve of the groupoid. For all these reasons, we will not attempt to define a notion of generalized submersions. Instead, we promote in the Definition~\ref{def:DiffGroupoid} of differentiable groupoid objects only those properties to axioms that are actually needed in the differentiation process.

%For all these reasons, in the Definition~\ref{def:DiffGroupoid} of differentiable groupoid objects in tangent categories we will identify and explicitly require the existence of the pullbacks and their properties needed in the differentiation process. 

\subsection{Main results}

The main result of this paper is to give satisfactory answers to Questions~\ref{questions:Main}. The first part of the answers is to identify the categorical structures needed in the differentiation process. As already explained, the ambient category $\calC$ has to be equipped with an abstract tangent structure (Definition~\ref{def:TangentStructure}), so that there is a Lie bracket of vector fields. We also have to assume that the tangent functor preserves finite products. Tangent categories with this property are called \textbf{cartesian} (Definition~\ref{def:CartTan}). Finally, we need a generalization of the scalar multiplication of tangent vectors by real numbers. For this, the tangent bundles $TX \to X$ have to be equipped with the fiberwise structure of an $R$-module for a ring object $R$ in $\calC$, satisfying some compatibility with the tangent structure. Such a module structure is called a \textbf{scalar $R$-multiplication} (Definition~\ref{def:RScalar}). All these categorical structures have been worked out in the companion paper \cite{AintablianBlohmann:cartan}. 

The second part of the answers to Questions~\ref{questions:Main} is to identify the properties of a groupoid object in $\catC$ needed for its differentiation. This leads to the first main concept of this paper:

\begin{Definition*}[Definition~\ref{def:DiffGroupoid}]
A groupoid object $G$ in a tangent category $\catC$ will be called \textbf{differentiable} if the pullbacks in the diagrams
\begin{equation}
\begin{tikzcd}[column sep=small]
T^n G_1 \times_{T^n G_0} G_k
\ar[d] \ar[r]
\arrow[dr, phantom, "\lrcorner", very near start] 
&
G_k
\ar[d, "\Zero{n}_{G_k}"]
\\
T^n G_1 \times_{T^n G_0} T^n G_k
\ar[d] \ar[r]
\arrow[dr, phantom, "\lrcorner", very near start] 
&
T^n G_k
\ar[d, "T^n t_k"]
\\
T^n G_1
\ar[r, "T^n s"']
&
T^n G_0
\end{tikzcd}
\quad
\begin{tikzcd}[column sep=small]
T_m G_1 \times_{T_m G_0} G_k
\ar[d] \ar[r]
\arrow[dr, phantom, "\lrcorner", very near start] 
&
G_k
\ar[d, "\ZeroDiag{m}{G_0} \circ \, t_k"]
\\
T_m G_1
\ar[r, "T_m s"']
&
T_m G_0    
\end{tikzcd}
\end{equation}
\begin{equation}
\begin{tikzcd}[column sep=small]
G_0 \times_{G_1} TG_1 \times_{TG_0} G_0
\ar[d] \ar[r]
\arrow[dr, phantom, "\lrcorner", very near start] 
&
TG_1 \times_{TG_0} G_0
\ar[d, "\pi_{G_1} \circ \, \pr_1"]
\\
G_0
\ar[r, "1"']
&
G_1    
\end{tikzcd}
\end{equation}
exist, and if the natural morphism
\begin{equation}
  T^n(G_1 \times_{G_0}^{s,t_k} G_k) 
  \longrightarrow T^n G_1 \times_{T^n G_0}^{T^ns,T^nt_k} T^n G_k
\end{equation}
is an isomorphism for all $n \geq 1$, $m \geq 2$ and $k \geq 0$.
\end{Definition*}

Here, $\Zero{n}_{G_0} = (T^{n-1}0 \, \circ \, \ldots \, \circ \, T0 \, \circ \, 0)_{G_0}$ is the iterated zero section, $\ZeroDiag{m}{G_0} := (0_{G_0}, \ldots, 0_{G_0})$ the diagonal zero section, and $t_k: G_k \to G_0$ the morphism to the target of the rightmost arrow of the nerve of the groupoid.

The notion of differentiability generalizes to any bundle $E \to G_0$ with a right $G$-action (Definition~\ref{def:DiffBunG}). On differentiable groupoid bundles there is a natural notion of invariant vector fields (Definition~\ref{def:InvVecField}). We show in Theorem~\ref{thm:InvariantBracket} that the set $\calX(E)^G$ of invariant vector fields on $E$ is a Lie subalgebra of $\mathfrak{X}(E)=\Gamma(E,TE)$. The second main concept of this paper is:

\begin{Definition*}[Definition~\ref{def:AbstractLieAlgd}]
Let $\calC$ be a cartesian tangent category with scalar $R$-multiplication. An \textbf{abstract Lie algebroid} in $\calC$ consists of a bundle of $R$-modules $A \to X$, a morphism $\rho: A \to TX$ of bundles of $R$-modules, called the \textbf{anchor}, and a Lie bracket on the abelian group $\Gamma(X,A)$, such that
\begin{align*}
  [a, fb] &= f[a,b] + \bigl((\rho \circ a) \cdot f \bigr) b
  \\
%\shortintertext{and}
  \rho \circ [a,b] &= [\rho \circ a , \rho \circ b]
\end{align*}
for all sections $a$, $b$ of $A$ and all morphisms $f: X \to R$ in $\catC$.
\end{Definition*}

The main generalization is that $A \to G_0$ is no longer required to be a vector bundle, but only a bundle of $R$-modules (Terminology~\ref{term:WibbleBundle}), so that the sheaf of sections of $A$ is generally not a locally free $\calC(G_0, R)$-module. Therefore, the Leibniz rule does not imply that the composition with the anchor is a homomorphism of Lie algebras, which has to be required separately. Our main results can be summarized as follows:

\begin{Theorem*}[Theorem~\ref{thm:SecAInvVec}, Theorem~\ref{thm:LieAlgdOfGroupoid}]
Let $G$ be a differentiable groupoid object in a cartesian tangent category $\calC$ with scalar $R$-multiplication. Then there is an abstract Lie algebroid $A \to G_0$ with the Lie bracket of invariant vector fields on $\Gamma(G_0, A) \cong \calX(G_1)^G$.
\end{Theorem*}

\subsection{Applications}

Broad categorical generalizations of geometric concepts like Lie groupoids and Lie algebroids risk being solutions in search of a problem. That is, while the generalization of a concept may be of interest in its own right, its usefulness in the originating area, differential geometry in our case, may be limited. The concept of differentiable groupoid objects and their abstract Lie algebroids, however, was motivated by and provides answers to existing research questions.

In \cite{BlohmannFernandesWeinstein:2013} the symmetry structure of the initial value problem of general relativity was studied.  
%\cite{BlohmannWeinstein:2018,BlohmannSchiavinaWeinstein:2022} these refs are not pertinant
The main result was the construction of a diffeological groupoid whose diffeological Lie algebroid has the same bracket as the somewhat mysterious Poisson bracket of the Gau\ss-Codazzi constraints for Ricci flat metrics. Without a theory of diffeological groupoids available, the differentiation of the diffeological groupoid of \cite{BlohmannFernandesWeinstein:2013} was carried out in an ad hoc manner. Our results can be viewed as a rigorous justification of this construction. In this sense, this application of our main theorem existed before the actual theorem. The groupoid in the general relativity example generalizes to the diffeological groupoids that arise from the gauge and diffeomorphism symmetries of classical field theory.

The construction of the groupoid symmetry of general relativity from \cite{BlohmannFernandesWeinstein:2013} is a particular case of a more general construction: reductions of action Lie groupoids by (not necessarily normal) subgroups. In \cite[Section~8.5]{BlohmannWeinstein:2024}, the Lie algebroid of reduced groupoids was computed in the setting of smooth finite-dimensional manifolds. With our results, this reduction can now be carried out for diffeomorphism groups that act on sections of natural bundles.

Diffeological groupoids arise in a number of other areas, such as diffeological integration of Lie algebroids \cite{Villatoro}, the holonomy groupoid of a singular foliation \cite{AndroulidakisSkandalis:2009}, and more generally, of a singular subalgebroid \cite{Zambon:2022}. In \cite{AndroulidakisZambon:2023}, Androulidakis and Zambon explain a differentiation procedure from holonomy-like diffeological groupoids to singular subalgebroids, where the original goal of the authors is to establish an integration method for singular subalgebroids via diffeological groupoids.

One of the original motivations of this project was the application of diffeological groupoids to geometric deformation theory. The idea is that many geometric structures, such as Riemannian metrics or complex structures, are equipped with a natural diffeology of smooth homotopies, so that their moduli space of structures modulo isomorphism are presented by diffeological groupoids. (Diffeological stacks are a $(2,1)$-category, so that higher categorical considerations come into play.) In this framework, deformations can be conceptualized by smooth paths in a diffeological stack, higher deformations by paths of paths, etc. This suggests that the natural object to consider is the path $\infty$-groupoid of the diffeological moduli stack. We posit that the infinitesimal object of these groupoids is closely related to the differential graded Lie algebras or the $L_\infty$-algebras describing the deformation theory of these moduli.

An outlook on the applications and classes of examples of our construction is presented in Section~\ref{sec:Examples}.

\subsection{Relation to previous work}
\label{sec:IntroPrevious}

After their introduction by Rosick\'y in 1984 \cite{Rosicky:1984}, tangent categories had been dormant for almost 30 years. It was Cockett and Cruttwell with their collaborators and students who fully realized their potential and developed the subject into an active area of research. In fact, we first learned about tangent categories from \cite{CockettCruttwell:2014} and are indebted to the work of their school. 

The present paper is guided by specific geometric applications (Section~\ref{sec:Examples}), in particular to diffeological groupoids. These require a number of assumptions that are both stronger than weaker than those of Cockett and Cruttwell, which means that there is generally no simple logical dependence between our results and theirs.

In Rosick\'y's original definition~\cite{Rosicky:1984}, the fibers of the tangent bundles are assumed to be abelian groups, which is needed for the definition of the Lie bracket of vector fields. In \cite{CockettCruttwell:2014}, this was relaxed to commutative monoids, and Rosick\'y's concept renamed to tangent structures ``with negatives''. Since the Lie bracket is an indispensable structure for our results and geometric applications, all our tangent categories are assumed to have negatives.

We assume that our tangent functors preserve finite products, that is, the natural morphism $\chi_{X,Y}: T(X \times Y) \to TX \times TY$ is an isomorphism. In the companion paper~\cite{AintablianBlohmann:cartan}, we will show for tangent categories with negatives that this condition implies that $\chi$ is a strong morphism of tangent categories. This means that the tangent category is cartesian in the sense of \cite[Definition~2.8]{CockettCruttwell:2014}.

The genuinely new structure we require of our tangent categories is a scalar multiplication by a ring object $R$ with good infinitesimal properties (Definition~\ref{def:RScalar}). The scalar multiplication equips the tangent bundles with a fiberwise structure of an $R$-module. Moreover, we require the tangent fiber of $R$ at zero to be isomorphic to $R$. This can be shown to imply that we have a trivialization $TR \cong R \times R$ as bundle of $R$-modules \cite{AintablianBlohmann:cartan}. It follows by recent work of Michael Ching \cite{Ching:2024}, that $R$ is a differential object in the sense of \cite[Definition~4.8]{CockettCruttwell:2014} and, by~Definition~\ref{def:TangString}, a differential ring. However $R$ is not an differential exponential rig in the sense of \cite[Definition~5.24]{CockettCruttwellLemay:2021} in any of our application. Already for infinite dimensional Lie groups, it follows by the no-go theorem of Omori that there is no exponential map \cite{Omori:1981}.

In addition, we need to require three compatibility conditions of the scalar multiplication with the tangent functor (Definition~\ref{def:RScalar}), so that vector fields act as derivations on $R$-valued functions (Proposition~\ref{prop:VecActFunc}) and the Lie bracket satisfies the Leibniz rule (Proposition~\ref{prop:LeibnizRule}). While Definition~\ref{def:RScalar} implies that the scalar multiplication is a morphism of differential bundles in the sense of in the sense of \cite[Definition~2.3]{CockettCruttwell:2017}, the converse is not true. The only condition for the tangent map of a morphism of differential bundles is its compatibility with the vertical lifts. This is not sufficient to deduce the second and third commutative diagram of Definition~\ref{def:RScalar}. In fact, Definition~\ref{def:RScalar} involves the symmetric structure on $T^2$, which is not an ingredient of differential bundles.

In~\cite{BurkeMacAdam:2019}, Burke and MacAdam introduce the notion of ``involution algebroids'' as generalizations of Lie algebroids in tangent categories. The underlying bundles of the involution algebroids are assumed to be differential bundles. In \cite[Theorem~1]{MacAdam:2021}, it was shown that a differential bundle in smooth manifolds is a vector bundle. By contrast, we assume that the bundle of an abstract Lie algebroid (Definition~\ref{def:AbstractLieAlgd}) is a bundle of $R$-modules. It is then a rather easy observation that, in the tangent category of smooth manifolds, a bundle of $\bbR$-modules is a vector bundle (Example~\ref{ex:VecBund}). This implies that an abstract Lie algebroid in smooth manifolds is a Lie algebroid in the usual sense (Proposition~\ref{prop:LieAlgdDiff}). We conclude that the fairly strong condition that the bundles be differential is not necessary to recover the usual Lie theory of Lie groupoids and their Lie algebroids.

In the case of diffeological spaces, our main application, bundles will be generally not differential. One of the reasons is that for a diffeological vector spaces, $T_0 A \cong A$ is a rather special property. For example, for locally convex topological vector spaces this is true if and only they are convenient in the sense of Kriegl and Michor \cite{KrieglMichor:1997}. Moreover, the groupoids we consider are not manifolds in a diffeological sense. In fact, even the diffeomorphism group of a manifold is not a diffeological manifold modelled on its tangent space.

%In~\cite[Definition~9]{MacAdam:2021}, MacAdam introduces is very strong. It follows from \cite[Lemma~1]{MacAdam:2021} that in manifolds $f: A \to B$ is a submersion if and only if $TA \to A \times_B TB$ has a section, that is, if $A \to B$ has a connection. This is not the case for all submersions, since we need a partition of unity to piece together local sections to a global section. This is to say, for \cite[Proposition~1]{MacAdam:2021} to be true it has to be assumed that smooth manifolds are paracompact. (This should have been stated explicitly, since there are interesting smooth manifolds that are not paracompact, e.g.~in Milnor's 1969 MIT lecture notes on foliations.) For infinite-dimensional manifolds or diffeological manifolds, let alone diffeological spaces, this will generally fail. The upshot is: In the zoo of generalized definitions of submersions, the one of \cite{MacAdam:2021} may be one of the strongest and, therefore, least applicable to our situation. A comment was added in the new Section~\ref{sec:IntroPrevious}.

\subsection{Outline of the paper}

In Section~\ref{sec:AbstractTangent}, we review Rosick{\'y}'s abstract tangent structures \cite{Rosicky:1984}. We study bundles whose fibers have the structure of an abelian group or an $R$-module, where $R$ is a ring object in the ambient category. Extending the ideas in \cite{blohmann2024}, we introduce cartesian tangent structures with scalar $R$-multiplication. A fundamental part of this section is devoted to the construction of the Lie bracket of vector fields in tangent categories, where we also state the Leibniz rule. Lastly, we describe the namesake example of tangent categories: the category of Euclidean spaces. 

In Section~\ref{sec:DiffGrpds}, we introduce one of the main concepts in this paper: differentiable groupoid objects. We first spell out what groupoid objects are, and then identify the conditions needed for their differentiability. More generally, in Section~\ref{sec:Gbun} we introduce the category of differentiable groupoid bundles and equivariant bundle morphisms. This paves the way for the definition of higher vertical tangent bundles and fiber products of vertical tangent bundles.

The goal of Section~\ref{sec:HigherVertTanFun} is to show that the assignments to a differentiable $G$-bundle its higher vertical tangent bundle and the fiber product of its vertical tangent bundle are functorial. Using an elegant technical lemma, we show that the tangent structure of the ambient category and the module structure admit a source-vertical restriction. In Section~\ref{sec:LieBracketInv} we introduce invariant vector fields on differentiable groupoid bundles. One of the main results is that the invariant vector fields are closed under the Lie bracket of vector fields. 

Section~\ref{sec:AbsractLieAlgd0} is the core of the paper. We establish a differentiation procedure of groupoid objects to their infinitesimal counterparts in a cartesian tangent category with scalar $R$-multiplication. On the infinitesimal level, we define abstract Lie algebroids in tangent categories. The main theorem states that our differentiation procedure applied to a differentiable groupoid object yields an abstract Lie algebroid, where the Lie bracket is obtained by the bracket of invariant vector fields.

The paper concludes in Section~\ref{sec:Examples} with classes of examples and possible applications. Here we also discuss the relation to previous work.

Lastly, in Appendix~\ref{app:Cat} we provide, for the convenience of the reader, a brief exposition of useful categorical machinery needed for the paper; and we fix some notation.

\subsection{Conventions and notation}

Categories will be denoted by the curly letters $\catC$ and $\catD$. We will denote the functor category with objects functors $F:\catC \to \catD$ and morphisms natural transformations by the exponential notation $\catD^{\catC}$. Given objects $A,B$ in a category $\catC$, the morphisms from $A$ to $B$ will be denoted by $\catC(A,B)$. The categories we consider in this paper are all locally small, that is, $\calC(A,B)$ is a set. We will explicitly state when a category is also small (for example groupoids), that is, the objects form a set. A \textbf{point} of an object $X$ of a category $\catC$ with terminal object $*$ is a morphism $x:* \to X$. The reader may refer to Appendix~\ref{app:Cat} for the parenthesis notation of morphisms and their fiber products, which are extensively used in this paper.

\subsection*{Acknowledgements}

For comments and fruitful discussions, we would like to thank David Aretz, Michael Ching, Geoffrey Cruttwell, Madeleine Jotz, Benjamin MacAdam, David Mi\-ya\-mo\-to, Leonid Ryvkin, and Tashi Walde. L.A.~was funded by the Hausdorff Center for Mathematics in Bonn.
\section{Abstract tangent structures}
\label{sec:AbstractTangent}

The goal of this section is to provide the reader with the necessary background on tangent structures needed for our purposes. In Section~\ref{sec:BunWib}, we give a categorical approach to bundles, whose fibers have the structure of an algebraic theory, yet there is no assumption on local triviality. 
%Examples include bundles of abelian groups and bundles of modules, which will be discussed in detail. 
We then proceed with symmetric structures on endofunctors in Section~\ref{sec:Symm}. The aim of Sections~\ref{sec:RosickysAxioms} and~\ref{sec:CatTanScal} is to discuss in detail abstract tangent functors in the sense of Rosick\'y and to extend the framework to cartesian tangent structures and scalar multiplications. 

Furthermore, in Section~\ref{sec:LieBracket} we explain the main construction of the Lie bracket of vector fields in a tangent category, using the ideas of Rosick\'y. We then state the Leibniz rule and several naturality results for the Lie bracket in Section~\ref{sec:Leibrule}. Lastly, the goal of Section~\ref{sec:TangentEuclidean} is to provide an intuition of tangent structures by considering the namesake example of a tangent category: the category of Euclidean spaces.

\subsection{Bundles with algebraic structure}
\label{sec:BunWib}

In this section, we elaborate Terminology~\ref{term:WibbleBundle} for the case of abelian groups and $R$-modules.
Let $\catC$ be a category with a terminal object $*$. Assume that the overcategory $\catC \Comma X$ has all finite products for all $X \in \catC$. In particular, this implies that $\catC \Comma * \cong \catC$ has all finite products. The proofs of the statements in this section can be found in \cite{Aintablian:2024} and \cite{AintablianBlohmann:cartan}.

\subsubsection{Bundles of abelian groups}
\label{sec:BunAbelGp}

Let $p: A \to X$ be a bundle of abelian groups, that is, an abelian group object in the overcategory $\catC \downarrow X$. (This is sometimes called a ``Beck module''.) The terminal object in $\calC\Comma X$ is given by $\id_X: X \to X$. Spelled out, the group structure consists of the morphisms \begin{equation*}
\begin{tikzcd}[column sep={tiny}]
A \times_X A \ar[rr, "+"] \ar[dr, "p \, \circ \, \pr_1 = p \, \circ \, \pr_2"'] & & A \ar[dl, "p"] \\
& X &
\end{tikzcd}
\qquad
\begin{tikzcd}[column sep={tiny}]
X \ar[rr, "0"] \ar[dr, "\id_X"'] & & A \ar[dl, "p"] \\
& X &
\end{tikzcd}
\qquad
\begin{tikzcd}[column sep={tiny}]
A \ar[rr, "\iota"] \ar[dr, "p"'] & & A \ar[dl, "p"] \\
& X &
\end{tikzcd}
\end{equation*}
of the addition, the zero, and the inverse, satisfying the usual conditions of associativity, unitality, invertibility, and commutativity of an abelian group. The subtraction is given by the commutative triangle 
\begin{equation}
\label{diag:difference}
\begin{tikzcd}
A \times_X A \ar[rr, "-"] \ar[dr, "{\id_A \times_X\, \iota}"'] &[-2em] &[-2em] A \\
& A \times_X A \ar[ur, "+"'] &
\end{tikzcd}
\end{equation}
Even though our notion of bundles does not assume any local triviality, the bundles behave in the usual way under base changes.

\begin{Definition}
Let $p:A \to X$ and $p': A' \to X'$ be bundles of abelian groups with additions $+$ and $+'$ respectively. A \textbf{morphism of bundles} is a commutative diagram
\begin{equation}
\label{eq:MorphBund1}
\begin{tikzcd}
A \ar[r, "\phi"] \ar[d, "p"'] & A' \ar[d, "p'"]
\\
X \ar[r, "f"'] & X' 
\end{tikzcd}    
\end{equation}
in the category $\catC$. It is called a \textbf{morphism of bundles of abelian groups} if furthermore the diagram
\begin{equation}
\label{eq:MorphBund3}
\begin{tikzcd}
A \times_X A \ar[r, "\phi \times_f \phi"] 
\ar[d, "+"'] & A' \times_{X'} A' \ar[d, "+'"]
\\
A \ar[r, "\phi"'] & A'
\end{tikzcd}    
\end{equation}
commutes.     
\end{Definition}

As for ordinary groups, it is implied that the zeros are intertwined, that is, $\phi\, \circ\, 0 = 0' \circ f$, where $0': X' \to A'$ is the zero of $p': A' \to X'$. Composing this equation with $p'$ on the left, we obtain $p' \circ \phi \circ 0 = f$, which shows that $f$ is uniquely determined by $\phi$. Therefore, we can denote a morphism~\eqref{eq:MorphBund1} of bundles of abelian groups by $\phi$.

\begin{Proposition}
\label{prop:PullbackBundle}
Let $A \to X$ be a bundle of abelian groups and $Y \to X$ a morphism in $\catC$. Then the pullback $\pr_1: Y \times_X A \to Y$ is a bundle of abelian groups and $\pr_2: Y \times_X A \to A$ is a morphism of bundles of abelian groups.
\end{Proposition}

\begin{Example}
Let $x:* \to X$ be a point in $X$. Then, the fiber $A_x = * \times_X A$ is an abelian group object in $\catC \cong \catC \Comma *$.    
\end{Example}

\begin{Definition}
\label{def:ker}
Let $A \to X$ and $A' \to X'$ be bundles of abelian groups. The \textbf{kernel} of a morphism $\phi:A \to A'$ of bundles of abelian groups, if it exists, is the pullback 
\begin{equation*}
%\label{eq:KernelDefDiag}
\begin{tikzcd}
\mathllap{\ker\phi := {} } 
X' \times_{A'} %^{0'\circ f, \phi} 
A
\ar[d, "p_{\ker \phi}"'] \ar[r, "i_{\ker \phi}"] 
%\arrow[dr, phantom, "\lrcorner", very near start] 
&
A
\ar[d, "\phi"]
\\
X' \ar[r, "0'"'] & 
A' 
\end{tikzcd}
\end{equation*}
where $0'$ is the zero of $A' \to X'$.
\end{Definition}

\begin{Proposition}
\label{prop:KerOfMorph}
Let $\phi: A \to A'$ be a morphism of the bundles of abelian groups $A \to X$ and $A' \to X'$. Then: 
\begin{itemize}

\item[(i)] The composition $\ker\phi \xrightarrow{i_{\ker \phi}} A \longrightarrow X$ equips the kernel of $\phi$ with the structure of a bundle of abelian groups over $X$.

\item[(ii)] The morphism $i_{\ker \phi}: \ker\phi \to A$ %defined in~\eqref{eq:KernelDefDiag} 
is a regular monomorphism of bundles of abelian groups over $X$.

\end{itemize}
\end{Proposition}

\begin{Notation}
\label{not:SectionsofBundle}
The set of sections of a bundle $p: A \to X$ will be denoted by 
\begin{equation*}
\Gamma(X,A) := \{a:X \to A \ | \ p \circ a = \id_X \} \,.
\end{equation*}
\end{Notation}

The functor of sections
\begin{equation}
\label{eq:SectionFunctor}
\begin{aligned}
  \Gamma: \catC \Comma X 
  &\longrightarrow \Set
  \\
  (A \to X) 
  &\longmapsto \Gamma(X,A)    
\end{aligned}
\end{equation}
preserves finite products. As is the case for any functor that preserves products, it maps (abelian) group objects in $\catC \Comma X$ to (abelian) groups. In other words, the set of sections of a bundle of abelian groups has the structure of an abelian group. The sum of two sections $a,b:X \to A$ of a bundle of abelian groups $A \to X$ is given by
\begin{equation}
\label{eq:SectionFunctorPlus}
a+b := + \circ (a,b) \,.
\end{equation}  

Furthermore, \eqref{eq:SectionFunctor} takes morphisms of group objects in $\calC \Comma X$ to group homomorphisms. That is, if $\phi:A \to A'$ is a morphism of bundles of abelian groups over $X$, then the map
\begin{equation}
\label{eq:MapSections}
\begin{aligned}
\phi_*: \Gamma(X,A) &\longrightarrow \Gamma(X,A') \\ 
a &\longmapsto \phi \circ a
\end{aligned}
\end{equation}
is a group homomorphism. If $\phi$ is a monomorphism, then \eqref{eq:MapSections} is injective.

\subsubsection{Bundles of \texorpdfstring{$R$}{R}-modules}
\label{sec:BundleRMod}

Let $R$ be a commutative ring object in the category $\calC$ with addition $\Radd$, zero $\Rzero$, multiplication $\Rmult$, and unit $\Runit$. Let $X \in \catC$ be an object. Since the functor 
\begin{equation*}
\begin{aligned}
  \catC 
  &\longrightarrow
  \catC \Comma X
  \\
  C 
  &\longmapsto 
  (X \times C \xrightarrow{~\pr_1~} X)
\end{aligned}
\end{equation*}
preserves finite products, the ring structure on $R$ induces a ring structure on $\pr_1: X \times R \to X$ in $\catC \Comma X$. An $(X \times R \to X)$-module object in $\catC \Comma X$ will be called, for short, a \textbf{bundle of $R$-modules} over $X$ (Terminology~\ref{term:WibbleBundle}). Explicitly, it is a bundle of abelian groups $p:A \to X$ together with a morphism 
\begin{equation}
\label{diag:ModStr1}
\begin{tikzcd}[column sep={tiny}]
R \times A 
\ar[rr, "\kappa"] 
\ar[dr, "p\, \circ\, \pr_2"'] &&
A \ar[dl, "p"]
\\
& X &
\end{tikzcd}
\end{equation}
satisfying the usual conditions of a left linear action, which are given by a number of commutative diagrams (see Appendix~\ref{sec:GrpRngMod}).

\begin{Example}
\label{ex:VecBund}
Let $p: A \to M$ be a bundle of $\bbR$-modules in $\Mfld$. The zero section satisfies $p \circ 0 = \id_M$, which implies that $Tp \circ T0 = \id_{TM}$, so the tangent map of $p$ is surjective at $0_m$ at any point $m \in M$. In other words, $p$ is a submersion at $0_m$. By the standard form of submersions, we can choose a neighborhood of $0_m$ that is diffeomorphic to $U \times V$, for an coordinate neighborhood $U$ of $m$ and $V \subset \bbR^q$ open, such that the restriction of $p$ is diffeomorphic to the projection $\pr_1: U \times V \to U$. It follows that there are smooth local sections $e_1, \ldots, e_q \in \Gamma\bigl(U, p^{-1}(U) \bigr)$ that are linearly independent at every point $m \in V$. This shows that every $p$-fiber over $U$ is isomorphic to $\bbR^q$ and the smooth map
\begin{equation*}
\begin{aligned}
  U \times \bbR^q &\longrightarrow p^{-1}(U) = A\bigr|_U
  \\
  ( u, c_1, \ldots, c_q )
  &\longmapsto c_1 e_1(u) + \ldots + c_q e_q(u)
\end{aligned}
\end{equation*}
is a isomorphism of bundles of $\bbR$-modules, that is, a local trivialization. We conclude that $A \to M$ is a vector bundle in the usual sense. The upshot is that bundles of $\bbR$-modules in $\Mfld$ are the same as vector bundles.\footnote{We thank the participants of the Bonn-Cologne-G\"ottingen-W\"urzburg PhD retreat 2025 in Burbach for figuring this out with us over a beer.}
\end{Example}

\begin{Definition}
Let $p:A \to X$ and $p': A' \to X'$ be bundles of $R$-modules with module structures $\kappa$ and $\kappa'$ respectively. A \textbf{morphism of bundles of $R$-modules} is a morphism $\phi: A \to A'$
of bundles of abelian groups such that the diagram
\begin{equation}
\label{eq:MorphBundR2}
\begin{tikzcd}[column sep=large]
R \times A \ar[r, "\id_R \times \phi"] 
\ar[d, "\kappa"'] & R \times A' \ar[d, "\kappa'"]
\\
A \ar[r, "\phi"'] & A'
\end{tikzcd}    
\end{equation}
commutes.     
\end{Definition}

\begin{Proposition}
\label{prop:PullbackBundleR}
Let $A \to X$ be a bundle of $R$-modules and $Y \to X$ a morphism in $\catC$. Then, the pullback $\pr_1: Y \times_X A \to Y$ is a bundle of $R$-modules and $\pr_2: Y \times_X A \to A$ is a morphism of bundles of $R$-modules.
\end{Proposition}

\begin{Proposition}
\label{prop:KerOfMorphR}
Let $\phi: A \to A'$ be a morphism of the bundles of $R$-modules $A \to X$ and $A' \to X'$. Then:
\begin{itemize}

\item[(i)] The composition $\ker\phi \xrightarrow{i_{\ker \phi}} A \longrightarrow X$ equips the kernel of $\phi$ with the structure of a bundle of $R$-modules over $X$.

\item[(ii)] The morphism $i_{\ker \phi}: \ker\phi \to A$ %defined in~\eqref{eq:KernelDefDiag} 
is a regular monomorphism of bundles of $R$-modules over $X$.

\end{itemize}
\end{Proposition}

\begin{Remark}
\label{rmk:RingFunctions}
The ring object $\pr_1: X \times R \to X$ in $\calC \Comma X$ is mapped by the functor of sections \eqref{eq:SectionFunctor} to a ring, which is isomorphic by
\begin{equation}
\label{eq:SectionsTrivial}
\begin{aligned}
  (\pr_2)_* : 
  \Gamma(X, X \times R) &\xrightarrow{~\cong~}
  \calC(X, R) \\
  a &\longmapsto \pr_2 \circ a
\end{aligned}
\end{equation}
to the ring of $R$-valued morphisms with addition and multiplication
\begin{equation*}
  f + g := \Radd \circ (f,g)
  \,,\quad
  fg := \Rmult \circ (f,g)
  \,,
\end{equation*}
for all $f, g \in \calC(X,R)$, with zero $X \to * \stackrel{\Rzero}{\to} R$, and unit $X \to * \stackrel{\Runit}{\to} R$.
\end{Remark}

Now, let $A \to X$ be a bundle of $R$-modules with module structure $\kappa: R \times A \to A$. Using the isomorphism \eqref{eq:SectionsTrivial}, we see that $\Gamma(X,A)$ has a $\catC(X,R)$-module structure given by 
\begin{equation}
\label{eq:SectionsModuleStr}
  fa := \kappa \circ (f,a) \,,
\end{equation}
for all $f \in \calC(X,R)$ and $a \in \Gamma(X,A)$. Moreover, if $\phi:A \to A'$ is a morphism of bundles of $R$-modules over $X$, then the map \eqref{eq:MapSections} is a morphism of $\calC(X,R)$-modules.

\subsection{Symmetric structures}
\label{sec:Symm}

\subsubsection{Compositions of natural transformations}

We will follow \cite{Rosicky:1984} for the notation of the compositions of functors and natural transformations. The composition of functors $G: \calA \to \calB$ and $F:\calB \to \calC$ will be denoted by juxtaposition $FG: \calA \to \calC$. Therefore, the horizontal composition of natural transformations $\alpha: F \to F'$ and $\beta: G \to G'$ will also be denoted by juxtaposition $\alpha\beta: FG \to F'G'$. It is the natural transformation with components $(\alpha \beta)_A$ given by the diagonal of the following commutative diagram
\begin{equation*}
%\label{diag:godement}
\begin{tikzcd}[column sep=4em,row sep=3em]
FGA \arrow[r, "F(\beta_A)"] \arrow[d, "\alpha_{GA}"'] \arrow[rd, "(\alpha \beta)_A"] & FG'A \arrow[d, "\alpha_{G'A}"] \\
F'GA \arrow[r, "F'(\beta_A)"'] & F'G'A
\end{tikzcd} 
\end{equation*}
for all $A \in \catA$. We will use the notation $F$ for the identity natural transformation $\id_F:F \to F$. Hence, we have the identities
\begin{align*}
  (F\beta)_A &= F(\beta_A)
  \\
  (\alpha G)_A &= \alpha_{GA}
  \,.
\end{align*}
The vertical composition of $\alpha$ with a natural transformation $\alpha': F' \to F''$ will be denoted by $\alpha' \circ \alpha: F \to F''$. Its components are given by $(\alpha' \circ \alpha)_B = \alpha'_B \circ \alpha_B$ for all $B \in \catB$. The strict monoidal category of endofunctors on $\catC$ will be denoted by $\End(\calC)$ and the identity endofunctor by $\Id: \catC \to \catC$. Proceeding iteratively, the $n$-fold composition of an endofunctor $F: \catC \to \catC$ with itself is also denoted by juxtaposition
\begin{equation*}
F^n := \underbrace{F \ldots F}_{n \text{ factors}}
\end{equation*}
for all $n \geq 1$. 

\subsubsection{Symmetric structure on an endofunctor}

The following notion is implicit in Ro-sick\'y's definition:

%\begin{Definition}[p.~1 in \cite{Rosicky:1984}]
%Let $F, G: \calC \to \calC$ be functors. A natural transformation $\pi: F \to G$ is called a \textbf{natural (abelian) group bundle over $G$} if it is an (abelian) group object in the category of functors over $G$.
%\end{Definition}

%Explicitly, an (abelian) group structure on $F \stackrel{\pi}{\to} G$ consists of natural transformations $+: F \times_G F \to F$, $0: G \to F$, and $-: F \to F$, which are over $G$, i.e.~$\pi \circ + = \pi \circ \pr$, $\pi \circ 0 = G$, and $\pi \circ - = \pi$, and satisfy the axioms of an (abelian) group. A morphism of bundles of groups from $F_1 \xrightarrow{\pi_1} G_1$ to $F_2 \xrightarrow{\pi_2} G_2$ is a pair of natural transformation $f: F_1 \to F_2$ and $g: G_1 \to G_2$ such that $\pi_2 \circ f = g \circ \pi_1$ and $+_2 \circ (f \times f) = f \circ +_1$. All bundles of natural groups considered here will be abelian.

\begin{Definition}
\label{def:SymStruc}
Let $F: \calC \to \calC$ be a functor and $\tau: F^2 \to F^2$ a natural transformation. Let $\tau_{12} := \tau\,F$ and $\tau_{23}:= F\,\tau$ be the two trivial extensions of $\tau$ to natural transformations $F^3 \to F^3$. We call $\tau$ a \textbf{braiding on $F$} if it satisfies the braid relation $\tau_{12} \circ \tau_{23} \circ \tau_{12} = \tau_{23} \circ \tau_{12} \circ \tau_{23}$. A braiding $\tau$ is called a \textbf{symmetric structure on $F$} if it satisfies $\tau\circ \tau = F^2$.
\end{Definition}

A symmetric structure on $F$ defines an action $S_n \to \Aut(F^n)$ of the symmetric group $S_n$ on $F^n$. In linear categories, in particular in vector spaces, the braid relation is also referred to as Yang-Baxter equation.
%(e.g. \cite[Remark~3.2.3]{Aintablian:2024}).

\begin{Definition}
\label{def:PreserveFibProd}
An endofunctor $F: \calC \to \calC$ \textbf{preserves the fiber products} of a bundle $p:A \to X$, if for all $k \geq 1$ the natural morphism of bundles over $FX$,
\begin{equation}
\label{eq:MorphBundles4}
  \nu_{k,X}:
  F(A \times_X^{p,p} \ldots \times_X^{p,p} A) \longrightarrow
  FA \times_{FX}^{Fp, Fp} \ldots \times_{FX}^{Fp, Fp} FA
  \,,
\end{equation}
where both sides have the same number $k$ of factors, is an isomorphism.
\end{Definition}

%Recall that the identity endofunctor on a category $\catC$ is denoted by $\Id: \catC \to \catC$.

\subsection{Rosick\'{y}'s axioms}
\label{sec:RosickysAxioms}

In~\cite{Rosicky:1984}, Rosick\'{y} introduced the notion of \textit{abstract tangent functors}, which captures the natural categorical structure of the tangent functor of manifolds that is needed to define the Lie bracket of vector fields. 

\begin{Definition}[Section~2 in\cite{Rosicky:1984}, Definition~2.3 in  \cite{CockettCruttwell:2014}]
\label{def:TangentStructure}
A \textbf{tangent structure} on a category $\calC$ consists of a functor $T: \calC \to \calC$, called \textbf{abstract tangent functor}, together with natural transformations $\pi: T \to \Id$, $0: \Id \to T$, $+: T_2 \to T$, $\lambda: T \to T^2$, and $\tau: T^2 \to T^2$, such that the following axioms hold:
\begin{itemize}
%\begin{trivlist}\setlength{\itemsep}{1ex}

\item[(T1)] \textbf{Fiber products:} 
The pullbacks
\begin{equation}
\label{eq:TanFun1}
  T_k := \underbrace{T \times_\Id T \times_\Id \ldots \times_\Id T}_{k \text{ factors}}
\end{equation}
over $T \stackrel{\pi}{\to} \Id$ exist for all $k \geq 1$, are pointwise, and preserved by $T$ (Definition~\ref{def:PreserveFibProd}).

\item[(T2)] \textbf{Bundle of abelian groups:}
$T \stackrel{\pi}{\to} \Id$ with neutral element $0$ and addition $+$ is a bundle of abelian groups over $1$ (Terminology~\ref{term:WibbleBundle}).

\item[(T3)] \textbf{Symmetric structure:}
$\tau: T^2 \to T^2$ is a symmetric structure on $T$ (Definition~\ref{def:SymStruc}). Moreover, $\tau$ is a morphism of bundles of abelian groups. That is, the diagrams
\begin{equation}
\label{eq:TanFun2}
\begin{tikzcd}[column sep=tiny]
T^2 \ar[rr, "\tau"] \ar[dr, "T \pi"'] && T^2 \ar[dl, "\pi T"]
\\
& T &
\end{tikzcd}    
\end{equation}
and
\begin{equation}
\label{eq:TanFun2b}
\begin{tikzcd}[column sep=large]
T^2 \times_T^{T\pi, T\pi} T^2
\ar[r, "\tau \times_T \tau"] 
\ar[d, "\nu_2^{-1}"']
&
T_2 T
\ar[dd, "+T"] 
\\
T T_2  \ar[d, "T+"']
&
\\
T^2 \ar[r, "\tau"'] 
&
T^2
\end{tikzcd}     
\end{equation}
commute, where $\nu_2$ is the morphism~\eqref{eq:MorphBundles4} for $A=TX \xrightarrow{\pi_X} X$, $F=T$, and $k=2$.

%Old diagram
%\begin{tikzcd}
%T T_2 \ar[r, "\nu_2"] \ar[d, "T+"']
%&
%T^2 \times_T^{T\pi, T\pi} T^2  
%\ar[r, "\tau \times_T \tau"] 
%&
%T^2 T
%\ar[d, "+T"] 
%\\
%T^2 \ar[rr, "\tau"'] 
%&& 
%T^2
%\end{tikzcd}

\item[(T4)] \textbf{Vertical lift:} 
The diagrams
\begin{equation}
\label{eq:TanFun3}
\begin{tikzcd}
T \ar[r, "\lambda"] \ar[d, "\pi"'] & T^2 \ar[d, "\pi T"]
\\
\Id \ar[r, "0"'] & T
\end{tikzcd}    
\qquad\qquad
\begin{tikzcd}
T \ar[r, "\lambda"] \ar[d, "\lambda"'] & 
T^2 \ar[d, "\lambda T"]
\\
T^2 \ar[r, "T \lambda"'] & T^3
\end{tikzcd}    
\end{equation}
commute. Moreover, the first diagram is a morphism of bundles of abelian groups, that is $(+T) \circ (\lambda \times_0 \lambda) = \lambda \circ +$.

\item[(T5)] \textbf{Compatibility of vertical lift and symmetric structure:}
The diagrams
\begin{equation}
\label{eq:TanFun5}
\begin{tikzcd}[column sep=tiny]
& T \ar[dl, "\lambda"'] \ar[dr, "\lambda"] &
\\
T^2 \ar[rr, "\tau"'] & &
T^2
\end{tikzcd}    
\qquad\qquad
\begin{tikzcd}
T^2 \ar[r, "T\lambda"] \ar[d, "\tau"'] & 
T^3 \ar[r, "\tau T"] &
T^3 \ar[d, "T\tau"]
\\
T^2 \ar[rr, "\lambda T"'] & &
T^3
\end{tikzcd}    
\end{equation}
commute.

\item[(T6)] \textbf{The vertical lift is a kernel:} The diagram
\begin{equation}
\label{eq:TanFun4}
\begin{tikzcd}
T \ar[r, "\lambda"] \ar[d, "\pi"'] & T^2 \ar[d, "{(\pi T, T\pi)}"]
\\
\Id \ar[r, "{(0, 0)}"'] & T_2
\end{tikzcd}    
\end{equation}
is a pointwise pullback.
%\end{trivlist}
\end{itemize}
A category together with a tangent structure is called a \textbf{tangent category}.
\end{Definition}

\begin{Terminology}
\label{term:negatives}
In \cite{CockettCruttwell:2014} and subsequent work, Rosick\'{y}'s original condition that $T \to \Id$ be a bundle of abelian groups was relaxed to a bundle of abelian monoids. In that terminology, Rosick\'{y}'s notion is called a ``tangent category with negatives'' or a ``Rosick\'{y} tangent category''. All tangent categories in this paper will be with negatives.
\end{Terminology}

\begin{Notation}
\label{not:Difference}
The subtraction of the bundle of abelian groups $T \to 1$, as defined by Diagram~\eqref{diag:difference}, will be denoted by $-:T_2 \to T$.
\end{Notation}

Using the naturality of $\pi$ and $0$, we get the following identities:
\begin{align}
\pi \circ T \pi &= \pi \circ \pi T \,,
\notag\\
T0 \circ 0 &= 0T \circ 0 \,.
\label{eq:T00T}
\end{align}

\begin{Remark}
\label{rem:differentialbundle}
Recall that the tangent space of a finite-dimensional vector space at a point is canonically isomorphic to the vector space itself \cite[Proposition~3.13]{lee}. Thus, the tangent spaces of the fibers of a smooth vector bundle can be identified with the fibers. This key concept has been generalized in \cite{CockettCruttwell:2017}, where the authors develop the notion of differential bundles and fibrations in tangent categories. A differential bundle is a bundle of abelian groups together with a compatible vertical lift \cite[Definition~2.3]{CockettCruttwell:2017}. The vertical lift in a tangent structure and Axiom~\eqref{eq:TanFun3} turn the tangent bundle $TX \to X$ into a differential bundle in this sense. 
\end{Remark}

The vertical lift can be extended by the additive bundle structure to the morphism
\begin{equation*}
  \lambda_2:
  T_2 \xrightarrow{~T0 \times_0 \lambda~}
  T_2 T
  \xrightarrow{~+T~}
  T^2
  \xrightarrow{~\tau~}
  T^2
  \,.
\end{equation*}
In components, 
\begin{equation}
\label{eq:VertLiftExt}
\begin{split}
  \lambda_{2,X}
  &= \tau_X \circ +_{TX} \circ (T0_X \times_{0_X} \lambda_X)
%  \\
%  &= T\!+_X \circ\, \nu_{2,X}^{-1} \circ (0_{TX} \times_{0_X} \lambda_X)
  \,,
\end{split}
\end{equation}
for all $X \in \catC$.
It was shown in \cite[Lemma~3.10]{CockettCruttwell:2014}, assuming all other axioms of a tangent structure (with negatives), that Axiom~\eqref{eq:TanFun4} is satisfied if and only if
\begin{equation}
\label{diag:VertLift}
\begin{tikzcd}
T_2 \ar[r, "\lambda_2"] \ar[d, "\pi \circ \pr_1"'] & 
T^2 \ar[d, "{T\pi}"]
\\
\Id \ar[r, "0"'] & T
\end{tikzcd}    
\end{equation}
is a pointwise pullback.

\subsection{The Lie bracket of vector fields}
\label{sec:LieBracket}

\begin{Definition}[Section~3 in \cite{Rosicky:1984}]
Let $\calC$ be a tangent category. A \textbf{vector field} on $X \in \calC$ is a section of $\pi_X: TX \to X$.
\end{Definition}

The bracket of two vector fields $v,w: X \to TX$ is defined as follows. The composition of $v$ and $Tw: TX \to T^2 X$ satisfies
\begin{equation}
\label{eq:piTwv}
\begin{split}
  \pi_{TX} \circ Tw \circ v 
  &= w \circ \pi_X \circ v
  = w \circ \id_X
  \\
  &= w
  \,.
\end{split}
\end{equation}
When we exchange $v$ and $w$, we have $\pi_{TX} \circ Tv \circ w = v$. In order to be able to subtract the two terms in the fiber product $T^2 X \times_{TX} T^2 X$, we have to apply the symmetric structure on $T^2$, so that by Axiom~\eqref{eq:TanFun2} we obtain
\begin{equation}
\label{eq:piTauTvw}
\begin{split}
  \pi_{TX} \circ \tau_X \circ Tv \circ w 
  &= T\pi_X \circ Tv \circ w
  = T\id_X \circ w
  = \id_{TX} \circ w
  \\
  &= w
  \,.
\end{split}
\end{equation}
This shows that $Tw \circ v$ and $\tau_X \circ Tv \circ w$ project to the same fiber of $\pi_{TX}: T^2 X \to TX$, and hence there is a unique map
\begin{equation*}
(Tw \circ v, \tau_X \circ Tv \circ w): X \longrightarrow T^2X \times_{TX}^{\pi_{TX}, \pi_{TX}} T^2X \,.
\end{equation*}
We denote the subtraction of the two components by
\begin{equation}
\label{eq:DeltaOrig}
  \delta(v,w)
  := 
  -_{TX} \circ 
  (Tw \circ v, \tau_X \circ Tv \circ w): X \longrightarrow T^2X
  \,,
\end{equation}
where the minus $-_{TX}$ denotes the difference in the bundle of abelian groups $\pi_{TX}: T^2 X \to TX$ (Notation~\ref{not:Difference}).
Similarly, observe that
\begin{equation}
\label{eq:vExchw1}
  T\pi_{X} \circ Tw \circ v 
  = T (\pi_X \circ w) \circ v
  = \id_{TX} \circ v
  = v
  \,,
\end{equation}
and
\begin{equation}
\label{eq:vExchw2}
T\pi_{X} \circ \tau_X \circ Tv \circ w
=
\pi_{TX} \circ Tv \circ w 
=
v
\,.
\end{equation}
It follows from Equations~\eqref{eq:vExchw1} and~\eqref{eq:vExchw2} that the left square in the following diagram
\begin{equation}
\label{diag:kerTpi}
\begin{tikzcd}
X
\ar[r, "{(Tw \, \circ \, v, \, \tau_X \, \circ \, Tv \, \circ w)}"]
\ar[d, "\id_X"']
&[4.5em]
T^2X \times_{TX}^{\pi_{TX}, \pi_{TX}} T^2X
\ar[r, "-_{TX}"]
\ar[d, "{T\pi_{X} \times_{\pi_X} T\pi_X}"]
&
T^2X
\ar[d, "T\pi_X"]
\\
X
\ar[r, "{(v,v)}"']
&
TX \times_X^{\pi_X,\pi_X} TX
\ar[r, "-_X"']
&
TX
\end{tikzcd}
\end{equation}
commutes. The commutativity of the right square follows from the naturality of the subtraction $-$. We conclude that the outer rectangle of Diagram \eqref{diag:kerTpi} commutes. Observe that the composition of the upper horizontal maps is $\delta(v,w)$. The composition of the lower horizontal maps is $0_X$ by the axiom of the inverse of a group structure. Therefore, we have
\begin{equation}
\label{eq:DeltaKerOrig}
T\pi_X \circ \delta(v,w) = 0_X \,.
\end{equation}
In other words, the map $\delta(v,w)$ takes values in the kernel of $T\pi_X$, which, by the pullback diagram~\eqref{diag:VertLift}, is isomorphic to $TX \times_X TX$. By projecting to the second factor we thus obtain the vector field $[v,w]: X \to TX$. It is the unique vector field satisfying
\begin{equation}
\label{eq:deltaBracketRel}
  \delta(v,w) = \lambda_{2,X} \circ \big(w, [v,w] \bigr)
  \,.
\end{equation}
The entire construction can be summarized by the following commutative diagram:
\begin{equation}
\label{eq:BracketDef}
\begin{tikzcd}[row sep=3.5em]
TX
&
X 
\ar[r, "{(v,w)}"] 
\ar[l, "{[v,w]}"']
\ar[dr, "{\exists! \ (Tw \, \circ \, v, \, \tau_X \, \circ \, Tv \, \circ \, w)}", dashed, near end]
\ar[dl, "{\exists!}"', dashed]
\ar[d, "{\delta(v,w)}"]
&
TX \times TX \ar[r, "Tw \times Tv"]
&
T^2 X \times T^2 X
\ar[d, "\id_{T^2X} \times \tau_X"]
\\
TX \times_X TX
\ar[u, "\pr_2"]
\ar[dr, phantom, "\lrcorner", very near start]
\ar[r, "\lambda_{2,X}"] 
\ar[d, "\pi_X \circ \pr_1"']
&
T^2 X 
\ar[d, "{T\pi_X}"]
&
T^2 X \times_{TX} T^2 X
\ar[dr, phantom, "\lrcorner", very near start]
\ar[l, "-_{TX}"']
\ar[r, "i"]
\ar[d]
&
T^2 X \times T^2 X
\ar[d, "\pi_{TX} \times \pi_{TX}"]
\\
X 
\ar[r, "{0_X}"']
&
TX
&
TX
\ar[r, "\Delta_{TX}"']
&
TX \times TX
\end{tikzcd}
\end{equation}
We see that all ingredients of the tangent structure are needed. It was announced in \cite{Rosicky:1984} and proved in \cite{CockettCruttwell:2015} with the input of Rosick\'y that $[v,w]$ satisfies the Jacobi relation.

Recall that $\Gamma(X, TX)$ is equipped with the structure of an abelian group given by the addition
\begin{equation*}
  v + w = +_X \circ (v,w)
\end{equation*}
for all $v, w \in \Gamma(X, TX)$ and the zero section $0_X$ \big(see Equation~\eqref{eq:SectionFunctorPlus}\big). From the associativity of the addition it follows that the subtraction $-_{TX}$ is linear in each argument, so that $\delta(v,w)$ and, therefore, the Lie bracket $[v,w]$ is bilinear. We conclude that $\Gamma(X,TX)$ is a $\bbZ$-Lie algebra.

\subsection{Cartesian tangent categories with scalar multiplication}
\label{sec:CatTanScal}

\begin{Definition}
\label{def:CartTan}
A tangent structure on $\calC$ is called \textbf{cartesian} if the tangent functor preserves finite products, that is, if the natural morphism
\begin{equation}
\label{eq:chiXY}
  \chi_{X,Y}: T(X \times Y) 
  \longrightarrow TX \times TY
\end{equation}
has an inverse for all $X,Y \in \calC$ and if $T* \cong *$ for the terminal object $*$ of $\catC$.
\end{Definition}

Definition~\ref{def:CartTan} is logically weaker than \cite[Definition~2.8]{CockettCruttwell:2014}, which additionally requires $\chi$ to define a strong morphism of tangent categories. We have used the same terminology as Cockett and Cruttwell, because we will show in \cite{AintablianBlohmann:cartan} that the two definitions are equivalent.

In a cartesian tangent category, we can define the \textbf{partial tangent morphisms} of a morphism $f: X \times Y \to Z$ by \cite[Definition~2.9]{CockettCruttwell:2014}
\begin{align*}
  \Tpart{1} f: TX \times Y 
  &\xrightarrow{~\id_{TX} \times 0_Y~}
  TX \times TY \xrightarrow{~\chi_{X,Y}^{-1}~}
  T(X \times Y) \xrightarrow{~Tf~} 
  TZ
  \\
  \Tpart{2} f: X \times TY 
  &\xrightarrow{~0_X \times \id_{TY}~}
  TX \times TY \xrightarrow{~\chi_{X,Y}^{-1}~}
  T(X \times Y)\xrightarrow{~Tf~} 
  TZ
  \,,
\end{align*}
where the index refers to the factor in the product. 

A ring object $R$ in $\calC$ gives rise to an endofunctor $\Id \times R: \calC \to \calC$, $X \mapsto X \times R$, which is equipped with the projection $\pr_1: \Id \times R \to \Id$. The ring structure of $R$ equips $\Id \times R \to \Id$ with the structure of a ring internal to endofunctors over $\Id$. An $(\Id \times R \to \Id)$-module in $\End(\calC) \Comma \Id$ will be called a \textbf{bundle of $R$-modules} over $\Id$ (Terminology~\ref{term:WibbleBundle}, Section~\ref{sec:BundleRMod}). An $(\Id \times R \to \Id)$-module structure on the tangent bundle $T \to \Id$ will be called, for short, an \textbf{$R$-module structure} on the tangent category. Explicitly, it consists of a natural morphism
\begin{equation}
  \kappa_X: R \times TX 
  \longrightarrow TX
\end{equation}
of bundles over $X$ such that the usual diagrams of associativity, unitality, and linearity in $R$ and $TX$ commute (see Appendix~\ref{sec:GrpRngMod}).

%\begin{Terminology}
%\label{term:RModStr}
%A tangent category with an $R$-module structure on the tangent bundle $T \to \Id$ will be called, for short, a tangent category with an $R$-module structure.
%\end{Terminology}

\begin{Definition}
\label{def:ModTangentStable}
An $R$-module $A$ in a cartesian tangent category $\calC$ with $R$-module structure will be called \textbf{tangent-stable}, if there is an isomorphism
\begin{equation*}
  T_{\hat{0}} A \cong A
\end{equation*}
of $R$-modules, where $\hat{0}: * \to A$ is the zero of the module, and $T_{\hat{0}} A = * \times_A^{\hat{0}, \, \pi_A} TA$.
\end{Definition}

\begin{Proposition}
\label{prop:TanStabAxA}
An $R$-module $A$ is tangent-stable if and only if its tangent bundle has a trivialization
\begin{equation*}
  TA \cong A \times A
  \,.
\end{equation*}
\end{Proposition}
\begin{proof}
The proof will be given in \cite{AintablianBlohmann:cartan}.
\end{proof}

\begin{Remark}
\label{rmk:DiffObjTanStab}
It follows from Proposition~\ref{prop:TanStabAxA} and \cite[Theorem~6]{Ching:2024} that $A$ is tangent stable if and only if it is a differential object in the sense of \cite[Definition~4.7]{CockettCruttwell:2014}. The condition $T_0 A \cong A$ is logically weaker, easier to check for our examples, and meaningful in more general situations, which is why it is our primary definition.
\end{Remark}

Let $A$ be a tangent-stable $R$-module. The projection onto the first factor of $TA \cong A \times A$ is the bundle projection onto the base. The projection onto the second factor, the fiber of the bundle, will be denoted by 
\begin{equation*}
  \Val_A: TA \longrightarrow A \,.
\end{equation*}

In a tangent category, we have to require an $R$-module structure to be compatible with the rest of the tangent structure in order to obtain the usual relations of a Cartan calculus.

\begin{Definition}
\label{def:TangString}
A commutative ring object in a cartesian tangent category will be called \textbf{tangent-stable} if it is tangent-stable as module over itself.
\end{Definition}

\begin{Definition}
\label{def:RScalar}
Let $R$ be a commutative ring in a cartesian tangent category $\catC$. An $R$-module structure $\kappa_X: R \times TX \to TX$ on the tangent bundle $\pi:T \to \Id$ will be called a \textbf{scalar multiplication} if $R$ is tangent-stable and if the following diagrams commute for all $X \in \calC$:
\begin{equation}
\label{diag:ScalarMult1}
\begin{tikzcd}[column sep={large}]
R \times TX 
\ar[r, "\id_R \times \lambda_X"] 
\ar[d, "\kappa_X"'] &
R \times T^2 X \ar[d, "\kappa_{TX}"]
\\
TX 
\ar[r, "\lambda_X"'] 
&
T^2 X
\end{tikzcd}
\end{equation}
\begin{equation}
\label{diag:ScalarMult2}
\begin{tikzcd}[column sep=6em]
TR \times TX
\ar[r, "\Tpart{1}\kappa_X"]
\ar[d, "{(\pi_R, \Val_R) \, \times \, \id_{TX}}"', "\cong"]
&[4em]
T^2X
%\ar[d, "\tau_X"]
\\
R \times R \times TX
\ar[r, "{\big(\kappa_X \circ (\pr_1,\pr_3), \, 
\kappa_X \circ (\pr_2,\pr_3)\big)}"']
&
T_2 X
\ar[u, "\lambda_{2,X}"']
\end{tikzcd}
\end{equation}
\begin{equation}
\label{diag:ScalarMult3}
\begin{tikzcd}[column sep=3em]
R \times T^2 X
\ar[r, "\Tpart{2}\kappa_X"]
\ar[d, "\id_R \times \tau_X"']
&
T^2 X
%\ar[d, "\tau_X"]
\\
R \times T^2 X
\ar[r, "\kappa_{TX}"']
&
T^2 X
\ar[u, "\tau_X"']
\end{tikzcd}
\end{equation}

% \begin{equation*}
% \begin{tikzcd}[column sep=3em]
% R \times T^2 X
% \ar[r, "0_R \times \tau_X"] 
% \ar[d, "\kappa_{TX}"'] 
% &
% TR \times T^2 X
% \ar[r, "\chi_{R,TX}^{-1}"]
% &
% T(R \times TX)
% \ar[d, "T\kappa_X"]
% \\
% T^2 X
% \ar[rr, "\tau_X"']
% &&
% T^2 X
% \end{tikzcd}
% \end{equation*}
\end{Definition}

\subsection{The Leibniz rule for vector fields and functions}
\label{sec:Leibrule}

The Lie bracket of vector fields on a smooth manifold $M$ satisfies the usual Leibniz rule
\begin{equation*}
  [v,fw] = (v \cdot f)w + f[v,w] 
\end{equation*}
for all smooth maps $f:M \to \bbR$ and vector fields $v,w$ on $M$ \cite[Proposition~8.28~(iv)]{lee}. In this section, we state that this identity holds in the setting of a cartesian tangent category with scalar multiplication over a ring. The detailed proofs of the statements in this sections will be given in \cite{AintablianBlohmann:cartan}.

First, we describe the action of a vector field on a ring-valued morphism on $X$. Let $\catC$ be a tangent category and $R \in \catC$ a ring object with addition $\Radd$, zero $\Rzero$, multiplication $\Rmult$, and unit $\Runit$. Recall from Remark~\ref{rmk:RingFunctions} that the set $\calC(X,R)$ of $R$-valued morphisms is equipped with a ring structure, where the addition and multiplication are given by
\begin{equation*}
  f + g := \Radd \circ (f,g)
  \,,\qquad
  fg := \Rmult \circ (f,g)
\end{equation*}
for all morphisms $f, g \in \calC(X,R)$. The zero is given by $X \to * \stackrel{\Rzero}{\to} R$, and the unit by $X \to * \stackrel{\Runit}{\to} R$.
Assume that $R$ is tangent-stable (Definition~\ref{def:TangString}). Then we can define an action of vector fields on $R$-valued morphisms by
\begin{equation}
\label{eq:VecFieldFunc}
  v \cdot f: X \xrightarrow{~v~} TX \xrightarrow{~Tf~} TR \xrightarrow{~\Val_R~} R 
  %\,.
\end{equation}
for all $v \in \Gamma(X,TX)$ and $f \in \catC(X,R)$.

\begin{Proposition}
\label{prop:VecActFunc}
In a cartesian tangent category $\catC$ with scalar $R$-multiplication, the action~\eqref{eq:VecFieldFunc} is a representation of the Lie algebra of vector fields by derivations on the ring of $R$-valued morphisms. That is,
\begin{equation*}
\begin{split}
%\label{eq:VecActFuncA}
  v \cdot (f + g)
  &= v \cdot f + v \cdot g
  \\
%\label{eq:VecActFuncB}
  (v+w) \cdot f
  &=
  v \cdot f + w \cdot f
  \\
%\label{eq:VecActFunc1}
  [v,w] \cdot f 
  &= v \cdot (w \cdot f) - w \cdot (v \cdot f)
  \\
%\label{eq:VecActFunc2}
  v \cdot (fg) 
  &= (v \cdot f) g + f(v \cdot g)
\end{split}    
\end{equation*}
for all $v, w \in \Gamma(X,TX)$ and $f,g \in \calC(X,R)$.
\end{Proposition}

Some of the key ingredients of the proof of Proposition~\ref{prop:VecActFunc} are the Definition~\eqref{eq:VecFieldFunc} of the action, the functoriality of $T$, the naturality of the tangent structure, Relation~\eqref{eq:deltaBracketRel} between $\lambda_{2,X}$ and $\delta(v,w)$, compatibility relations between the abelian group structures of the ring $R$ and the bundle $\pi_R: TR \to R$, the projection $\eta_R:TR \to R$ onto the fiber, and the natural isomorphism $\chi_{R,R}$ given in~\eqref{eq:chiXY}.

\begin{Remark}
\label{rem:hom}
Proposition~\ref{prop:VecActFunc} shows that there is a homomorphism from the Lie algebra of vector fields on $X$ to the Lie algebra of derivations on the ring $\calC(X,R)$, where the Lie bracket is given by the commutator. However, this homomorphism is generally neither injective nor surjective, so that we cannot identify vector fields on $X$ with derivations on the structure ring of $X$.
\end{Remark}

Assume now that the tangent structure has a scalar $R$-multiplication $\kappa$. In particular, this means that $TX \to X$ is an ($X \times R \to X$)-module in $\calC \Comma X$. Applying the functor of sections, we see that $\Gamma(X, TX)$ is equipped with the structure of a module over $\Gamma(X, X \times R) \cong \calC(X,R)$, given by
\begin{equation*}
  fv := \kappa_X \circ (f,v)
  \,,
\end{equation*}
for all $f \in \catC(X,R)$ and $v \in \Gamma(X,TX)$, as explained in Equation~\eqref{eq:SectionsModuleStr}.

\begin{Proposition}
\label{prop:LeibnizRule}
In a cartesian tangent category $\catC$ with scalar $R$-multiplication, the \textbf{Leibniz rule}
\begin{equation}
\label{eq:LeibnizRule}
  [v,fw] = (v \cdot f)w + f[v,w] 
\end{equation}
holds for all vector fields $v, w \in \Gamma(X,TX)$ and all morphisms $f \in \catC(X,R)$.
\end{Proposition}

Some of the key ingredients of the proof of Proposition~\ref{prop:LeibnizRule} are the Definition~\eqref{eq:VecFieldFunc} of the action, the Definition~\eqref{eq:DeltaOrig} of $\delta(v,w)$, the axioms of the scalar multiplication (Definition~\ref{def:RScalar}), the naturality of the tangent structure and the scalar multiplication, the expression\footnote{This expression can be found in \cite[Proposition~2.10]{CockettCruttwell:2014}, where the authors label the partial tangent morphisms by the objects themselves.} of $Tf$ in terms of the partial tangent morphisms $\Tpart{1}f$ and $\Tpart{2}f$ (Section~\ref{sec:CatTanScal}), the linearity of $\Tpart{2}f$ in the second argument, the associativity of $+$ and $-$, the linearity of $\kappa_{TX}$ in $T^2X$, the Definition~\eqref{eq:VertLiftExt} of $\lambda_{2,X}$, that it is linear in the second argument, Relation~\eqref{eq:deltaBracketRel} between $\lambda_{2,X}$ and $\delta(v,w)$, and that $\lambda_{2,X}$ is a monomorphism. In other words, all the structures and properties we have are needed for the proof of the Leibniz identity.

As is the case for any pullback, the map
\begin{equation*}
\begin{aligned}
  \calC(*,R) 
  &\longrightarrow \calC(X,R)
  \\
  (* \to R)
  &\longmapsto
  (X \xrightarrow{!_X} * \to R)
\end{aligned}
\end{equation*}
is a ring homomorphism (Lemma~\ref{lem:PullbackRingHom}), where $!_X:X \to *$ is the unique morphism to the terminal object. Its image is the constant $R$-valued morphisms on $X$. By precomposing the $\calC(X,R)$-module structure of $\Gamma(X,TX)$ with this ring homomorphism, we equip $\Gamma(X,TX)$ with the structure of a $\calC(*,R)$-module (Lemma~\ref{lem:PullbackRingHom0}). Spelled out, the module structure is given by
\begin{equation}
\label{eq:RModStructure}
rv := \kappa_X \circ (r \, \circ \, !_X, v) \,,
\end{equation}
for all $r \in \catC(*,R)$ and $v \in \Gamma(X,TX)$.

\begin{Corollary}
\label{cor:LieAlgebraBilinear}
Let $X$ be an object in a cartesian tangent category $\catC$ with scalar $R$-multiplication. Then the Lie bracket on $\Gamma(X,TX)$ is $\catC(*,R)$-bilinear.
\end{Corollary}

\subsubsection{Naturality of the Lie bracket}

\begin{Definition}
\label{def:fRel}
Let $\varphi:X \to Y$ be a morphism in a tangent category. A vector field $v$ on $X$ is called \textbf{$\varphi$-related} to a vector field $v'$ on $Y$ if the diagram
\begin{equation}
\label{diag:ProjVect}
\begin{tikzcd}
X \ar[r, "v"] \ar[d, "\varphi"']
&
TX \ar[d, "T\varphi"]
\\
Y \ar[r, "v'"'] 
&
TY
\end{tikzcd}
\end{equation}
commutes. 
\end{Definition}

The terminology of Definition~\ref{def:fRel} is standard in differential geometry. In \cite[Definition~2.8]{CockettCruttwellLemay:2021} a commutative diagram of the form~\ref{diag:ProjVect} is called a vector field morphism.

\begin{Proposition}[Proposition~2.9 in \cite{CockettCruttwellLemay:2021}]
\label{prop:ProjectVectorFields}
Let $\varphi: X \to Y$ be a morphism in a tangent category. If two vector fields $v$ and $w$ on $X$ are $\varphi$-related to vector fields $v'$ and $w'$ on $Y$, respectively, then:
\begin{itemize}

\item[(i)] $v+w$ is $\varphi$-related to $v'+w'$;

\item[(ii)] $[v,w]$ is $\varphi$-related to $[v', w']$.

\end{itemize}
\end{Proposition}

\begin{Definition}
\label{def:fRelFunc}
Let $\varphi:X \to Y$ be a morphism in some category. A morphism $f:X \to R$ is called \textbf{$\varphi$-related} to a morphism $f':Y \to R$ if the diagram
\begin{equation}
\label{diag:ProjFunc}
\begin{tikzcd}
X \ar[r, "f"] \ar[d, "\varphi"']
&
R
\\
Y \ar[ru, "f'"'] 
&
\end{tikzcd}
\end{equation}
commutes. 
\end{Definition}

\begin{Proposition}
\label{prop:ProjectFunctions}
Let $\varphi: X \to Y$ be a morphism in a cartesian tangent category with scalar $R$-multiplication. If $v \in \Gamma(X,TX)$ is $\varphi$-related to $v' \in \Gamma(Y,TY)$ and $f:X \to R$ is $\varphi$-related to $f': Y \to R$, then:
\begin{itemize}

\item[(i)] $fv$ is $\varphi$-related to $f'v'$;

\item[(ii)] $v \cdot f$ is $\varphi$-related to $v' \cdot f'$.

\end{itemize}
\end{Proposition}

\begin{Corollary}
\label{cor:ProjectFunctions}
Let $\varphi: X \to Y$ be a morphism in a cartesian tangent category with scalar $R$-multiplication. Let $r:* \to R$ be a point. If $v \in \Gamma(X,TX)$ is $\varphi$-related to $v' \in \Gamma(Y,TY)$, then $rv$ is $\varphi$-related to $rv'$.
\end{Corollary}

\subsection{The tangent structure of Euclidean spaces}
\label{sec:TangentEuclidean}

The namesake example for tangent structures is the tangent functor of open subsets of real vector spaces, which is the local model for the tangent functor of smooth manifolds. Let $\Eucl$ denote the category which has open subsets of $\bbR^n$, $n\geq 0$, as objects and smooth maps as morphisms. $\Eucl$ will be called the category of \textbf{Euclidean spaces}. Let
\begin{equation*}
  \eFun{T}: \Eucl \longrightarrow \Eucl
\end{equation*}
be its tangent functor. On an open subset $U \subset \bbR^n$, the functors that appear in the Definition~\ref{def:TangentStructure} of a tangent category are given explicitly by
\begin{equation*}
\begin{aligned}
  \eFun{T} U &= U \times \bbR^n
  \\
  \eFun{T}^2 U &= U \times \bbR^n \times \bbR^n \times \bbR^n
  \\
  \eFun{T}^k U &= U \times (\bbR^n)^{2^k-1}
  \\
  \eFun{T}_2 U &= U \times \bbR^n \times \bbR^n
  \\
  \eFun{T}_k U &= U \times (\bbR^n)^k
  \,.
\end{aligned}
\end{equation*}
On a smooth map $f: U \to V \subset \bbR^m$, the functors are given by
\begin{align}
  \eFun{T} f: (u,u_1^i) &\longmapsto
  \Bigl( f(u), \frac{\partial f^a}{\partial x^i} u_1^i \Bigr)
  \label{eq:T1f}\\
  \eFun{T}^2 f: (u, u_1^i, u_2^i, u_{12}^i) &\longmapsto
  \Bigl( f(u), 
  \frac{\partial f^a}{\partial x^i} u_1^i , 
  \frac{\partial f^a}{\partial x^i} u_2^i ,
  \frac{\partial f^a}{\partial x^i} u_{12}^i +
  \frac{\partial^2 f^a}{\partial x^i \partial x^j} u_1^i u_2^j 
  \Bigr)
  \label{eq:T2f}\\
  \eFun{T}_2 f: (u, u_1^i, v_1^i) &\longmapsto
  \Bigl( 
  f(u), 
  \frac{\partial f^a}{\partial x^i} u_1^i ,
  \frac{\partial f^a}{\partial x^i} v_1^i 
  \Bigr)
  \,. \notag
\end{align}
The formulas for $\eFun{T}^k$ and $\eFun{T}_k$ are analogous. The natural transformations of the tangent category are given by
\begin{align*}
  \eFun{\pi}_U : (u,u_1) &\longmapsto u
  \\
  \eFun{0}_U : u &\longmapsto (u, 0)
  \\
  \eFun{+}_U : (u, u_1, v_1) &\longmapsto (u, u_1 + v_1)
  \\
  \eFun{\lambda}_U : (u,u_1) &\longmapsto (u,0,0,u_1)
  \\
  \eFun{\tau}_U : (u, u_1, u_2, u_{12} ) &\longmapsto (u, u_2, u_1, u_{12} )
  \,.
\end{align*}
The negative is given by $(u,u_1) \mapsto (u, -u_1)$ and hence, the subtraction, by
\begin{equation*}
  -_U: (u, u_1, v_1) \longmapsto (u, u_1 - v_1) \,.
\end{equation*}
The commutativity of $\eFun{T}_2$ and $\eFun{T}$ is given by the isomorphism
\begin{equation*}
\begin{aligned}
  \eFun{T}(\eFun{T}_2 U) &\xrightarrow{~\cong~} \eFun{T}_2(\eFun{T} U)
  \\
  \bigl( (u, u_1, v_1), (u_2, u_{12}, v_{12}) \bigr)
  &\longmapsto
  \bigl( ( u, u_2), (u_1, u_{12}), (v_1, v_{12}) \bigr) \,.
\end{aligned}
\end{equation*}
The bundle projection extends to $\eFun{T}^2$ as
\begin{align*}
  (\eFun{\pi}\eFun{T})_U = \eFun{\pi}_{TU}: (u, u_1, u_2, u_{12}) &\longmapsto (u, u_1)
  \\
  (\eFun{T}\eFun{\pi})_U = \eFun{T}\eFun{\pi}_U: (u, u_1, u_2, u_{12}) &\longmapsto (u, u_2)
  \,.
\end{align*}
The other natural transformations that appear in Definition~\ref{def:TangentStructure}, $\eFun{+}\eFun{T}$, $\eFun{T}\eFun{+}$, $\eFun{\lambda}\eFun{T}$, $\eFun{T}\eFun{\lambda}$, $\eFun{\tau}\eFun{T}$, and $\eFun{T}\eFun{\tau}$, are obtained in a similar way. The extension~\eqref{eq:VertLiftExt} of the vertical lift is given by
\begin{equation*}
  \eFun{\lambda}_{2,U}: (u, u_1, v_1) \longmapsto (u, u_1, 0,  v_1)
  \,.
\end{equation*}
Moreover, we have the natural fiberwise multiplication by real numbers
\begin{equation*}
\begin{aligned}
  \eFun{\kappa}_U: \bbR \times \eFun{T} U 
  &\longrightarrow \eFun{T} U
  \\
  \bigl(r, (u,u_1) \bigr)
  &\longmapsto (u, r u_1)
  \,.
\end{aligned}
\end{equation*}

\begin{Proposition}
\label{prop:EuclTangent}
The tangent functor $\eFun{T}: \Eucl \to \Eucl$ with the natural transformations $\eFun{\pi}$, $\eFun{0}$, $\eFun{+}$, $\eFun{\tau}$, $\eFun{\lambda}$, and $\eFun{\kappa}$ is a cartesian tangent structure with scalar $\bbR$-multiplication.
\end{Proposition}
\section{Differentiable groupoids}
\label{sec:DiffGrpds}

The first step towards defining differentiable groupoid objects is to state precisely what we mean by a groupoid internal to a category $\calC$. This is the subject of Section~\ref{sec:GrpdObj}. We then proceed to identify the minimal axioms on a groupoid object needed for its differentiation. In Section~\ref{sec:Differentiable}, we motivate and define the notion of differentiability of groupoid objects in a tangent category $\catC$.

\subsection{Groupoid objects}
\label{sec:GrpdObj}

Let $\sDelta$ be the \textbf{simplex category} which has non-empty finite ordinals $[0]$, $[1]$, $[2]$, etc., as objects, and order-preserving maps as morphisms. A \textbf{simplicial object} in a category $\catC$ is a contravariant functor $X: \sDelta^{\op} \to \catC$. Equivalently, it can be described by a family $\{X_n\}_{n \geq 0}$ of objects in $\catC$ together with morphisms $d_{n,i}:X_n \to X_{n-1}$ and $s_{n,i}:X_n \to X_{n+1}$, called the \textbf{face} and \textbf{degeneracy} morphisms, satisfying the simplicial identities. A simplicial structure can be depicted by
\begin{equation*}
\begin{tikzcd}[column sep=3em,row sep=3em]
X_0 \ar[r]
& X_1 \ar[r, shift left=2] \ar[r, shift right=2] \ar[l, shift left=2] \ar[l, shift right=2]
& X_2 \ar[l] \ar[l, shift right=4] \ar[l, shift left=4] \quad \cdots
\end{tikzcd}
\end{equation*}
The \textbf{standard simplicial $n$-simplex} is defined by $\Delta^n = y[n] = \sDelta\big(\Empty, [n]\big)$, where $y$ is the Yoneda embedding. Its $i^{\text{th}}$ \textbf{horn} is the subsimplicial set $\Lambda^n_i$ which is obtained by removing the $i^{\text{th}}$ face from it as well as its unique non-degenerate $n$-simplex. 

To every set-theoretic groupoid $G$ there is an associated simplicial set $G_{\bullet}$, called its \textbf{nerve}. Its $k$-simplices are given by the pullback
\begin{equation*}
%\label{eq:nervepullback}
\underbrace{G_1 \times_{G_0} \ldots \times_{G_0} G_1}_{\text{$k$ factors}}
\end{equation*}
of the strings of $k$ composable arrows, its face maps by composition by arrows, and the degeneracy maps by inserting identities. This simplicial set satisfies horn filling conditions in all degrees and unique horn filling conditions for degrees $>1$. In other words, the natural horn projection 
\begin{equation}
\label{eq:HornProj}
p_{n,i}:G_n \longrightarrow G(\Lambda^n_i)     
\end{equation}
is a surjection for all $n \geq 1$ and a bijection for all $n>1$, $0 \leq i \leq n$. In fact, one uniquely recovers a groupoid from a given simplicial set with these horn filling conditions. By the Yoneda lemma, there is a natural bijection $G_n \cong \Set^{\sDelta^{\op}}(\Delta^n,G)$ of sets. In~\eqref{eq:HornProj}, $G(\Lambda^n_i) := \Set^{\sDelta^{\op}}(\Lambda^n_i, G)$ is the set of $(n,i)$-horns of $G$ and $p_{n,i}$ is the natural morphism induced by the monomorphism $\Lambda^n_i \to \Delta^n$. We will use this observation for the definition of groupoid objects in any category.

\begin{Definition}
\label{def:Horn}
Let $X$ be a simplicial object in $\catC$. Then, 
\begin{equation}
\label{eq:Ghorns}
X(\Lambda^n_i) := \big(\Ran_{y^{\op}} X\big) (\Lambda^n_i) \cong \lim_{\Delta^k \to \Lambda^n_i} X_k
\end{equation}
is called the \textbf{object of ${(n,i)}$-horns} in $X$ and the morphism
\begin{equation}
\label{eq:HornProjection}
p_{n,i} := X(\Lambda^n_i \to \Delta^n): X_n \longrightarrow X(\Lambda^n_i)    
\end{equation}
is called the \textbf{${(n,i)}$-horn projection} for $n \geq 1$ and $0 \leq i \leq n$.
\end{Definition}

In Equation~\eqref{eq:Ghorns}, $\Ran_{y^{\op}}$ denotes the right Kan extension along the opposite of the Yoneda embedding. The isomorphism is a result of pointwise Kan extensions \cite[Theorem~X.3.1]{maclane1998}. Before we proceed with the definition, let us point out one subtlety. Without any completeness assumptions on $\catC$, these objects are a priori only viewed as presheaves on $\catC$. In our definition of a groupoid object, we explicitly ask for the existence of these horns.

\begin{Definition}
\label{def:GrpdObjC}
Let $\catC$ be a category. A \textbf{groupoid object} in $\catC$ is a simplicial object $G: \sDelta^\op \to \calC$ such that the horns $G(\Lambda^n_i)$ exist in $\catC$ and the horn projections $p_{n,i}: G_n \to G(\Lambda^n_i)$ are isomorphisms for all $n>1$ and $0 \leq i \leq n$.
\end{Definition}

\begin{Definition}
A \textbf{morphism of groupoid objects} in $\catC$ is a morphism of simplicial objects, i.e.\! a natural transformation of functors.
\end{Definition}

Let us recall the description of a groupoid object $G$ in $\catC$ in terms of its structure morphisms. $G_0$ is the object of points (objects) and $G_1$ the object of arrows. The face morphisms $d_{1,0},d_{1,1}: G_1 \to G_0$ are the source $s \equiv d_{1,1}$ and the target $t \equiv d_{1,0}$. The degeneracy morphism $s_{0,0}: G_0 \to G_1$ is the identity bisection $1 \equiv s_{0,0}$. The objects of $(2,0)$-horns, $(2,1)$-horns and $(2,2)$-horns of $G$ are given respectively by
\begin{equation}
\label{eq:HornsG}
\begin{split}
G(\Lambda^2_0) &\cong G_1 \times_{G_0}^{s,s} G_1
\\
G(\Lambda^2_1) &\cong G_1 \times_{G_0}^{s,t} G_1
\\
G(\Lambda^2_2) &\cong G_1 \times_{G_0}^{t,t} G_1 \,.
\end{split}
\end{equation}
It follows by induction from the isomorphism $G_2 \cong G_1 \times_{G_0}^{s,t} G_1$ that
\begin{equation*}
G_k 
\cong \underbrace{G_1 \times_{G_0}^{s,t} \ldots
\times_{G_0}^{s,t} G_1}_{\text{$k$ factors}}
\end{equation*}
for all $k \geq 2$. The groupoid multiplication is given by the composition
\begin{equation*}
m: G_1 \times_{G_0}^{s,t} G_1 \xrightarrow[~\cong~]{~p_{2,1}^{-1}~} G_2 \xrightarrow{~d_{2,1}~} G_1
\,.
\end{equation*}
The inverse is given by the composition
\begin{equation*}
i: G_1 \xrightarrow{~(\id_{G_1}, 1 \circ s)~} G_1 \times_{G_0}^{s,s} G_1 \xrightarrow[~\cong~]{~p_{2,0}^{-1}~} G_2\xrightarrow{~d_{2,0}~} G_1 \,.
\end{equation*}
By the (unique) horn filling conditions and the simplicial identities, the usual axioms of a groupoid hold. These are depicted by the commutativity of the following diagrams:

\begin{itemize}
\item[(i)] Conditions on $s$ and $t$:
\begin{equation}
\label{diag:SourceTargetCond}
\begin{tikzcd}
G_1 \times^{s,t}_{G_0} G_1 \ar[r, "m"] \ar[d, "\pr_2"'] & G_1 \ar[d, "s"]
\\
G_1 \ar[r, "s"'] & G_0
\end{tikzcd}
\hspace{2em}
\begin{tikzcd}
G_1 \times^{s,t}_{G_0} G_1 \ar[r, "m"] \ar[d, "\pr_1"'] & G_1 \ar[d, "t"]
\\
G_1 \ar[r, "t"'] & G_0
\end{tikzcd}
\hspace{2em}
\end{equation}

\begin{equation}
\label{diag:SourceTargetCond2}
\begin{tikzcd}
G_0 \ar[r, "1"] \ar[dr, "\id_{G_0}"'] & G_1 \ar[d, "s"]
\\
& G_0
\end{tikzcd}
\hspace{1em}
\begin{tikzcd}
G_0 \ar[r, "1"] \ar[dr, "\id_{G_0}"'] & G_1 \ar[d, "t"]
\\
& G_0
\end{tikzcd}
\hspace{1em}
\begin{tikzcd}
G_1 \ar[r, "i"] \ar[dr, "t"'] & G_1 \ar[d, "s"]
\\
& G_0
\end{tikzcd}
\hspace{1em}
\begin{tikzcd}
G_1 \ar[r, "i"] \ar[dr, "s"'] & G_1 \ar[d, "t"]
\\
& G_0
\end{tikzcd}
\end{equation}

\item[(ii)] Associativity:
\begin{equation}
\label{eq:Assocofm}
\begin{tikzcd}[column sep=4em, row sep=large]
G_1 \times^{s,t}_{G_0} G_1 \times^{s,t}_{G_0} G_1 \ar[r, "m \times_{G_0} \id_{G_1}"] \ar[d, "\id_{G_1} \times_{G_0} m"'] & G_1 \times^{s,t}_{G_0} G_1 \ar[d, "m"]
\\
G_1 \times^{s,t}_{G_0} G_1 \ar[r, "m"'] & G_1
\end{tikzcd}    
\end{equation}
    
\item[(iii)] Unitality:
\begin{equation}
\label{eq:Unit1}
\begin{tikzcd}[column sep=4em, row sep=large]
G_1 
\ar[dr, "\id_{G_1}"'] 
\ar[r, "{(1 \circ t, \, \id_{G_1})}"] 
& 
G_1 \times^{s,t}_{G_0} G_1 
\ar[d, "m"'] 
& 
G_1 
\ar[dl, "\id_{G_1}"] 
\ar[l, "{(\id_{G_1}, \, 1 \circ s)}"']
\\
& G_1 &
\end{tikzcd} 
\end{equation}
    
\item[(iv)] Invertibility:
\begin{equation}
\label{eq:Inverofi}
\begin{tikzcd}[column sep={large}]
G_1
\ar[r, "{(\id_{G_1}, \, i)}"] 
\ar[d, "t"'] 
&
G_1 \times_{G_0}^{s,t} G_1 
\ar[d, "m"']
&
G_1
\ar[l, "{(i, \, \id_{G_1})}"'] 
\ar[d, "s"] 
\\
G_0 
\ar[r, "1"'] 
&
G_1
&
G_0
\ar[l, "1"] 
\end{tikzcd}
\end{equation}
\end{itemize}
The commutativity of the second and fourth diagrams in~\eqref{diag:SourceTargetCond2} is implied by the first and third diagrams in~\eqref{diag:SourceTargetCond2} and the invertibility axiom~\eqref{eq:Inverofi}.

\begin{Remark}
\label{rem:HornProjEpi}
Let $G:\sDelta^{\op} \to \catC$ be a simplicial object. For $n=1$, the objects of $(1,0)$-horns and $(1,1)$-horns always exist and are isomorphic to the object of $0$-simplices. That is,
\begin{equation*}
G(\Lambda^1_0) \cong G_0 \cong G(\Lambda^1_1) \,.
\end{equation*}
Now, let $G$ be a groupoid object in $\catC$. Under these isomorphisms, the $(1,0)$-horn projection is given by the source morphism $p_{1,0} \equiv d_{1,1} \equiv s$ and the $(1,1)$-horn projection by the target morphism $p_{1,1} \equiv d_{1,0} \equiv t$. Since $t \circ 1 = \id_{G_0}$ and $s \circ 1 = \id_{G_0}$, we conclude that the horn projections $p_{1,0}$ and $p_{1,1}$ are split epimorphisms, and hence regular epimorphisms.
\end{Remark}

Definition~\ref{def:Horn} can be stated more generally as follows. If $X:\sDelta^{\op} \to \catC$ is a simplicial object and $S:\sDelta^{\op} \to \Set$ is a simplicial set, we define
\begin{equation*}
X(S) := \big(\Ran_{y^{\op}} X\big) (S) \cong \lim_{\Delta^k \to S} X_k \,.
\end{equation*}
Spelling out this limit in terms of equalizers and products using \cite[Theorem~V.2.2]{maclane1998}, we observe that $X(S) \equiv \Hom(S,X)$ is equivalent to Henrique's definition \cite[Equation~(8)]{henriques2008}. If the limit does not exist in $\catC$ we view it as a presheaf on $\catC$ as mentioned earlier.

In \cite{henriques2008}, \cite{Zhu:2009} and \cite{MeyerZhu:2015}, the authors solve the subtlety of the existence of the horns by equipping the category with a \textit{Grothendieck pretopology} with additional good properties. A pretopology defines a class of morphisms called covers that play the role of surjective submersions in the category of smooth manifolds. In that case, the horn $X(\Lambda^n_i)$ exists as an object in $\catC$, using the fact that $\Lambda^n_i$ is a \textit{collapsible} simplicial set \cite[Corollary~2.5]{henriques2008}. This makes it possible to talk about a geometric analog of the horn filling conditions of simplicial sets. Our definition of groupoid objects works in any category without any extra structure of a pretopology.

Furthermore, Lurie defines groupoid objects in $\catC$ using a similar simplicial approach \cite[Section~6.1.2]{lurie2009}. However, $\catC$ is assumed to have finite limits, a condition which we do not assume for reasons explained above. In his definition of groupoid objects in the context of higher categories, the author requires the existence of pullback squares, which encode both the existence of the horns and the horn filling conditions in one step \cite[Definition~6.1.2.7, Proposition~6.1.2.6~(4")]{lurie2009}. For instance, $G_2$ is required to be isomorphic to each of the fiber products in~\eqref{eq:HornsG}. We do not employ this approach since for us it is important to emphasize how these isomorphisms are obtained via the horn projections.

\subsection{Differentiability}

\label{sec:Differentiable}

One of the main goals of this section is to identify the minimal axioms needed on a groupoid object in order to differentiate it to its (abstract) Lie algebroid. Let us first fix some notation. 

\subsubsection{Notation on the iterated and diagonal zero sections}

Let $\catC$ be a tangent category. For an object $X \in \catC$, let
\begin{equation}
\label{eq:HigherZeroSection}
 \Zero{n}_X : X \longrightarrow T^n X
\end{equation}
denote the natural morphism given by iterating the zero section for $n \geq 1$. For example, we have
\begin{equation*}
  \Zero{n}_X
  = T^{n-1}0_X \circ \ldots \circ T0_X \circ 0_X
  \,.
\end{equation*}
Since $T0_X \circ 0_X = 0_{TX} \circ 0_X$ from Equation~\eqref{eq:T00T}, there are many different ways of expressing this morphism. In the trivial case $n=0$, we define $\Zero{0}_X := \id_{X}$.  Furthermore, let
\begin{equation*}
 \ZeroDiag{m}{X} := \underbrace{(0_X, \ldots, 0_X)}_{\text{$m$ factors}} : X \longrightarrow T_m X
\end{equation*}
denote the diagonal zero section for all $m \geq 1$. 

Let $G: \sDelta^\op \to \calC$ be a groupoid object. Denote by
\begin{equation}
\label{eq:ZeroVertexMap}
  t_k := t \circ \pr_1: G_k \longrightarrow G_0
\end{equation}
the morphism that sends the $k$-simplices in $G_k$ to their terminal vertex for $k \geq 1$. In the trivial case $k=0$, we define $t_0 := \id_{G_0}$. Using the naturality of the (iterated and diagonal) zero section, we obtain the relations
\begin{align}
\label{eq:lkzeroRel1}
  T^n t_k \circ \Zero{n}_{G_k} 
  &= \Zero{n}_{G_0} \circ t_k
  \\
\label{eq:lkzeroRel2}
  T_m t_k \circ \ZeroDiag{m}{G_k} 
  &= \ZeroDiag{m}{G_0} \circ t_k
  \,.
\end{align}

\subsubsection{The main definition, its motivation and consequences}

\begin{Definition}
\label{def:DiffGroupoid}
A groupoid object $G$ in a tangent category $\catC$ will be called \textbf{differentiable} if the pullbacks in the diagrams
\begin{equation}
\label{diag:DiffGroupoidG01}
\begin{tikzcd}[column sep=small]
T^n G_1 \times_{T^n G_0} G_k
\ar[d] \ar[r]
\arrow[dr, phantom, "\lrcorner", very near start] 
&
G_k
\ar[d, "\Zero{n}_{G_k}"]
\\
T^n G_1 \times_{T^n G_0} T^n G_k
\ar[d] \ar[r]
\arrow[dr, phantom, "\lrcorner", very near start] 
&
T^n G_k
\ar[d, "T^n t_k"]
\\
T^n G_1
\ar[r, "T^n s"']
&
T^n G_0
\end{tikzcd}
\quad
\begin{tikzcd}[column sep=small]
T_m G_1 \times_{T_m G_0} G_k
\ar[d] \ar[r]
\arrow[dr, phantom, "\lrcorner", very near start] 
&
G_k
\ar[d, "\ZeroDiag{m}{G_0} \circ \, t_k"]
\\
T_m G_1
\ar[r, "T_m s"']
&
T_m G_0    
\end{tikzcd}
\end{equation}
\begin{equation}
\label{diag:DiffGroupoidG02}
\begin{tikzcd}[column sep=small]
G_0 \times_{G_1} TG_1 \times_{TG_0} G_0
\ar[d] \ar[r]
\arrow[dr, phantom, "\lrcorner", very near start] 
&
TG_1 \times_{TG_0} G_0
\ar[d, "\pi_{G_1} \circ \, \pr_1"]
\\
G_0
\ar[r, "1"']
&
G_1    
\end{tikzcd}
\end{equation}
exist, and if the natural morphism
\begin{equation}
\label{eq:TFinPullG}
  T^n(G_1 \times_{G_0}^{s,t_k} G_k) 
  \longrightarrow T^n G_1 \times_{T^n G_0}^{T^ns,T^nt_k} T^n G_k
\end{equation}
is an isomorphism for all $n \geq 1$, $m \geq 2$ and $k \geq 0$.
\end{Definition}

\begin{Example}
\label{ex:LieGrpdDiff}
A groupoid object in the tangent category $\Mfld$ is differentiable if and only if it is a Lie groupoid (see Proposition~\ref{prop:LieGrpdDiff}).
\end{Example}

Let us motivate Definition~\ref{def:DiffGroupoid}. Let $G:\sDelta^\op \to \calC$ be a differentiable groupoid. For $n=1$, the outer pullback on the left of~\eqref{diag:DiffGroupoidG01} is the pullback of $Ts$ along the zero section~\eqref{eq:lkzeroRel1}. The existence of this pullback makes it possible to restrict the tangent bundle $TG_1$ \textit{vertically} with respect to the source fibers. Similarly, for $n \geq2$ we get higher vertical restrictions of the tangent bundle. 

For $m \geq 2$, the pullback on the right of~\eqref{diag:DiffGroupoidG01} is the pullback of $T_ms$ along the diagonal zero section~\eqref{eq:lkzeroRel2}. The existence of this pullback enables us to talk about fiber products of the vertical tangent bundle. This will be necessary when we vertically restrict the fiberwise addition of the tangent structure. Higher vertical tangent bundles and vertical restrictions of tangent structures are studied thoroughly in Section~\ref{sec:HigherVertTanFun} in the context of right groupoid bundles.

The pullback in~\eqref{diag:DiffGroupoidG02} is needed for the definition of the abstract Lie algebroid of $G$, where we consider the pullback~\eqref{eq:LABundle} of the vertical tangent bundle along the identity bisection.

To understand the importance of the existence of the pullback and the isomorphism in~\eqref{eq:TFinPullG}, we consider the case for $n=k=1$, which states that the natural morphism
\begin{equation*}
  T(G_1 \times_{G_0}^{s,t} G_1) 
  \longrightarrow TG_1 \times_{TG_0}^{Ts,Tt} TG_1
\end{equation*}
is an isomorphism. This implies that for the simplicial object $TG:\sDelta^{\op} \to \catC$, given by $(TG)_k=TG_k$, the object of $(2,1)$-horns 
\begin{equation*}
(TG)(\Lambda^2_1) \cong TG_1 \times_{TG_0}^{Ts,Tt} TG_1    
\end{equation*}
exists in $\catC$. Moreover, the horn projection
\begin{equation*}
TG_2 \xrightarrow[~Tp_{2,1}~]{~\cong~} T(G_1 \times_{G_0}^{s,t} G_1) \xrightarrow{~\cong~} TG_1 \times_{TG_0}^{Ts,Tt} TG_1 \,,
\end{equation*}
being the composition of isomorphisms, is an isomorphism. Using the isomorphism $Ti:TG_1 \to TG_1$ and the existence of $G_1 \times_{G_0}^{s,s} G_1$ and $G_1 \times_{G_0}^{t,t} G_1$ in $\catC$, we conclude that the objects of $(2,0)$-horns and $(2,2)$-horns of $TG$, given respectively by
\begin{equation*}
\begin{split}
(TG)(\Lambda^2_0) &\cong TG_1 \times_{TG_0}^{Ts,Ts} TG_1
\\
(TG)(\Lambda^2_2) &\cong TG_1 \times_{TG_0}^{Tt,Tt} TG_1 \,,
\end{split}
\end{equation*}
exist in $\catC$ too. Analogously, the respective horn projections are isomorphisms. Proceeding similarly, we conclude that $TG$ is a groupoid object in $\catC$.

The structure maps of $TG$ are inherited from those on the groupoid $G$. Explicitly, its source, target, unit and inverse morphisms are the tangent morphisms of those of $G$. The multiplication of $TG$ is given by
\begin{equation*}
TG_1 \times_{TG_0}^{Ts,Tt} TG_1 \xleftarrow{~\cong~} T(G_1 \times_{G_0}^{s,t} G_1) \xrightarrow{~Tm~} TG_1 \,.
\end{equation*}
If $k=2$ and by applying the isomorphism~\eqref{eq:TFinPullG} twice we get that the natural morphism
\begin{equation*}
TG_3 \longrightarrow TG_1 \times_{TG_0}^{Ts,Tt} TG_1 \times_{TG_0}^{Ts,Tt} TG_1
\end{equation*}
is an isomorphism. This allows us to talk about associativity of the multiplication of $TG$.

For general $n \geq 1$ and $k \geq 0$, it follows by induction from the isomorphism~\eqref{eq:TFinPullG}, that the natural morphism
\begin{equation}
\label{eq:TnGkIso}
  T^n G_k 
  \longrightarrow
  \underbrace{T^n G_1 \times_{T^n G_0}^{T^n s,T^n t} \ldots 
    \times_{T^n G_0}^{T^n s,T^n t} T^n G_1}_{\text{$k$ factors}}
\end{equation}
is an isomorphism. By similar arguments as for the case when $n=1$, this implies that $T^n G: \sDelta^\op \to \calC$ given by $(T^n G)_k = T^n G_k$ is a groupoid object in $\catC$. Analogously, we get that $T_m G: \sDelta^\op \to \calC$, given by $(T_m G)_k = T_m G_k$ is a groupoid object in $\catC$ for all $m \geq 2$.

\begin{Remark}
For the construction of the abstract Lie algebroid of a differentiable groupoid $G$, we only need the existence of the pullbacks in Definition~\ref{def:DiffGroupoid} for $n=1,2$, $m \geq 2$ and $k \geq 0$. The reason that we ask for the existence of these pullbacks for all $n \geq 1$ can be summarized as follows:
\begin{itemize}
\item[(1)]
The concepts of the higher vertical tangent bundle and the vertical restriction of the tangent structure that we develop in Section~\ref{sec:HigherVertTanFun} make sense for all $n \geq 0$ (with $n=0$ being the trivial case). This enables us to talk about the functor $\VTan{n}$, which behaves like the powers of the vertical restriction functor $V$ (see Remarks~\ref{rem:VEnotDiff0} and~\ref{rem:VEnotDiff}).
\item[(2)]
It is the first step to getting a vertical restriction endofunctor on differentiable groupoids, and more generally an endofunctor on differentiable right groupoid bundles (Definition~\ref{def:DiffBunG}). This makes it possible to obtain an honest tangent structure on differentiable right groupoid bundles (and not just a vertical restriction of the tangent structure). 
\item[(3)]
It seems unnatural to only consider the cases for $n=1$ and $n=2$.
\end{itemize}
\end{Remark}
\section{Differentiable groupoid bundles}
\label{sec:Gbun}

Given a groupoid object $G$, its source morphism $s:G_1 \to G_0$ is a $G$-bundle together with the right groupoid multiplication. If $G$ is differentiable, the tangent morphism $Ts:TG_1 \to TG_0$ is a $TG$-bundle. Using the axioms of differentiability, the pasting lemma and universal constructions, we will show that the restriction of the tangent bundle \textit{vertically} with respect to the source fibers yields a $G$-bundle
\begin{equation*}
VG_1 := TG_1 \times_{TG_0}^{Ts, \, 0_{G_0}} G_0 \longrightarrow G_0 \,.
\end{equation*}
In this section, we develop the language of $G$-bundles, which will be the cornerstone of proving that the Lie bracket of (right) invariant vector fields is invariant in the setting of tangent categories (Section~\ref{sec:LieBracketInv}). This will be crucial for the definition of the Lie bracket on the sections of the abstract Lie algebroid of $G$ via its identification to the invariant vector fields (Theorem~\ref{thm:SecAInvVec}).

In Section~\ref{sec:GbunEqMor}, we define the category of right $G$-bundles with $G$-equivariant bundle morphisms. We then discuss the diagonal action of a groupoid object on pullbacks of groupoid bundles in Section~\ref{sec:diagonal}. Lastly, we generalize the notion of differentiability of groupoids to groupoid bundles in a tangent category in Section~\ref{sec:DifferentiabilityGrpdBun}. This generalization allows us to make sense of (higher) vertical tangent bundles and their fiber products. For a gentle introduction to Lie groupoids and their Lie algebroids we refer the reader to \cite{CannasdaSilvaWeinstein:Models}.

\subsection{Groupoid bundles and equivariant morphisms}
\label{sec:GbunEqMor}

\begin{Definition}
\label{def:rightGbun}
Let $G$ be a groupoid object in a category $\calC$. A \textbf{right groupoid bundle} or \textbf{right $G$-bundle} is a morphism $r:E \to G_0$ in $\catC$ such that the pullbacks
\begin{equation}
\label{eq:EGkNerve}
E \times^{r, t_k}_{G_0} G_k \cong E \times_{G_0}^{r,t} \underbrace{G_1 \times_{G_0}^{s,t} \ldots \times_{G_0}^{s,t} G_1}_{\text{$k$ factors}}
\end{equation}
exist for all $k \geq 1$, together with a morphism  
\begin{equation*}
\beta: E \times^{r,t}_{G_0} G_1 \longrightarrow E \,,
\end{equation*}
called the \textbf{right action}, such that the following diagrams commute:
\begin{itemize}
\item[(i)] Condition on fibers:
\begin{equation*}
\begin{tikzcd}
E \times^{r,t}_{G_0} G_1 \ar[r, "\beta"] \ar[d, "\pr_2"'] 
& 
E \ar[d, "r"]
\\
G_1 \ar[r, "s"'] 
& 
G_0
\end{tikzcd}
\end{equation*}

\item[(ii)] Unitality:
\begin{equation*}
\begin{tikzcd}[column sep=large]
E \ar[dr, "\id_E"'] \ar[r, "{(\id_E, 1 \circ r)}"] 
& 
E \times^{r,t}_{G_0} G_1 \ar[d, "\beta"]
\\
& 
E
\end{tikzcd}
\end{equation*}

\item[(iii)] Associativity:
\begin{equation*}
\begin{tikzcd}[column sep=4em, row sep=large]
E \times^{r,t_2}_{G_0} G_2 
\ar[r, "\id_E \times_{G_0} m"] 
\ar[d, "\beta \times_{G_0} \id_{G_1}"'] 
& 
E \times^{r,t}_{G_0} G_1 \ar[d, "\beta"]
\\
E \times^{r,t}_{G_0} G_1 \ar[r, "\beta"'] 
& 
E
\end{tikzcd}
\end{equation*}
\end{itemize}
\end{Definition}

\begin{Definition}
\label{def:GBunMor}
Let $r_E: E \to G_0$ and $r_F: F \to G_0$ be right $G$-bundles. A morphism $\Phi: E \to F$ in $\catC$ is called a \textbf{morphism of bundles} if the diagram
\begin{equation*}
\begin{tikzcd}[column sep=1em]
E \ar[dr, "r_E"'] \ar[rr, "{\Phi}"] 
& & 
F \ar[dl, "r_F"]
\\
& G_0 &
\end{tikzcd}
\end{equation*}
commutes. It is called \textbf{$G$-equivariant} if the diagram
\begin{equation*}
\begin{tikzcd}[column sep=4em]
E \times^{r_E,t}_{G_0} G_1 \ar[r, "\Phi \times_{G_0} \id_{G_1}"] \ar[d, "\beta_E"'] 
& F \times^{r_F,t}_{G_0} G_1 \ar[d, "\beta_F"]
\\
E \ar[r, "\Phi"'] & F
\end{tikzcd}
\end{equation*}
commutes, where the vertical arrows are the $G$-actions.
\end{Definition}

\begin{Terminology}
\label{term:right}
In this paper we will only consider right groupoid bundles. Hence, we will often refer to right groupoid bundles by groupoid bundles or $G$-bundles, dropping the adjective \textit{right}. 
\end{Terminology}

Given a groupoid object $G$ in $\catC$, $G$-bundles together with $G$-equivariant bundle morphisms form a category, which will be denoted by $\BunG$.

\begin{Remark}
\label{rmk:ActGrpdNerve}
The pullbacks~\eqref{eq:EGkNerve} required to exist are the levels of the action groupoid $E \rtimes G \rightrightarrows E$. Such a condition is generally needed if the ambient category does not have finite limits. There is a fully simplicial version of Definition~\ref{def:rightGbun}, which uses bisimplicial objects or higher correspondences \cite{BlohmannKrishnaZhu}.
\end{Remark}

% In short, right $G$-groupoid bundles will be called \textbf{$G$-bundles} and morphisms between them \textbf{$G$-equivariant maps}. A $G$-bundle is usually depicted by
% \begin{equation*}
% \begin{tikzcd}[row sep=2.3em, column sep=0.1cm]
% E
% \ar[rrd, "r"']
% &
% \mathlarger{\mathlarger{\mathlarger{\circlearrowleft}}}
% &
% G_1
% \ar[d, shift left=1.5, "s"]
% \ar[d, shift right=1.5, "t"']
% \\
% & &
% G_0
% \end{tikzcd}
% \end{equation*}

\begin{Example}
\label{ex:G0}
The identity map $E:=G_0 \xrightarrow{~\id~} G_0$ with the right $G$-action $\beta := s \circ \pr_2: G_0 \times_{G_0}^{\id,t} G_1 \to G_0$ is the terminal object in $\BunG$.
\end{Example}

\begin{Example}
\label{ex:G1}
The source morphism $G_1 \xrightarrow{~s~} G_0$ and the right groupoid multiplication $\beta:=m: G_1 \times_{G_0}^{s,t} G_1 \to G_1$ equip $E := G_1$ with the structure of a $G$-bundle.
\end{Example}

\subsubsection{Restriction of a groupoid bundle}

The following lemma shows that if we restrict a groupoid bundle $F$ to a subbundle $E \to F$ and pull back the action along a morphism of groupoids which is an inclusion at the level of 0-simplices, the restricted action on $E$ takes the target fibers to the source fibers, and inherits unitality and associativity from the ambient bundle $F$.

\begin{Lemma}
\label{lem:ActionRestrict}
Let $G$ and $H$ be groupoids; let $r_E: E \to G_0$ be a bundle together with a morphism $\beta_E: E \times_{G_0}^{r_E,t_G} G_1 \to E$; let $r_F: F \to H_0$ be an $H$-bundle with action $\beta_F$; let $\phi: G \to H$ be a morphism of groupoids with simplicial components $\phi_k: G_k \to H_k$ for all $k \geq 0$; let $i: E \to F$ be a morphism in $\catC$ such that the diagrams
\begin{equation}
\label{eq:ActionRestrict01}
\begin{tikzcd}
E
\ar[r, "i"]
\ar[d, "r_E"']
&
F
\ar[d, "r_F"]
\\
G_0
\ar[r, "\phi_0"']
&
H_0
\end{tikzcd}
\qquad
\begin{tikzcd}[column sep=3em]
E \times_{G_0} G_1
\ar[r, "i \times_{\phi_0} \phi_1"]
\ar[d, "\beta_E"']
&
F \times_{H_0} H_1
\ar[d, "\beta_F"]
\\
E
\ar[r, "i"']
&
F
\end{tikzcd}
\end{equation}
commute. If $i$ and $\phi_0$ are monomorphisms, then $\beta_E$ is a right $G$-action on $E$.
\end{Lemma}

\begin{proof}
For the unitality, we consider the following diagram:
\begin{equation*}
\begin{tikzcd}
E
\ar[rrr, "{(\id_E, \, 1_G \circ r_E)}"] 
\ar[ddd, "\id_E"'] 
\ar[dr, "i"]
& &[2em] &[-2.5em]
E \times_{G_0} G_1
\ar[ddd, "\beta_E"] 
\ar[dl, "i \times_{\phi_0} \phi_1"']
\\
&
F
\ar[d, "\id_F"'] 
\ar[r, "{(\id_F, \, 1_H \circ r_F)}"] 
&
F \times_{H_0} H_1 \ar[d, "\beta_F"]
&
\\
&
F\ar[r, "{\id_F}"'] &
F
&
\\
E \ar[rrr, "{\id_E}"']
\ar[ur, "i"'] 
& & &
E
\ar[ul, "i"]
\end{tikzcd}    
\end{equation*}
The inner square commutes since by assumption $F$ is a groupoid bundle, so that the $H$-action on $F$ is unital. The right trapezoid is the commutative diagram on the right of~\eqref{eq:ActionRestrict01}. The upper trapezoid commutes since $\phi$ is a morphism of groupoids and by using the commutativity of the left diagram of \eqref{eq:ActionRestrict01}. Spelled out,
\begin{equation*}
\begin{split}
(i \times_{\phi_0} \phi_1) \circ (\id_E, 1_G \circ r_E) 
&= (i \circ \id_E, \phi_1 \circ 1_G \circ r_E) 
\\
&= (\id_F \circ i, 1_H \circ \phi_0 \circ r_E)
\\
&= (\id_F \circ i, 1_H \circ r_F \circ i)
\\
&= (\id_F, 1_H \circ r_F) \circ i
\,.
\end{split}
\end{equation*}
The left and bottom trapezoids commute trivially. Since $i$ is by assumption a monomorphism, it follows from Lemma~\ref{lem:InnerOuterSquares} that the outer square commutes. This shows that $\beta_E$ is unital.

For the associativity, we consider the following diagram:
\begin{equation*}
\begin{tikzcd}
E \times_{G_0} G_2
\ar[rrr, "\id_{E} \times_{G_0} m_G"] 
\ar[ddd, "\beta_E \times_{G_0} \id_{G_1}"'] 
\ar[dr, "i \times_{\phi_0} \phi_2"]
&[-2em] &[1.8em] &[-2em]
E \times_{G_0} G_1
\ar[ddd, "\beta_E"] 
\ar[dl, "i \times_{\phi_0} \phi_1"']
\\
&
F \times_{H_0} H_2
\ar[d, "\beta_F \times_{H_0} \id_{H_1}"'] 
\ar[r, "{\id_F \times_{H_0} m_H}"] 
&
F \times_{H_0} H_1 \ar[d, "\beta_F"]
&
\\
&
F \times_{H_0} H_1 \ar[r, "{\beta_F}"'] &
F
&
\\
E \times_{G_0} G_1 \ar[rrr, "{\beta_E}"']
\ar[ur, "i \times_{\phi_0} \phi_1"'] 
& & &
E
\ar[ul, "i"]
\end{tikzcd}    
\end{equation*}
The inner square commutes since $F$ is by assumption a groupoid bundle, so that the $H$-action on $F$ is associative. The right and bottom trapezoids are both the commutative diagram on the right of~\eqref{eq:ActionRestrict01}. The upper trapezoid commutes since $\phi$ is a morphism of groupoids and by functoriality. Spelled out,
\begin{equation*}
\begin{split}
(i \times_{\phi_0} \phi_1) \circ (\id_E \times_{G_0} m_G) 
&= (i \circ \id_E) \times_{\phi_0} (\phi_1 \circ m_G) 
\\
&= (\id_F \circ i) \times_{\phi_0} (m_H \circ \phi_2)
\\
&= (\id_F \times_{H_0} m_H) \circ (i \times_{\phi_0} \phi_2)
\,.
\end{split}
\end{equation*}
The commutativity of the left trapezoid follows from the commutativity of the right diagram of~\eqref{eq:ActionRestrict01}, observing that $G_2 \cong G_1 \times_{G_0} G_1$ and $\phi_2 = \phi_1 \times_{\phi_0} \phi_1$. Since $i$ is by assumption a monomorphism, it follows from Lemma~\ref{lem:InnerOuterSquares} that the outer square commutes. This shows that $\beta_E$ is associative.

To show that $\beta_E$ takes the target fibers to the source fibers of $G$, we consider the following diagram:
\begin{equation*}
\begin{tikzcd}
E \times_{G_0} G_1
\ar[rrr, "{\beta_E}"] 
\ar[ddd, "\pr_2"'] 
\ar[dr, "i \times_{\phi_0} \phi_1"]
&[-2.5em] & &
E
\ar[ddd, "r_E"] 
\ar[dl, "i"']
\\
&
F \times_{H_0} H_1
\ar[d, "\pr_2"'] 
\ar[r, "{\beta_F}"] 
&
F \ar[d, "r_F"]
&
\\
&
H_1 \ar[r, "{s_H}"'] &
H_0
&
\\
G_1 \ar[rrr, "{s_G}"']
\ar[ur, "\phi_1"'] 
& & &
G_0
\ar[ul, "\phi_0"]
\end{tikzcd}    
\end{equation*}
The inner square commutes since by assumption $F$ is a groupoid bundle, so that the $H$-action on $F$ takes the target fibers to the source fibers of $H$. The right and upper trapezoids are the commutative diagrams in~\eqref{eq:ActionRestrict01}. The lower trapezoid commutes since $\phi$ is a morphism of groupoids. The left trapezoid commutes trivially. Since $\phi_0$ is by assumption a monomorphism, it follows from Lemma~\ref{lem:InnerOuterSquares} that the outer square commutes. Hence, $\beta_E$ takes the target fibers to the source fibers of $G$.

We conclude that $\beta_E$ satisfies the commutative diagrams in Definition~\ref{def:rightGbun} and thus is a right $G$-action on $E$.
\end{proof}

\subsection{Pullbacks and the diagonal action}
\label{sec:diagonal}

Let $r_A: A \to G_0$, $r_B: B \to G_0$ and $r_C: C \to G_0$ be $G$-bundles with right actions $\beta_A$, $\beta_B$, and $\beta_C$. Let $\phi: A \to B$ and $\psi: C \to B$ be $G$-equivariant bundle morphisms. Assume that the pullbacks
\begin{equation}
\label{eq:DiagActPullback}
  (A \times_B C) \times_{G_0}^{r_C \circ \pr_2, \, t_k} G_k
\end{equation}
exist in $\calC$ for all $k \geq 0$. The pullback $A \times_B C$ is equipped with the natural morphism
\begin{equation*}
  r_{A \times_B C}:
  A \times_B C \longrightarrow G_0
  \,,
\end{equation*}
which is given explicitly by $r_{A \times_B C} = r_A \circ \pr_1 = r_C \circ \pr_2$. The condition that $\phi$ and $\psi$ are $G$-equivariant is equivalent to the commutativity of the following diagram:
\begin{equation*}
\begin{tikzcd}[column sep=4em, row sep=2em]
A \times_{G_0} G_1
\ar[r, "\phi \times_{G_0} \id_{G_1}"] 
\ar[d, "\beta_A"']
&
B \times_{G_0}  G_1 
\ar[d, "\beta_B"]
&
C \times_{G_0} G_1
\ar[l, "\psi \times_{G_0} \id_{G_1}"']
\ar[d, "\beta_C"]
\\
A \ar[r,"\phi"']
&
B
&
C \ar[l, "\psi"]
\end{tikzcd}
\end{equation*}
This induces a unique morphism from the limit of the top row to the limit of the bottom row,
\begin{equation}
\label{eq:FiberProdActions1}
  \beta_A \times_{\beta_B} \beta_C: (A \times_{G_0} G_1) \times_{B \times_{G_0} G_1} (C \times_{G_0} G_1) \longrightarrow A \times_B C
  \,.
\end{equation}
The domain is isomorphic to
\begin{equation*}
\begin{split}
  (A \times_{G_0} G_1) \times_{B \times_{G_0} G_1} (C \times_{G_0} G_1)
  &\cong (A \times_B C) \times_{G_0 \times_{G_0} G_0}
  (G_1 \times_{G_1} G_1)
  \\
  &\cong
  (A \times_B C) \times_{G_0} G_1
  \,,
\end{split}
\end{equation*}
where we have used that pullbacks commute with pullbacks and the cancellation of pullbacks over the identity. By composing the isomorphism with~\eqref{eq:FiberProdActions1}, we obtain the \textbf{diagonal action}
\begin{equation}
\label{eq:DiagAct}
  \beta_{A \times_B C}: A \times_B C \times_{G_0} G_1
  \longrightarrow
  A \times_B C
  \,.
\end{equation}
Explicitly, we can express the diagonal action as
\begin{equation}
\label{eq:DiagActExp}
  \beta_{A \times_B C}
  = \bigl( \beta_A \circ (\pr_1,\pr_3), \beta_C \circ (\pr_2,\pr_3) \bigr)
  \,.
\end{equation}

\begin{Lemma}
The morphism~\eqref{eq:DiagAct} is a right $G$-action on $A \times_B C$. 
\end{Lemma}

\begin{proof}
To prove unitality, we write $\id_{A \times_B C} = (\pr_1, \pr_2)$ and calculate
\begin{equation*}
\begin{split}
&\beta_{A \times_B C} \circ (\id_{A \times_B C}, 1 \circ r_{A \times_B C})
\\
= \
&\beta_{A \times_B C} \circ (\pr_1, \pr_2, 1 \circ r_{A \times_B C})
\\
= \
&\bigl( \beta_A \circ (\pr_1,\pr_3), \beta_C \circ (\pr_2,\pr_3) \bigr) \circ (\pr_1, \pr_2, 1 \circ r_{A \times_B C})
\\
= \
&\bigl( \beta_A \circ (\pr_1,1 \circ r_{A \times_B C}), \beta_C \circ (\pr_2,1 \circ r_{A \times_B C}) \bigr)
\\
= \
&\bigl( \beta_A \circ (\pr_1,1 \circ r_{A} \circ \pr_1), \beta_C \circ (\pr_2,1 \circ r_{C} \circ \pr_2) \bigr)
\\
= \
&\bigl( \beta_A \circ (\id_A,1 \circ r_{A}) \circ \pr_1, \beta_C \circ (\id_C,1 \circ r_{C}) \circ \pr_2 \bigr)
\\
= \
&(\id_A \circ \pr_1, \id_C \circ \pr_2)
\\
= \
&\id_{A \times_B C} \,,
\end{split}
\end{equation*}
where we have used the explicit formulation \eqref{eq:DiagActExp} of the diagonal action, and the unitality of the actions $\beta_A$ and $\beta_C$.

For the associativity, we have
\begin{align}
\beta_{A \times_B C} \circ (\id_{A \times_B C} \times_{G_0} m)
&=
\beta_{A \times_B C} \circ (\pr_1, \pr_2, m \circ \pr_3)
\nonumber \\
&=
\bigl( \beta_A \circ (\pr_1,\pr_3), \beta_C \circ (\pr_2,\pr_3) \bigr) \circ (\pr_1, \pr_2, m \circ \pr_3)
 \nonumber \\
&=
\bigl( \beta_A \circ (\pr_1, m \circ \pr_3), \beta_C \circ (\pr_2, m \circ \pr_3) \bigr) 
\nonumber \\
&=
\bigl( \beta_A \circ (\beta_A \times_{G_0} \id_{G_1}), \beta_C \circ (\beta_C \times_{G_0} \id_{G_1}) \bigr) \,,
\label{eq:diag1}
\end{align}
where we have used the diagonal action explicitly, and the associativity of $\beta_A$ and $\beta_C$. On the other hand, writing $G_2 \cong G_1 \times_{G_0} G_1$, we calculate
\begin{align}
&\beta_{A \times_B C} \circ (\beta_{A \times_B C} \times_{G_0} \id_{G_1})
\nonumber \\
= \
&\bigl( \beta_A \circ (\pr_1,\pr_3), \beta_C \circ (\pr_2,\pr_3) \bigr) \circ \bigl( \beta_A \circ (\pr_1,\pr_3), \beta_C \circ (\pr_2,\pr_3), \pr_4 \bigr)
\nonumber \\
= \
&\bigl( \beta_A \circ (\beta_A \circ (\pr_1,\pr_3),\pr_4), \beta_C \circ (\beta_C \circ (\pr_2,\pr_3),\pr_4) \bigr)
\nonumber \\
= \
&\bigl( \beta_A \circ (\beta_A \times_{G_0} \id_{G_1}), \beta_C \circ (\beta_C \times_{G_0} \id_{G_1}) \bigr) \,.
\label{eq:diag2}
\end{align}
Comparing Equations \eqref{eq:diag1} and \eqref{eq:diag2}, we get that
\begin{equation*}
\beta_{A \times_B C} \circ (\id_{A \times_B C} \times_{G_0} m) = \beta_{A \times_B C} \circ (\beta_{A \times_B C} \times_{G_0} \id_{G_1}) \,.
\end{equation*}
This shows that $\beta_{A \times_B C}$ is associative.

To show that the target fibers are taken to the source fibers, we have that
\begin{equation*}
\begin{split}
r_{A \times_B C} \circ \beta_{A \times_B C}
&=
r_A \circ \pr_1 \circ \bigl( \beta_A \circ (\pr_1,\pr_3), \beta_C \circ (\pr_2,\pr_3) \bigr)
\\
&=
r_A \circ \beta_A \circ (\pr_1,\pr_3)
\\
&=
s \circ \pr_2 \circ (\pr_1,\pr_3)
\\
&=
s \circ \pr_3 \,,
\end{split}
\end{equation*}
where we use that $\beta_A$ is a groupoid action and so takes the target fibers to the source fibers of $G$. We conclude that $\beta_{A \times_B C}$ is a right $G$-action.
%
%Consider the following diagram:
%\begin{equation*}
%\begin{tikzcd}
%A \times_{B} C \times_{G_0} G_1
%\ar[rrr, "{\beta_{A \times_{\!B} C}}"] 
%\ar[dd, "\pr_3"'] 
%\ar[dr, "{(\pr_1,\pr_3)}"]
%&[-2.5em] &[1em] &
%A \times_B C
%\ar[dd, "r_{A \times_{\!B} C}"] 
%\ar[dl, "\pr_1"']
%\\
%&
%A \times_{G_0} G_1
%\ar[dl, "\pr_2"'] 
%\ar[r, "{\beta_A}"] 
%&
%A \ar[dr, "r_A"]
%&
%\\
%G_1 \ar[rrr, "{s}"']
%& & &
%G_0
%\end{tikzcd}    
%\end{equation*}
%The upper trapezoid commutes by the explicit formulation \eqref{eq:DiagActExp}. The lower trapezoid commutes since $\beta_A$ is a groupoid action and so takes the target fibers to the source fibers. The right triangle commutes by definition and the left triangle commutes trivially. We conclude that the outer rectangle is commutative. Hence, $\beta_{A \times_B C}$ takes the target fibers to the source fibers of $G$.
\end{proof}

\begin{Proposition}
\label{prop:PullGBun}
Let $A \rightarrow B \leftarrow C$ be $G$-equivariant bundle morphisms such that the pullbacks~\eqref{eq:DiagActPullback} exist for all $k \geq 0$. Then the pullback $A \times_B C$ in $\calC$ with the diagonal $G$-action is the pullback in the category $\BunG$ of $G$-bundles and equivariant bundle morphisms.
\end{Proposition}
\begin{proof}
Let the $G$-equivariant bundle morphisms be denoted by $\phi: A \to B$, and $\psi: C \to B$. Consider the pullback diagram
\begin{equation}
\label{diag:PullSqC}
\begin{tikzcd}
A \times_B C \ar[r, "\pr_2"]
\ar[d, "\pr_1"']
&
C \ar[d, "\psi"]
\\
A \ar[r, "\phi"']
&
B
\end{tikzcd}
\end{equation}
in $\catC$. The diagonal action~\eqref{eq:DiagActExp} satisfies
\begin{equation*}
\begin{split}
\pr_1 \circ \beta_{A \times_B C} 
&=
\pr_1 \circ \bigl( \beta_A \circ (\pr_1,\pr_3), 
\beta_C \circ (\pr_2,\pr_3) \bigr)
\\
&= \beta_A \circ (\pr_1,\pr_3)
\\
&= \beta_A \circ (\pr_1 \times_{G_0} \id_{G_1}) 
\,,
\end{split}
\end{equation*}
which shows that $\pr_1$ is $G$-equivariant. In an analogous way, we can show that $\pr_2$ is $G$-equivariant. This shows that the pullback square \eqref{diag:PullSqC} is a commutative square in $\BunG$.

Let $r_E: E \to G_0$ be another $G$-bundle with right groupoid action $\beta_E$; let $f: E \to A$ and $g: E \to C$ be $G$-equivariant bundle morphisms such that $\phi \circ f = \psi \circ g$; let $(f,g): E \to A \times_B C$ be the unique morphism in $\catC$ given by the universal property of the pullback in $\catC$. It remains to show that $(f,g)$ is a $G$-equivariant bundle morphism. The bundle projections satisfy
\begin{equation*}
\begin{split}
  r_{A \times_B C} \circ (f,g) 
  &= r_A \circ \pr_1 \circ (f,g) 
   = r_A \circ f
  \\
  &= r_E 
  \,,    
\end{split}
\end{equation*}
where in the last step we have used that $f$ is a morphism of bundles. This shows that $(f,g)$ is a morphism of bundles. The actions satisfy
\begin{equation*}
\begin{split}
  \beta_{A \times_B C} \circ 
  \bigl( (f,g) \times_{G_0} \id_{G_1} \bigr)
  &=
  \bigl( \beta_A \circ (\pr_1,\pr_3), 
         \beta_C \circ (\pr_2,\pr_3) \bigr) \circ 
  (f \circ \pr_1, g \circ \pr_1, \pr_2)
  \\
  &=
  \bigl( \beta_A \circ (f \circ \pr_1, \pr_2), 
         \beta_C \circ (g \circ \pr_1, \pr_2) 
  \bigr)
\\
&=
(f \circ \beta_E, g \circ \beta_E)
\\
&=
(f,g) \circ \beta_E \,,
\end{split}
\end{equation*}
where we have used the explicit form \eqref{eq:DiagActExp} of the diagonal action and that $f$ and $g$ are $G$-equivariant. This shows that $(f,g)$ is $G$-equivariant.
\end{proof}

\begin{Example}
Let $A$ and $C$ be $G$-bundles and $B$ the terminal $G$-bundle $G_0 \stackrel{\id}{\to} G_0$ from Example~\ref{ex:G0}. Since $G_0$ is the terminal object in $\BunG$ we have unique morphisms of $G$-bundles $A \to G_0 \leftarrow C$. The pullback is the fiber product $A \times_{G_0} C$ equipped with the diagonal action.
\end{Example}

\subsection{Differentiability}
\label{sec:DifferentiabilityGrpdBun}

\begin{Definition}
\label{def:DiffBunG}
Let $G$ be a groupoid in a tangent category $\calC$. A $G$-bundle $r:E \to G_0$ in $\calC$ will be called \textbf{differentiable} if $G$ is differentiable, if the pullbacks in the diagrams
\begin{equation}
\label{eq:DiffBunG01}
\begin{tikzcd}[column sep=small]
T^n\!E \times_{T^n G_0} G_k
\ar[d] \ar[r]
\arrow[dr, phantom, "\lrcorner", very near start] 
&
G_k
\ar[d, "\Zero{n}_{G_k}"]
\\
T^n\!E \times_{T^n G_0} T^n G_k
\ar[d] \ar[r]
\arrow[dr, phantom, "\lrcorner", very near start] 
&
T^n G_k
\ar[d, "T^n t_k"]
\\
T^n\!E
\ar[r, "T^n r"']
&
T^n G_0    
\end{tikzcd}
\quad
\begin{tikzcd}[column sep=small]
T_m E \times_{T_m G_0} G_k
\ar[d] \ar[r]
\arrow[dr, phantom, "\lrcorner", very near start] 
&
G_k
\ar[d, "{\ZeroDiag{m}{G_0} \, \circ \, t_k}"]
\\
T_m E
\ar[r, "T_m r"']
&
T_m G_0    
\end{tikzcd}
\end{equation}
exist, and if the natural morphism
\begin{equation}
\label{eq:TFinPullE}
  \nu_{n,k}:
  T^n(E \times_{G_0} G_k) 
  \longrightarrow T^n\!E \times_{T^n G_0} T^n G_k
\end{equation}
is an isomorphism for all $n \geq 1$, $m \geq 2$ and $k \geq 0$.
\end{Definition}

\begin{Example}
\label{ex:G1diff}
For a differentiable groupoid $G$, the $G$-bundle $E=G_1 \xrightarrow{s} G_0$ of Example~\ref{ex:G1} is differentiable.
\end{Example}

The rest of this section is devoted to motivating the axioms in Definition~\ref{def:DiffBunG} and stating their consequences.

\subsubsection{The higher vertical tangent bundle}

The condition that~\eqref{eq:TFinPullE} is defined and is an isomorphism means that $T^n$ applied to the nerve of the action groupoid $E \rtimes G \rightrightarrows E$ (Remark~\ref{rmk:ActGrpdNerve}) is the nerve of an action groupoid $T^n\!E \rtimes T^n G \rightrightarrows T^n\!E$. More explicitly, $T^n r: T^n\!E \to T^n G_0$ with the action
\begin{equation}
\label{eq:betaTE}
  \beta_{T^n\!E}: 
  T^n\!E \times_{T^n G_0} T^n G_1 \xrightarrow[\cong]{~\nu_{n,1}^{-1}~} T^n(E \times_{G_0} G_1) 
  \xrightarrow{~T^n \beta_E~} T^n\!E
\end{equation}
is a $T^n G$-bundle. This action can be restricted along the zero section $\Zero{n}_{G_1}: G_1 \to T^n G_1$ to a $G$-action as follows.

Applying the pasting lemma to the left diagram of~\eqref{eq:DiffBunG01}, using Equation~\eqref{eq:lkzeroRel1}, and using the pasting lemma again, we obtain the natural isomorphisms
\begin{equation}
\label{eq:DiffBunG02}
\begin{split}
  (T^n\!E \times_{T^n G_0} T^n G_k) 
  \times_{T^n G_k}^{\pr_2, \, \Zero{n}_{G_k}} G_k
  &\cong
  T^n\!E \times_{T^n G_0}^{T^n r, \, T^n t_k \, \circ \, \Zero{n}_{G_k}} G_k
  \\
  &=
  T^n\!E \times_{T^n G_0}^{T^n r, \, \Zero{n}_{G_0} \circ \, t_k} G_k
  \\
  &\cong
  (T^n\!E \times_{T^n G_0}^{T^n r, \, \Zero{n}_{G_0}} 
  G_0 ) \times_{G_0}^{\pr_2, \, t_k} G_k
  \\
  &=
  \VTan{n}E \times_{G_0}^{\pr_2, \, t_k} G_k
  \,,
\end{split}
\end{equation}
where
\begin{equation}
\label{eq:DefVTanE}
  \VTan{n}E 
  := T^n\!E \times_{T^n G_0}^{T^n r, \, \Zero{n}_{G_0}} G_0
  = \ker T^n r
\end{equation}
is the higher \textbf{vertical tangent bundle} for $0 \leq n \leq 2$. For $n=0$, $\VTan{0}E \cong E$. For $n = 1$, we will write $V\! E \equiv \VTan{1}E$. Consider the following diagram:
\begin{equation}
\label{eq:RightActVE}
\begin{tikzcd}
T^n\!E \times^{T^n r,T^n t}_{T^n G_0} T^n G_1
\ar[r, "\pr_2"] \ar[d, "\beta_{T^n\!E}"']
&
T^n G_1 
\ar[d, "T^n s"]
&
G_1 
\ar[l, "\Zero{n}_{G_1}"']
\ar[d, "s"]
\\
T^n\!E
\ar[r, "T^n r"']
&
T^n G_0
&
G_0
\ar[l, "\Zero{n}_{G_0}"]
\end{tikzcd}
\end{equation}
The left square commutes because $\beta_{T^n\!E}$ is a groupoid action. The right square commutes due to the naturality of the iterated zero section. The diagram induces a morphism from the limit of the top row to the limit of the bottom row. By precomposing it with the isomorphism~\eqref{eq:DiffBunG02}, we obtain a morphism 
\begin{equation}
\label{eq:betaVTanE}
  \beta_{\VTan{n}E}:
  \VTan{n}E \times_{G_0} G_1
  \longrightarrow
  \VTan{n}E  
\end{equation}
that equips $\VTan{n}E \to G_0$ with a right $G$-action (Proposition \ref{prop:VE}).

\subsubsection{Fiber products of the vertical tangent bundle}

Since, by assumption, the tangent functor commutes with the nerve of the action groupoid and since limits commute with limits, so do its fiber products,
\begin{equation*}
  T_m (E \times_{G_0} G_k)
  \cong
  T_m E \times_{T_m G_0} T_m G_k
  \,.
\end{equation*}
This implies that $T_m r: T_m E \to T_m G_0$ is a $T_m G$-bundle with the action
\begin{equation}
\label{eq:betaTmE}
  \beta_{T_m E}: 
  T_m E \times_{T_m G_0} T_m G_1 \xrightarrow{~\cong~} T_m (E \times_{G_0} G_1) 
  \xrightarrow{~T_m \beta_E~} T_m E
  \,.
\end{equation}
The $T_m G$-action can be restricted along the diagonal zero section
\begin{equation*}
  \ZeroDiag{m}{G_1}: G_1 
  \longrightarrow T_m G_1   
\end{equation*}
to a $G$-action as follows.

The pullback on the right side of~\eqref{eq:DiffBunG01} for $m=2$ is naturally isomorphic to
\begin{equation}
\label{eq:T2EV2E}
\begin{split}
%T_2 E \times_{T_2 G_0} G_k
%&\cong
%T_2 E 
%\times_{T_2 G_0}^{T_2 r, T_2 t_k \circ (0_{G_k}, 0_{G_k})} G_k
%\\
%&\cong
T_2 E 
\times_{T_2 G_0}^{T_2 r, \, 0_{2,G_0} \, \circ \, t_k} G_k
&\cong 
T_2 E \times_{T_2 G_0}^{T_2 r, \, 0_{2,G_0}} G_0 \times_{G_0}^{\id, \, t_k} G_k 
\\
&\cong
(TE \times_E TE) \times_{TG_0 \times_{G_0} TG_0} (G_0 \times_{G_0} G_0) \times_{G_0} G_k 
\\
&\cong
(TE \times_{TG_0} G_0) \times_{E \times_{G_0} G_0}
(TE \times_{TG_0} G_0)
\times_{G_0} G_k 
\\
&\cong
(V\!E \times_E V\!E) \times_{G_0} G_k 
\\
&\cong
V_2 E \times_{G_0} G_k 
\,,
\end{split}
\end{equation}
where we have used %Equation~\eqref{eq:lkzeroRel2}, 
the pasting lemma, the definition of $T_2$, that pullbacks commute with pullbacks, the definition of the vertical tangent bundle, and finally the notation
\begin{equation}
\label{eq:DefVmE}
 V_m E := \underbrace{V\!E \times_E \ldots \times_E V\!E}_{\text{$m$ factors}} 
\end{equation}
for the fiber products of the vertical tangent bundle. Here, the morphism $V\!E \to E$ is given by the composition
\begin{equation}
\label{eq:NewPi}
V\!E = TE \times_{TG_0} G_0 \xrightarrow{~\pr_1~} TE \xrightarrow{~\pi_E~} E \,.
\end{equation}
It is straightforward to generalize the isomorphism \eqref{eq:T2EV2E} to an isomorphism
\begin{equation}
\label{eq:DiffBunG03}
  T_m E \times_{T_m G_0} G_k
  \cong
  V_m E \times_{G_0} G_k
\end{equation}
for all $m > 2$ and $k \geq 0$.
% Do we need this:
% \begin{equation*}
% (T_2 E \times_{T_2 G_0}^{T_2 r, T_2 t_k} T_2 G_k) 
% \times_{T_2 G_k}^{\pr_2, (0_{G_k}, 0_{G_k})} G_k
% \cong
% T_2 E 
% \times_{T_2 G_0}^{T_2 r, T_2 t_k \circ (0_{G_k}, 0_{G_k})} G_k
% \end{equation*}
Consider the diagram:
\begin{equation}
\label{eq:DiffBunG04}
\begin{tikzcd}[column sep=3em]
T_m E \times^{T_m r,T_m t}_{T_m G_0} T_m G_1
\ar[r, "\pr_2"] \ar[d, "\beta_{T_m E}"']
&
T_m G_1 
\ar[d, "T_m s"]
&
G_1 
\ar[l, "\ZeroDiag{m}{G_1}"']
\ar[d, "s"]
\\
T_m E
\ar[r, "T_m r"']
&
T_m G_0
&
G_0
\ar[l, "\ZeroDiag{m}{G_0}"]
\end{tikzcd}
\end{equation}
where the left square commutes because $\beta_{T_mE}$ is a groupoid action and the right square commutes due to the naturality of the diagonal zero section. The diagram
induces a morphism from the limit of the top row to the limit of the bottom row. By the pasting lemma and the isomorphism \eqref{eq:DiffBunG03} for $k=1$, the limit of the top row is isomorphic to
\begin{equation*}
T_mE \times_{T_mG_0} T_m G_1 \times_{T_mG_1} G_1 \cong V_mE \times_{G_0} G_1 \,. 
\end{equation*}
On the other hand, by the isomorphism~\eqref{eq:DiffBunG03} for $k=0$, the limit of the bottom row is isomorphic to
\begin{equation*}
T_mE \times_{T_mG_0} G_0 \cong V_m E \times_{G_0} G_0 \cong V_mE \,.
\end{equation*}
Thus, we obtain a morphism 
\begin{equation}
\label{eq:betaVmE}
  \beta_{V_m E}: 
  V_m E \times_{G_0} G_1 
  \longrightarrow
  V_m E
\end{equation}
that equips the bundle $V_m E \to G_0$ with a right $G$-action (Proposition \ref{prop:FiberProdVE}).

\subsubsection{Summary of the purpose of the axioms of \texorpdfstring{Definition~\ref{def:DiffBunG}}{the Definition}}

\begin{enumerate}

\item The isomorphism~\eqref{eq:TFinPullE} is needed so that by differentiating the action of a $G$-bundle $E \to G_0$ we obtain a $T^n G$-action on $T^n\!E \to T^n G_0$ and a $T_m G$-action on $T_m E \to T_m G_0$.

\item The existence of the pullback $T^n\!E \times_{T^n G_0} G_k$ is needed so that the $T^n G$-action on $T^n\!E$ restricts to a $G$-action on the higher vertical tangent bundle $\VTan{n}E \to G_0$.

\item The existence of the pullback $T_m E \times_{T_m G_0} G_k$ is needed so that the $T_m G$-action on $T_m E$ restricts to a $G$-action on the fiber product $V_m E \to G_0$ of the vertical tangent bundle.

\end{enumerate}
%We still have to prove that the $G$-action morphisms obtained by restriction are associative and unital.
\section{Higher vertical tangent functors}
\label{sec:HigherVertTanFun}

In order to achieve our ultimate goal of proving that the Lie bracket of invariant vector fields is invariant (Section~\ref{sec:LieBracketInv}), we will prove that each step of the construction of the Lie bracket of vector fields is invariant. As explained in Section~\ref{sec:LieBracket}, this involves all the natural transformations of the tangent category. 

We will start by showing that the assignment to a differentiable $G$-bundle its higher vertical tangent bundle (and the fiber products of its vertical tangent bundle) is functorial. This is explained in Section~\ref{sec:FunctorialityVertFUn}. Then, we prove a fundamental technical lemma, which states that natural transformations between different powers and fiber products of $T$ can be vertically restricted, under the assumption that the zero section is preserved. This is accomplished in Section~\ref{sec:VertRestNatTransf}. 

Using this powerful tool, we prove that the tangent structure and its $R$-module structure admit vertical restrictions in Sections~\ref{sec:VertRestTan} and~\ref{sec:VertRestModStr} respectively. Moreover, Section~\ref{sec:VertProlong} provides a method of vertical prolongation of $G$-equivariant bundles morphisms. As a consequence, we show that the vertically restricted vertical lift is a pointwise pullback. These results will play a crucial role in showing that each step of the Lie bracket construction is invariant (Theorem~\ref{thm:InvariantBracket}).

\subsection{Functoriality}
\label{sec:FunctorialityVertFUn}

Let $G$ be a groupoid object in a category $\catC$. The goal of this section is to show that the assignments $E \mapsto \VTan{n}E$ and $E \mapsto V_mE$ define functors from differentiable $G$-bundles to $G$-bundles. 

\begin{Proposition}
\label{prop:VE}
If a $G$-bundle $E \to G_0$ is differentiable, then the higher vertical tangent bundle $\VTan{n}E \to G_0$ defined in~\eqref{eq:DefVTanE} together with the action~\eqref{eq:betaVTanE} is a $G$-bundle.
\end{Proposition}
\begin{proof}
Consider the pullback diagram that defines the higher vertical tangent bundle,
\begin{equation}
\label{eq:VnE}
\begin{tikzcd}
\makebox[0pt][r]{${\VTan{n}E} = \ $}
T^n\!E \times_{T^n G_0} G_0
\ar[r] \ar[d]
&
T^n\!E
\ar[d, "T^nr"]
\\
G_0 \ar[r ,"\Zero{n}_{G_0}"']
&
T^n G_0
\end{tikzcd}
\end{equation}
Let us denote the top horizontal arrow, which is the projection onto the first factor, by
\begin{equation*}
  i_{\VTan{n}E}:  \VTan{n} E \longrightarrow 
  T^n\!E
  \,.
\end{equation*}
Since $\Zero{n}_{G_0}$ is a split monomorphism, it is a fortiori a strong monomorphism. Since strong monomorphisms are preserved by pullbacks, $i_{\VTan{n}E}$ is a strong monomorphism. Moreover, the iterated zero sections $\Zero{n}_{G_k}: G_k \to T^n G_k$ define a morphism $G \to T^n G$ of groupoids.

By definition of the groupoid actions on $T^n\!E$ and $\VTan{n}E$, the diagram
\begin{equation}
\label{eq:VE02}
\begin{tikzcd}[column sep=6em]
\VTan{n}E \times_{G_0} G_1
\ar[r, "i_{\VTan{n}E}  \, \times_{\Zero{n}_{G_0}} \, \Zero{n}_{G_1}"]
\ar[d, "\beta_{\VTan{n}E}"']
&
T^n\!E \times_{T^n G_0} T^n G_1
\ar[d, "\beta_{T^n\!E}"]
\\
\VTan{n}E
\ar[r, "i_{\VTan{n}E}"']
&
T^n\! E
\end{tikzcd}
\end{equation}
commutes. It now follows from Lemma~\ref{lem:ActionRestrict}, that $\beta_{\VTan{n}E}$ is a right $G$-action.
\end{proof}

\begin{Proposition}
\label{prop:FiberProdVE}
If a $G$-bundle $E \to G_0$ is differentiable, then the fiber products of the vertical tangent bundle $V_m E \to G_0$ defined in~\eqref{eq:DefVmE} together with the action~\eqref{eq:betaVmE} is a $G$-bundle.
\end{Proposition}
\begin{proof}
The proof is analogous to the proof of Proposition~\ref{prop:VE}.
\end{proof}

\begin{Remark}
The right $G$-action~\eqref{eq:betaVmE} is precisely the diagonal action~\eqref{eq:DiagAct}.
\end{Remark}

\begin{Proposition}
\label{prop:MorphismsV}
Let $E \to G_0$ and $F \to G_0$ be $G$-bundles and $\Phi:E \to F$ a $G$-equivariant bundle morphism. If $E$ and $F$ are differentiable, then the morphism
\begin{equation}
\label{eq:Vphi}
  \VTan{n}E 
  = T^n\! E \times_{T^n G_0} {G_0}
  \xrightarrow{~T^n \Phi \times_{T^n G_0} \id_{G_0}~}
  T^n\!F \times_{T^n G_0} G_0
  = \VTan{n}F
  \,,
\end{equation}
which we will denote by $\VTan{n}\,\Phi$, is a $G$-equivariant bundle morphism.
\end{Proposition}

\begin{proof}
Since by definition $\VTan{n}\,\Phi$ is the identity on $G_0$, it is a morphism of bundles over $G_0$. For the equivariance, we consider the following diagram:
\begin{equation}
\label{diag:VnGeq}
\begin{tikzcd}[row sep=3em]
{\VTan{n}E} \times_{G_0} G_1
\ar[rrr, "{\VTan{n}\,\Phi \times_{G_0} \id}"] 
\ar[ddd, "\beta_{{\VTan{n}E}}"'] 
\ar[dr, "{i_{{\VTan{n}E}} \, \times_{\Zero{n}_{G_0}} \, \Zero{n}_{G_1}}" near end]
&[-4em] &[+3em] &[-4em]
\VTan{n}F \times_{G_0} G_1
\ar[ddd, "\beta_{\VTan{n}F}"] 
\ar[dl, "{i_{\VTan{n}F} \, \times_{\Zero{n}_{G_0}} \, \Zero{n}_{G_1}}"' near end]
\\
&
T^n\!E \times_{T^n G_0} T^n G_1
\ar[d, "\beta_{T^n\!E}"']
\ar[r, "{T^n \Phi \times_{T^n G_0} \id}"] 
&
T^n\!F \times_{T^n G_0} T^n G_1 \ar[d, "\beta_{T^n\!F}"]
&
\\
&
T^n\!E \ar[r, "{T^n \Phi}"'] 
&
T^n\!F
&
\\
{\VTan{n}E} \ar[rrr, "{\VTan{n}\, \Phi}"']
\ar[ur, "i_{{\VTan{n}E}}"'] 
& & &
\VTan{n}F
\ar[ul, "i_{\VTan{n}F}"]
\end{tikzcd}    
\end{equation}

Spelling out the actions $\beta_{T^n\!E}$ and $\beta_{T^nF}$ as given in \eqref{eq:betaTE}, we get that the inner square is the outer rectangle of the following diagram 
\begin{equation*}
\begin{tikzcd}[column sep=5em, row sep=2em]
T^n\!E \times_{T^nG_0} T^nG_1
\ar[r, "T^n \Phi \times_{T^nG_0} \id"] 
\ar[d, "{\nu_{n,1}^{-1}}"', "\cong"]
&
T^n\!F \times_{T^nG_0} T^nG_1
\ar[d, "{\nu_{n,1}^{-1}}", "\cong"']
\\
T^n(E \times_{G_0} G_1) 
\ar[r, "T^n (\Phi \times_{G_0} \id)"] 
\ar[d, "T^n\beta_E"']
&
T^n(F \times_{G_0} G_1)
\ar[d, "T^n\beta_F"]
\\
T^n\!E \ar[r,"T^n\Phi"']
&
T^nF
\end{tikzcd}
\end{equation*}
which commutes since $\Phi$ is $G$-equivariant and $T^n$ is a functor.

The left and right trapezoids are both the commutative diagram~\eqref{eq:VE02}. The bottom trapezoid commutes since $i_{{\VTan{n}E}}$ and $i_{{\VTan{n}F}}$ are projections onto the first factor and since $\VTan{n}\,\Phi$ is, by definition, $T^n\Phi$ on the first factor. The commutativity of the top trapezoid follows from the commutativity of the bottom trapezoid and the  functoriality of the pullback along the zero sections. Explicitly,
\begin{equation}
\label{eq:ExplFunc}
\begin{split}
\bigl( T^n\Phi \times_{T^nG_0} \id \bigr) \circ \bigl( i_{{\VTan{n}E}} \times_{\Zero{n}_{G_0}} \Zero{n}_{G_1} \bigr)
&=
\bigl( T^n\Phi \circ i_{{\VTan{n}E}} \bigr) \times_{\Zero{n}_{G_0}} \bigl( \id \circ \Zero{n}_{G_1} \bigr)
\\
&=
\bigl( i_{{\VTan{n}F}} \circ \VTan{n}\,\Phi \bigr) \times_{\Zero{n}_{G_0}} \bigl( \Zero{n}_{G_1} \circ \id \bigr)
\\
&=
\bigl( i_{{\VTan{n}F}} \times_{\Zero{n}_{G_0}} \Zero{n}_{G_1} \bigr) \circ \bigl( \VTan{n}\,\Phi \times_{G_0} \id \bigr) \,.
\end{split}
\end{equation}
The morphism $i_{\VTan{n}F}$ is a monomorphism, as was shown in the proof of Proposition~\ref{prop:VE}. It follows from Lemma~\ref{lem:InnerOuterSquares} that the outer square commutes, which is the condition of equivariance.
\end{proof}

\begin{Proposition}
\label{prop:MorphismsVFiberprod}
Let $E \to G_0$ and $F \to G_0$ be $G$-bundles and $\Phi:E \to F$ a $G$-equivariant bundle morphism. If $E$ and $F$ are differentiable, then the morphism
\begin{equation}
\label{eq:Vphi2}
  V_m \Phi := \underbrace{V\Phi \times_\Phi \ldots \times_\Phi V\Phi}_{\text{$m$ factors}}
  : V_m E \longrightarrow V_m F
  \,,
\end{equation}
is a $G$-equivariant bundle morphism, where we write $V\Phi \equiv \VTan{1}\,\Phi$.
\end{Proposition}

\begin{proof}
Firstly, consider the following diagram:
\begin{equation*}
\begin{tikzcd}
V\!E
\ar[r, "i_{V\!E}"]
\ar[d, "V\Phi"']
&
TE
\ar[r, "\pi_E"]
\ar[d, "T\Phi"]
&
E
\ar[d, "\Phi"]
\\
V\!F
\ar[r, "i_{V\!F}"']
&
TF
\ar[r, "\pi_F"']
&
F
\end{tikzcd}
\end{equation*}
The right square commutes by the naturality of $\pi$. The left square is the commutative bottom trapezoid of Diagram \eqref{diag:VnGeq} with $i_{V\!E} \equiv i_{\VTan{1}E}$. Using \eqref{eq:NewPi}, we conclude that the map \eqref{eq:Vphi2} is well-defined. The rest of the proof is analogous to the proof of Proposition~\ref{prop:MorphismsV}.
\end{proof}

\begin{Remark}
\label{rem:VEnotDiff0}
Let $\BunGdiff$ denote the full subcategory of differentiable $G$-bundles in $\BunG$. Propositions~\ref{prop:MorphismsV} and~\ref{prop:MorphismsVFiberprod} show that the higher vertical tangent bundle and the fiber product of the vertical tangent bundle are functors 
\begin{equation*}
  \VTan{n}, V_m: \BunGdiff \longrightarrow \BunG
  \,.
\end{equation*}
For $n=0$, the functor $\VTan{0}$, which forgets the differentiability axioms, will be denoted by $1$.
\end{Remark}

\begin{Remark}
\label{rem:VEnotDiff}    
If we apply the vertical tangent functor twice and if the category $\calC$ has all pullbacks, we obtain an object 
\begin{equation*}
  V^2 E 
  \cong V (TE \times_{TG_0} G_0)
  \cong T(TE \times_{TG_0} G_0) \times_{TG_0} G_0
  \,.
\end{equation*}
In general, the limit on the right hand side does not exist. This means that $V\!E$ is generally not differentiable. Only if we assume the existence of this limit and that the tangent functor commutes with the pullback defining $V\!E$, we obtain the isomorphism
\begin{equation*}
  V^2 E 
  \cong T^2 E \times_{T^2 G_0} \times TG_0 \times_{T G_0} G_0
  \cong V^{[2]} E
  \,.
\end{equation*}
However, we will \emph{not} make the assumptions for this isomorphism to exist.
As it turns out, it is not needed for the main results of this paper. In all relevant aspects, the operators $\VTan{n}$ behave like the powers of the vertical tangent functor, as we will show in the following sections.
\end{Remark}

\subsection{Vertical restriction of natural transformations}
\label{sec:VertRestNatTransf}

In this section, we present a powerful technical tool which will be essential in showing that the tangent structure can be vertically restricted.

\begin{Lemma}
\label{lem:BigTechLemma}
Let $\alpha: T^n \to T^m$ be a natural transformation for some $n,m \geq 0$; let $G$ be a differentiable groupoid. If
\begin{equation*}
    \alpha_{G_0} \circ \Zero{n}_{G_0} = \Zero{m}_{G_0} 
    \quad \text{and} \quad
    \alpha_{G_1} \circ \Zero{n}_{G_1} = \Zero{m}_{G_1} \,,
%\label{eq:BigTechLemma01}
\end{equation*}
then there is a unique natural transformation $\alpha':\VTan{n} \to \VTan{m}$ of functors $\BunGdiff \to \BunG$ such that
\begin{equation}
\label{diag:newalpha0}
\begin{tikzcd}
\VTan{n}E
\ar[r, "\alpha'_{E}"]
\ar[d, "i_{\VTan{n}E}"']
&
\VTan{m}E
\ar[d, "i_{\VTan{m}E}"]
\\
T^n\!E
\ar[r, "\alpha_{E}"']
&
T^m\! E
\end{tikzcd}
\end{equation}
commutes for all differentiable $G$-bundles $E$.
\end{Lemma}

\begin{proof}
Let $r:E \to G_0$ be a differentiable $G$-bundle. Consider the following diagram:
\begin{equation*}
%\label{diag:newalpha}
\begin{tikzcd}%[column sep=4.3em, row sep=2em]
T^n\!E
\ar[r, "T^nr"] 
\ar[d, "\alpha_E"']
&
T^nG_0 
\ar[d, "\alpha_{G_0}"]
&
G_0
\ar[l, "\Zero{n}_{G_0}"']
\ar[d, "\id_{G_0}"]
\\
T^m\!E \ar[r, "T^mr"']
&
T^mG_0
&
G_0 
\ar[l, "\Zero{m}_{G_0}"]
\end{tikzcd}
\end{equation*}
The left square commutes by the naturality of $\alpha$. The right square commutes by assumption. Hence, the diagram induces a unique map from the limit of the top row to the limit of the bottom row
\begin{equation*}
  \alpha'_E 
  := \alpha_E \times_{\alpha_{G_0}} \id_{G_0}: 
  \VTan{n}E \longrightarrow \VTan{m}E \,,
\end{equation*}
such that Diagram~\eqref{diag:newalpha0} and the diagram
\begin{equation*}
\begin{tikzcd}[column sep=1em]
\VTan{n}E \ar[dr] \ar[rr, "{\alpha'_E}"] 
& & 
\VTan{m}E \ar[dl]
\\
& G_0 &
\end{tikzcd}
\end{equation*}
commute. This shows that $\alpha'_E$ is a bundle morphism. To show that $\alpha'_E$ is $G$-equivariant, we consider the following diagram:
\begin{equation}
\label{diag:alphaGEquiv}
\begin{tikzcd}[row sep=3em]
\VTan{n}E \times_{G_0} G_1
\ar[rrr, "{\alpha'_E \times_{G_0} \id}"] 
\ar[ddd, "\beta_{\VTan{n}E}"'] 
\ar[dr, "{i_{\VTan{n}E} \, \times_{\Zero{n}_{G_0}} \, \Zero{n}_{G_1}}" near end]
&[-4em] &[+2em] &[-4em]
\VTan{m}E \times_{G_0} G_1
\ar[ddd, "\beta_{\VTan{m}E}"] 
\ar[dl, "{i_{\VTan{m}E} \, \times_{\Zero{m}_{G_0}} \, \Zero{m}_{G_1}}"' near end]
\\
&
T^n\!E \times_{T^nG_0} T^nG_1
\ar[d, "\beta_{T^n\!E}"']
\ar[r, "{\alpha_E \times_{\alpha_{G_0}} \alpha_{G_1}}"] 
&
T^m\!E \times_{T^mG_0} T^mG_1 \ar[d, "\beta_{T^m\!E}"]
&
\\
&
T^n\!E \ar[r, "{\alpha_E}"'] 
&
T^m\!E
&
\\
\VTan{n}E \ar[rrr, "{\alpha'_E}"']
\ar[ur, "{i_{\VTan{n}E}}"'] 
& & &
\VTan{m}E
\ar[ul, "{i_{\VTan{m}E}}"]
\end{tikzcd}    
\end{equation}
The inner square is the outer rectangle of the following diagram:
\begin{equation*}
\begin{tikzcd}[column sep=5em]
T^n\!E \times_{T^nG_0} T^nG_1
\ar[r, "\alpha_E \times_{\alpha_{G_0}} \alpha_{G_1}"]
\ar[d, "{\nu_{n,1}^{-1}}"', "\cong"]
&
T^m\!E \times_{T^mG_0} T^mG_1
\ar[d, "{\nu_{m,1}^{-1}}", "\cong"']
\\
T^n(E \times_{G_0} G_1)
\ar[r, "\alpha_{E \times_{G_0} G_1}"']
\ar[d, "T^n\beta_E"']
&
T^m(E \times_{G_0} G_1)
\ar[d, "T^m\beta_E"]
\\
T^n\!E \ar[r, "\alpha_E"']
&
T^m\!E
\end{tikzcd}
\end{equation*}
The bottom square commutes by the naturality of $\alpha$, the upper square by Lem\-ma~\ref{lem:FiberProdMaps}. We conclude that the outer rectangle, and thus, the inner square of Diagram~\eqref{diag:alphaGEquiv} is commutative.

The right and left trapezoids of~\eqref{diag:alphaGEquiv} are both the commutative diagram \eqref{eq:VE02}. The lower trapezoid is Diagram~\eqref{diag:newalpha0}, which we have already shown to commute. The commutativity of the top trapezoid follows from the commutativity of the bottom trapezoid, the assumption that $\alpha_{X} \circ \Zero{n}_{X} = \Zero{m}_{X}$ for both $X = G_0$ and $X = G_1$, and from the functoriality of the pullback along the zero sections. The explicit calculation is analogous to that in \eqref{eq:ExplFunc}.

The morphism $i_{\VTan{m}E}$ is a monomorphism, as was shown in the proof of Proposition~\ref{prop:VE}. It follows from Lemma~\ref{lem:InnerOuterSquares} that the outer square commutes, which is the condition of equivariance.

Since $\alpha$ is a natural transformation and pullbacks are natural, $\alpha'_E := \alpha_E \times_{\alpha_{G_0}} \id_{G_0}$ is natural in $E \in \BunGdiff$.
\end{proof}

\begin{Lemma}
\label{lem:BigTechLemma2}
Let $\alpha: T_n \to T_m$, $\beta: T_n \to T^m$, $\gamma: T^n \to T_m$ be natural transformations for some $n,m \geq 0$; let $G$ be a differentiable groupoid. If
\begin{align*}
  \alpha_{G_0} \circ \ZeroDiag{n}{G_0}
  &= \ZeroDiag{m}{G_0}
  \,, &
  \alpha_{G_1} \circ \ZeroDiag{n}{G_1}
  &= \ZeroDiag{m}{G_1}
  \,,
  \\
  \beta_{G_0} \circ \ZeroDiag{n}{G_0}
  &= \Zero{m}_{G_0} 
  \,, &
  \beta_{G_1} \circ \ZeroDiag{n}{G_1}
  &= \Zero{m}_{G_1} 
  \,,
  \\
  \gamma_{G_0} \circ \Zero{n}_{G_0}
  &= \ZeroDiag{m}{G_0}
  \,, &
  \gamma_{G_1} \circ \Zero{n}_{G_1}
  &= \ZeroDiag{m}{G_1}
  \,,
\end{align*}
then there are unique natural transformations $\alpha':V_n \to V_m$, $\beta': V_n \to \VTan{m}$, $\gamma: \VTan{n} \to V_m$ of functors $\BunGdiff \to \BunG$ such that
\begin{equation}
\label{diag:newAlphaBetaGamma}
\begin{tikzcd}
V_n E
\ar[r, "\alpha'_{E}"]
\ar[d, "i_{V_nE}"']
&
V_m E
\ar[d, "i_{V_mE}"]
\\
T_nE
\ar[r, "\alpha_{E}"']
&
T_m E
\end{tikzcd}
\quad
\begin{tikzcd}
V_n E
\ar[r, "\beta'_{E}"]
\ar[d, "i_{V_n E}"']
&
\VTan{m}E
\ar[d, "i_{\VTan{m}E}"]
\\
T_n E
\ar[r, "\beta_{E}"']
&
T^m\!E
\end{tikzcd}
\quad
\begin{tikzcd}
\VTan{n} E
\ar[r, "\gamma'_{E}"]
\ar[d, "i_{\VTan{n}E}"']
&
V_m E
\ar[d, "i_{V_m E}"]
\\
T^n\!E
\ar[r, "\gamma_{E}"']
&
T_mE
\end{tikzcd}
\end{equation}
commute for all differentiable $G$-bundles $E$.
\end{Lemma}

\begin{proof}
The proof is analogous to the proof of Lemma~\ref{lem:BigTechLemma}.
\end{proof}

% \begin{Lemma}
% \label{lem:BigTechLemma3}
% Let $\alpha: T_m \to T^n$ be a natural transformation for some $n,m \geq 0$; let $G$ be a differentiable groupoid. If 
% \begin{equation*}
%   \alpha_{X} \circ \underbrace{(0_X, \ldots, 0_X)}_{\text{$m$ factors}} 
%   = \Zero{n}_X  
% \end{equation*}
% for $X =  G_0$ and $X =  G_1$, then there is a unique natural transformation $\alpha':V_m \to \VTan{n}$ of functors $\BunGdiff \to \BunG$ such that
% \begin{equation}
% \label{diag:newalpha0}
% \begin{tikzcd}
% V_m E
% \ar[r, "\alpha'_{E}"]
% \ar[d, "i_{V_mE}"']
% &
% \VTan{n}E
% \ar[d, "i_{\VTan{n}E}"]
% \\
% T_m E
% \ar[r, "\alpha_{E}"']
% &
% T^n \! E
% \end{tikzcd}
% \end{equation}
% commutes for all differentiable $G$-bundles $E$.
% \end{Lemma}

% \begin{proof}
% The proof is analogous to the proof of Lemma~\ref{lem:BigTechLemma}.
% \end{proof}

\begin{Lemma}
\label{lem:BigTechLemma03}
Let $\alpha : T^n \to T^m$ and $\beta: T^m \to T^l$ be natural transformations for some $n, m, l \geq 0$ and $\gamma := \beta \circ \alpha: T^n \to T^l$ their composition; let $G$ be a differentiable groupoid. If
\begin{align*}
  \alpha_{G_0} \circ \Zero{n}_{G_0}
  &= \Zero{m}_{G_0}
  \,, &
  \alpha_{G_1} \circ \Zero{n}_{G_1}
  &= \Zero{m}_{G_1}
  \,,
  \\
  \beta_{G_0} \circ \Zero{m}_{G_0}
  &= \Zero{l}_{G_0}
  \,, &
  \beta_{G_1} \circ \Zero{m}_{G_1}
  &= \Zero{l}_{G_1}
  \,,
\end{align*}
then 
\begin{equation*}
  \gamma_{G_0} \circ \Zero{n}_{G_0} = \Zero{l}_{G_0}
  \quad \text{and} \quad
  \gamma_{G_1} \circ \Zero{n}_{G_1} = \Zero{l}_{G_1}
  \,,
\end{equation*}
and the vertical restrictions $\alpha': \VTan{n} \to \VTan{m}$, $\beta': \VTan{m} \to \VTan{l}$, $\gamma': \VTan{n} \to \VTan{l}$ from Lemma~\ref{lem:BigTechLemma} satisfy $\gamma' = \beta' \circ \alpha'$. The statement also holds, mutatis mutandis, for natural transformations of type $T_n \to T_m$, $T_n \to T^m$, and $T^n \to T_m$.
\end{Lemma}

\begin{proof}
For $X=G_0$ and $X=G_1$, we have
\begin{equation*}
  \gamma_X \circ \Zero{n}_X 
  = \beta_X \circ \alpha_X \circ \Zero{n}_X
  = \beta_X \circ \Zero{m}_X
  = \Zero{l}_X
  \,,
\end{equation*}
which implies that $\gamma$ has a vertical restriction. By Lemma~\ref{lem:BigTechLemma}, we have the commutative diagram
\begin{equation*}
\begin{tikzcd}
\VTan{n}E
\ar[r, "\alpha'_{E}"]
\ar[d, "i_{\VTan{n}E}"']
&
\VTan{m}E
\ar[r, "\beta'_{E}"]
\ar[d, "i_{\VTan{m}E}"]
&
\VTan{l}E
\ar[d, "i_{\VTan{l}E}"]
\\
T^n\!E
\ar[r, "\alpha_{E}"']
\ar[rr, "\gamma_E"', bend right]
&
T^m\!E
\ar[r, "\beta_{E}"']
&
T^l\! E
\end{tikzcd}
\end{equation*}
This shows that $\beta'_E \circ \alpha'_E$ is the vertical restriction of $\gamma_E$. Since the vertical restriction is unique, it follows that $\beta'_E \circ \alpha'_E = \gamma'_E$ for all $E \in \BunGdiff$. The proof of the statements for natural transformations of type $T_n \to T_m$, $T_n \to T^m$, and $T^n \to T_m$ is analogous.
\end{proof}

\subsection{Vertical prolongation of equivariant bundle morphisms}
\label{sec:VertProlong}

Recall that $\VTan{k}:\BunGdiff \to \BunG$ is a functor (Remark~\ref{rem:VEnotDiff0}). Given differentiable $G$-bundles $E$ and $F$, a $G$-equivariant bundle morphism $\Phi:E \to \VTan{n}F$ can in general not be extended to $\VTan{k} \, \Phi:\VTan{k}E \to \VTan{n+k}F$. The reason is that $\VTan{n}F$ might not be differentiable (Remark~\ref{rem:VEnotDiff}), and so we cannot simply apply the functor $\VTan{k}$ to the morphism $\Phi$. The following proposition solves this subtlety by introducing the notion of vertical prolongation using a universal construction.

\begin{Proposition}
\label{prop:Prolongation}
Let $E$ and $F$ be differentiable $G$-bundles and $\Phi: E \to \VTan{n} F$ a $G$-equivariant bundle morphism. Then for every $k \geq 1$ there is a unique $G$-equivariant bundle morphism $\Phi^{[k]}: \VTan{k} E \to \VTan{n+k} F$, called the $k$-th \textbf{vertical prolongation} of $\Phi$, such that
\begin{equation}
\label{diag:Prolongation}
\begin{tikzcd}[column sep=3em]
\VTan{k}E
\ar[rr, "\Phi^{[k]}"]
\ar[d, "i_{\VTan{k}E}"']
&&
\VTan{n+k}F
\ar[d, "i_{\VTan{n+k}F}"]
\\
T^k\! E
\ar[r, "T^k \Phi"']
&
T^k \VTan{n} F
\ar[r, "T^k i_{\VTan{n}F}"']
&
T^{n+k} F
\end{tikzcd}
\end{equation}
commutes.
\end{Proposition}

\begin{proof}
Let
\begin{equation*}
  \Psi := 
  T^k i_{\VTan{n}F} \circ T^k \Phi \circ i_{\VTan{k}E}
\end{equation*}
denote the composition of the counterclockwise arrows of Diagram~\eqref{diag:Prolongation}. Let $r_E: E \to G_0$, $r_F: F \to G_0$, $r_{\VTan{k}E}: \VTan{k}E \to G_0$ and $r_{\VTan{k}F}: \VTan{k}F \to G_0$ denote the bundle projections. Then
\begin{equation*}
\begin{split}
  T^{n+k}r_F \circ \Psi
  &= T^{n+k}r_F \circ T^k i_{\VTan{n}F} 
  \circ T^k \Phi \circ i_{\VTan{k}E}
  \\
  &= T^k ( T^n r_F \circ i_{\VTan{n}F} 
  \circ \Phi) \circ i_{\VTan{k}E}
  \\
  &= T^k ( \Zero{n}_{G_0} \circ r_{\VTan{n}F}  
  \circ \Phi) \circ i_{\VTan{k}E}
  \\
  &= T^k ( \Zero{n}_{G_0} \circ r_E ) \circ i_{\VTan{k}E}
  \\
  &= T^k \Zero{n}_{G_0} \circ T^k r_E  \circ i_{\VTan{k}E}
  \\
  &= T^k \Zero{n}_{G_0} \circ \Zero{k}_{G_0} 
     \circ r_{\VTan{k}E}
  \\
  &= \Zero{n+k}_{G_0} \circ r_{\VTan{k}E}
  \,,
\end{split}
\end{equation*}
where we have used the pullback squares defining $\VTan{n}F$ and $\VTan{k}E$ and that $\Phi$ is a bundle morphism. It follows from the universal property of the pullback defining $\VTan{n+k}F$, that there is a unique morphism $\Phi^{[k]}$, such that
\begin{equation*}
\begin{tikzcd}
\VTan{k}E 
\ar[ddr, "\Psi"', bend right] 
\ar[rrd, "r_{\VTan{k}E}", bend left] 
\ar[dr, "\exists! \, \Phi^{[k]}", dashed]
&&
\\
&
\VTan{n+k} F
\ar[r]
\ar[d, "{i_{\VTan{n+k}F}}"]
&
G_0
\ar[d, "\Zero{n+k}_{G_0}"]
\\
&
T^{n+k}\!F
\ar[r, "T^{n+k} r_F"']
&
T^{n+k}G_0
\end{tikzcd}
\end{equation*}
commutes. In particular, we have that $\Phi^{[k]}$ is a bundle morphism. It remains to show that it is $G$-equivariant.

Let $\psi := i_{\VTan{n} F} \circ \Phi$, so that $i_{\VTan{n+k} F} \circ \Phi^{[k]} = T^k\psi \circ i_{\VTan{k} E}$. Consider the following diagram:
\begin{equation}
\label{diag:psiActions}
\begin{tikzcd}[column sep=6em]
E \times_{G_0} G_1
\ar[r, "\Phi \times_{G_0} \id_{G_1}"]
\ar[d, "\beta_E"']
&
\VTan{n}F \times_{G_0} G_1
\ar[r, "{i_{\VTan{n}F} \times_{\Zero{n}_{G_0}} \Zero{n}_{G_1}}"]
\ar[d, "\beta_{\VTan{n}F}"]
&
T^n\!F \times_{T^nG_0} T^nG_1
\ar[d, "\beta_{T^nF}"]
\\
E
\ar[r, "\Phi"']
&
\VTan{n}F
\ar[r, "i_{\VTan{n}F}"']
&
T^nF
\end{tikzcd}
\end{equation}
The left square commutes since $\Phi$ is, by assumption, $G$-equivariant. The right square is the commutative diagram \eqref{eq:VE02}. The composition of the lower horizontal arrows is $\psi$ and that of the upper horizontal arrows is
\begin{equation*}
\begin{split}
\bigl(i_{\VTan{n}F} \times_{\Zero{n}_{G_0}} \Zero{n}_{G_1}\bigr)
\circ
\bigl(\Phi \times_{G_0} \id_{G_1}\bigr)
&=
\bigl(i_{\VTan{n}F} \circ \Phi\bigr) \times_{\Zero{n}_{G_0}} \Zero{n}_{G_1}
\\
&=
\psi \times_{\Zero{n}_{G_0}} \Zero{n}_{G_1} \,,
\end{split}
\end{equation*}
by functoriality.

Applying the functor $T^k$ to Diagram \eqref{diag:psiActions} and stacking it with the commutative square obtained via the natural universal morphisms to the pullbacks, we get the following commutative diagram:

\begin{equation}
\label{diag:psiActions2}
\begin{tikzcd}[column sep=6em]
T^k \!E \times_{T^k G_0} T^k G_1
\ar[d, "\cong", "\nu_{k,1}^{-1}"']
%\ar[dd, "\beta_{T^k\! E}"', bend right=75]
\ar[r, "{T^k\psi \times_{T^k \Zero{n}_{G_0}} T^k \Zero{n}_{G_1}}"] 
&
T^{n+k} F \times_{T^{n+k} G_0} T^{n+k} G_1 
\ar[d, "\cong"', "\nu_{k,1}^{-1}"]
%\ar[dd, "\beta_{T^{n+k}\! F}", bend left=85]
\\
T^k( E \times_{G_0} G_1 )
\ar[d, "T^k \beta_E"']
\ar[r, "{T^k\bigl( \psi \times_{\Zero{n}_{G_0}} \Zero{n}_{G_1} \bigr)}"] 
&
T^k( T^n\!F \times_{T^n G_0} T^n G_1 )
\ar[d, "T^k \beta_{T^n\!F}"]
\\
T^k \!E 
\ar[r, "T^k \psi"'] 
&
T^{n+k} F
\end{tikzcd}    
\end{equation}
By definition, the composition of the left vertical arrows is the $T^kG$-action $\beta_{T^k\!E}$ and that of the right vertical arrows is the $T^{n+k}G$-action $\beta_{T^{n+k}E}$. Thus, the outer square of Diagram \eqref{diag:psiActions2} is the inner square in the following diagram:
\begin{equation*}
\begin{tikzcd}[row sep=3em]
\VTan{k} E \times_{G_0} G_1
\ar[rrr, "{\Phi^{[k]} \times_{G_0} \id}"] 
\ar[ddd, "\beta_{\VTan{k} E}"'] 
\ar[dr, "{i_{\VTan{k} E} \, \times_{\Zero{k}_{G_0}} \, \Zero{k}_{G_1}}" near end]
&[-4em] &[3em] &[-4em]
\VTan{n+k} F \times_{G_0} G_1
\ar[ddd, "\beta_{\VTan{n+k} F}"] 
\ar[dl, "{i_{\VTan{n+k} F} \, \times_{\Zero{n+k}_{G_0}} \, \Zero{n+k}_{G_1}}"' near end]
\\
&
T^k \!E \times_{T^k G_0} T^k G_1
\ar[d, "\beta_{T^k\! E}"']
\ar[r, "{T^k\psi \times_{T^k \Zero{n}_{G_0}} T^k \Zero{n}_{G_1}}"] 
&
T^{n+k} F \times_{T^{n+k} G_0} T^{n+k} G_1 
\ar[d, "\beta_{T^{n+k} F}"]
&
\\
&
T^k \!E
\ar[r, "T^k \psi"']
&
T^{n+k} F
&
\\
\VTan{k} E \ar[rrr, "\Phi^{[k]}"']
\ar[ur, "{i_{\VTan{k} E}}"'] 
& & &
\VTan{n+k} F
\ar[ul, "{i_{\VTan{n+k} F}}"]
\end{tikzcd}
\end{equation*}
The bottom and top trapezoids commute by the definition of $\Phi^{[k]}$ and the functoriality of pullbacks. The left and right trapezoids commute by the definition of the $G$-actions on the vertical bundles (Diagram \eqref{eq:VE02}). Since $i_{\VTan{n+k}F}$ is a monomorphism, it follows from Lemma~\ref{lem:InnerOuterSquares} that the outer square commutes, which shows that $\Phi^{[k]}$ is $G$-equivariant.
\end{proof}

It is a consequence of Lemmas~\ref{lem:BigTechLemma} and \ref{lem:BigTechLemma2} that vertical prolongations also exist for vertically restricted natural transformations. 

\begin{Lemma}
Let $\alpha: T^n \to T^m$ be a natural transformation for some $n,m \geq 0$; let $G$ be a differentiable groupoid. Assume that
\begin{equation*}
%\label{eq:BigTechLemma01}
  \alpha_X \circ \Zero{n}_X = \Zero{m}_X
\end{equation*}
for $X =  G_0$ and $X =  G_1$, and let $\alpha':\VTan{n} \to \VTan{m}$ be the unique natural transformation from Lemma~\ref{lem:BigTechLemma}. Then, there exists a unique natural transformation
\begin{equation*}
\alpha'^{[k]}: \VTan{n+k} \longrightarrow \VTan{m+k} \,,
\end{equation*}
called the $k$-th \textbf{vertical prolongation} of $\alpha'$, such that
\begin{equation}
\label{diag:newalphaProlong}
\begin{tikzcd}
\VTan{n+k}E
\ar[r, "\alpha_{E}'^{[k]}"]
\ar[d, "i_{\VTan{n+k}E}"']
&
\VTan{m+k}E
\ar[d, "i_{\VTan{m+k}E}"]
\\
T^{n+k}\!E
\ar[r, "T^k \alpha_{E}"']
&
T^{m+k}\! E
\end{tikzcd}
\end{equation}
commutes for all differentiable $G$-bundles $E$.
\end{Lemma}

\begin{proof}
The natural transformation $T^k \alpha: T^{n+k} \to T^{m+k}$ is given componentwise by $(T^k \alpha)_X = T^k \alpha_X$ for all $X \in \catC$. It satisfies
\begin{equation*}
\begin{split}
(T^k \alpha)_X \circ \Zero{n+k}_X 
&=
T^k \alpha_X \circ \Zero{n+k}_X
\\
&=
T^k \alpha_X \circ T^k \Zero{n}_X
\\
&=
T^k (\alpha_X \circ \Zero{n}_X)
\\
&=
T^k \Zero{m}_X
\\
&=
\Zero{m+k}_X
\end{split}
\end{equation*}
for $X=G_0$ and $X=G_1$. It then follows from Lemma~\ref{lem:BigTechLemma} that there is a unique natural transformation $\alpha'^{[k]}:= (T^k \alpha)': \VTan{n+k} \longrightarrow \VTan{m+k}$ such that Diagram~\eqref{diag:newalphaProlong} commutes.
\end{proof}

\subsection{Vertical restriction of the tangent structure}
\label{sec:VertRestTan}

The next proposition is a crucial technical result: We show that the natural transformations of the tangent structure on $\catC$ can be restricted to the (higher) vertical tangent bundles and their fiber products. It is a consequence of Lemma~\ref{lem:BigTechLemma}.

\begin{Proposition}
\label{prop:TangStrRest}
Let $G$ be a differentiable groupoid in a tangent category $\catC$. Let $\Id: \BunGdiff \to \BunG$ denote the inclusion. There are natural transformations $\pi':V \to \Id$, $0':\Id \to V$, $+':V_2 \to V$, $\tau':\VTan{2} \to \VTan{2}$, $\lambda':V \to \VTan{2}$ and $\lambda'_2:V_2 \to \VTan{2}$ of functors $\BunGdiff \to \BunG$ such that the following diagrams
\begin{equation}
\begin{gathered}
\label{diag:newNatTransf}
\begin{tikzcd}[column sep=0.7em]
V\!E \ar[dr, "i_{V\!E}"'] \ar[rr, "{\pi'_E}"] 
& & 
E
\\
& TE \ar[ur, "\pi_E"'] &
\end{tikzcd}
\quad
\begin{tikzcd}[column sep=0.7em]
E \ar[dr, "0_E"'] \ar[rr, "{0'_E}"] 
& & 
V\!E \ar[dl, "i_{V\!E}"]
\\
& TE &
\end{tikzcd}
\quad
\begin{tikzcd}
V_2E
\ar[r, "+'_{E}"]
\ar[d, "i_{V_2E}"']
&
V\!E
\ar[d, "i_{V\!E}"]
\\
T_2E
\ar[r, "+_{E}"']
&
T\!E
\end{tikzcd}
\\
\begin{tikzcd}
\VTan{2}E
\ar[r, "\tau'_{E}"]
\ar[d, "i_{\VTan{2}E}"']
&
\VTan{2}E
\ar[d, "i_{\VTan{2}E}"]
\\
T^2E
\ar[r, "\tau_{E}"']
&
T^2E
\end{tikzcd}
\quad
\begin{tikzcd}
V\!E
\ar[r, "\lambda'_{E}"]
\ar[d, "i_{V\!E}"']
&
\VTan{2}E
\ar[d, "i_{\VTan{2}E}"]
\\
TE
\ar[r, "\lambda_{E}"']
&
T^2E
\end{tikzcd}
\quad
\begin{tikzcd}
V_2E
\ar[r, "\lambda'_{2,E}"]
\ar[d, "i_{V_2E}"']
&
\VTan{2}E
\ar[d, "i_{\VTan{2}E}"]
\\
T_2E
\ar[r, "\lambda_{2,E}"']
&
T^2E
\end{tikzcd}
\end{gathered}
\end{equation}
commute in $\calC$ for all differentiable $G$-bundles $E$.
\end{Proposition}

\begin{proof}
We will show that the assumption $\alpha_{X} \circ \Zero{n}_{X} = \Zero{m}_{X}$ 
of Lemma~\ref{lem:BigTechLemma} is satisfied for the natural transformations $\alpha=\pi, 0, \tau, \lambda$ (with the right choice of $n$ and $m$) of the tangent structure for all $X \in \calC$, so in particular for $X = G_0$ and $X = G_1$. Similarly, we will show that the assumptions $\alpha_X \circ 0_{2,X} = 0_X$ and $\beta_X \circ 0_{2,X} = \Zero{2}_{X}$ of Lemma~\ref{lem:BigTechLemma2} hold for $\alpha=+:T_2 \to T$ and $\beta=\lambda_2:T_2 \to T^2$.

Since $0$ is a section of $\pi$, we have that $\pi_X \circ 0_X = \id_X = 0_X^{[0]}$. For the zero section, we have
\begin{equation*}
  0_X \circ \Zero{0}_X = 0_X \circ \id_X = 0_{X} 
  \,.
\end{equation*}
By Definition~\ref{def:TangentStructure} of tangent structures, $\tau$ is a morphism of bundles of abelian groups, so it maps the zero section to the zero section,
\begin{equation}
\label{eq:TauZero}
\tau_{X} \circ \Zero{2}_{X} = \tau_{X} \circ T0_{X} \circ 0_{X} = 0_{TX} \circ 0_{X} = \Zero{2}_{X} \,.
\end{equation}
Similarly, $\lambda$ is a morphism of bundles of abelian groups over $0: 1 \to T$, so that it maps the zero section to the zero section,
\begin{equation}
\label{eq:LambdaZero}
  \lambda_{X} \circ 0_{X} = 0_{TX} \circ 0_{X} = \Zero{2}_{X} \,.
\end{equation}
The natural transformation $\alpha = +$ is an abelian group structure, so that 
\begin{equation*}
+_X \circ (0_X,0_X) = 0_X \,.    
\end{equation*}
By Definition~\eqref{eq:VertLiftExt} of the vertical lift, Equation~\eqref{eq:LambdaZero}, that the addition of zero and zero is zero, and Equation~\eqref{eq:TauZero}, we get that
\begin{equation*}
\begin{split}
\lambda_{2,X} \circ (0_X,0_X) 
&= \tau_X \circ +_{TX} \circ (T0_X \times_{0_X} \lambda_X) \circ (0_X,0_X)
\\
&= \tau_X \circ +_{TX} \circ (T0_X \circ 0_X, \lambda_X \circ 0_X)
\\
&= \tau_X \circ +_{TX} \circ (\Zero{2}_X, \Zero{2}_X)
\\
&= \tau_X \circ \Zero{2}_X
\\
&=
\Zero{2}_X \,.
\end{split}
\end{equation*}
The proof now follows from Lemma~\ref{lem:BigTechLemma} for $\alpha = \pi, 0, \tau, \lambda$ and from Lemma~\ref{lem:BigTechLemma2} for $\alpha = +$ and $\beta=\lambda_2$.
\end{proof}

\begin{Remark}
By Definition~\eqref{eq:VertLiftExt}, the extension $\lambda_2$ of the vertical lift is given by the composition $\lambda_2 := \tau \circ (+T) \circ (T0 \times_0 \lambda)$. It is natural to ask why its vertical restriction $\lambda_2'$ is not defined similarly by composing the corresponding vertical restrictions.
%\begin{equation*}
%  \lambda'_{2,E}:
%  V\!E \times_E V\!E \xrightarrow{~0_E'^{[1]} \times_{0'_E} \lambda'_E~}
%  \VTan{2}E \times_{V\!E} \VTan{2}E
%  \xrightarrow{~+'_{V\!E}~}
%  \VTan{2}E
%  \xrightarrow{~\tau'_E~}
%  \VTan{2}E
%  \,,
%\end{equation*}
%for all $E \in \BunGdiff$. Here, $0_E'^{[1]}: V\!E \to \VTan{2}E$ is the vertical prolongation of $0'_E:E \to V\!E$.
The reason is that $V\!E$ is generally not differentiable as described in Remark~\ref{rem:VEnotDiff}. Thus, the pullback $V_2 V\!E = V^2E \times_{V\!E} V^2E$ does not generally exist and so the addition $+'_{V\!E}:V_2 V\!E \to V^2E$ is not well-defined. Proposition~\ref{prop:TangStrRest} shows that $\lambda_2'$ is still well-defined as the vertical restriction of $\lambda_2$.
\end{Remark}

The next proposition shows that $\lambda_2'$ is still a vertical lift, where the notion of vertical prolongation from Proposition~\ref{prop:Prolongation} becomes crucial.

\begin{Proposition}
\label{prop:NewVertLiftKernel2}
The diagram
\begin{equation}
\label{diag:newVertLiftKernel2}
\begin{tikzcd}
V_2 \ar[r, "\lambda_{2}'"] \ar[d, "\pi' \circ \, \pr_1"'] & \VTan{2} \ar[d, "{\pi'^{[1]}}"]
\\
1 \ar[r, "{0'}"'] & V
\end{tikzcd}    
\end{equation}
is a pointwise pullback in $\BunG$, where $\pi'^{[1]}$ is the 1st vertical prolongation of $\pi'$. 
\end{Proposition}

\begin{proof}
Let $r_E: E \to G_0$ be a differentiable $G$-bundle. Consider the following diagram:
\begin{equation}
\label{diag:NewVertLiftPullback0}
\begin{tikzcd}
V_2E
\ar[rrr, "{\lambda'_{2,E}}"] 
\ar[ddd, "\pi'_E \circ \, \pr_1"'] 
\ar[dr, "{i_{{V_2E}}}"]
%&[-4em] &[+3em] &[-4em]
& & &
\VTan{2}E
\ar[ddd, "\pi_E'^{[1]}"] 
\ar[dl, "{i_{\VTan{2}E}}"']
\\
&
T_2E
\ar[d, "\pi_E \circ \, \pr_1"']
\ar[r, "{\lambda_{2,E}}"] 
&
T^2E 
\ar[d, "T\pi_E"]
&
\\
&
E \ar[r, "{0_E}"'] 
&
TE
&
\\
E \ar[rrr, "{0'_E}"']
\ar[ur, "\id"'] 
& & &
V\!E
\ar[ul, "i_{V\!E}"]
\end{tikzcd}    
\end{equation}
The commutativity of the upper, lower and left trapezoids follows from the corresponding commutative diagrams in \eqref{diag:newNatTransf}. The right trapezoid is the commutative diagram \eqref{diag:newalphaProlong} with $\alpha=\pi$ and $k=1$. The inner square is the commutative square~\eqref{diag:VertLift} evaluated at $E \in \catC$. Since $i_{V\!E}$ is a monomorphism, we get, by Lemma~\ref{lem:InnerOuterSquares}, that the outer square is a commutative square in $\BunG$.

Now, consider the following diagram:
\begin{equation*}
\label{diag:IterPullVertLift2}
\begin{tikzcd}%[column sep=5em, row sep=4em]
E \ar[r, "{0_E}"] \ar[d, "r_E"'] 
& 
TE \ar[d, "{Tr_E}"] 
& 
T^2E \ar[d, "T^2r_E"] \ar[l, "{T \pi_E}"']
\\
G_0 \ar[r, "{0_{G_0}}"] 
& 
T{G_0} 
& 
T^2{G_0} \ar[l, "{T\pi_{G_0}}"'] 
\\
G_0 \ar[r, "\id"'] \ar[u, "\id"] 
& 
G_0 \ar[u, "{0_{G_0}}"'] 
& 
G_0 \ar[u, "\Zero{2}_{G_0}"'] \ar[l, "\id"]
\end{tikzcd}
\end{equation*}
The upper left square commutes by the naturality of $0$ and the upper right square commutes by the naturality of $\pi$ and the functoriality of $T$. The lower right square commutes since $T\pi \circ T0 \circ 0 = T(\pi \circ 0) \circ 0 = T1 \circ 0 = 0$. The commutativity of the lower left square is trivial. It follows from the commutativity of pullbacks that
\begin{align}
&(E \times_{TE} T^2E) \times_{G_0 \times_{TG_0} T^2G_0} (G_0 \times_{G_0} G_0) \label{eq:IsoVertLiftleft}
\\ 
\cong \
&(E \times_{G_0} G_0) \times_{TE \times_{TG_0} G_0} (T^2E \times_{T^2G_0} G_0) \,.
\label{eq:IsoVertLiftright}
\end{align}

Using the fact that Diagram \eqref{diag:VertLift} is a pointwise pullback, we know that 
\begin{equation*}
\begin{split}
E \times_{TE} T^2E &\cong T_2E
\\
G_0 \times_{TG_0} T^2G_0 &\cong T_2{G_0} \,.
\end{split}
\end{equation*}
We conclude that \eqref{eq:IsoVertLiftleft} is isomorphic to
\begin{equation*}
\begin{split}
T_2E \times_{T_2G_0} G_0
&\cong
V_2E \times_{G_0} G_0
\\
&\cong V_2E \,,
\end{split}
\end{equation*}
where we have used the isomorphism in \eqref{eq:T2EV2E} for $k=0$ and cancellation of pullbacks along the identity. On the other hand, \eqref{eq:IsoVertLiftright} is isomorphic to $E \times_{V\!E} \VTan{2}E$. The isomorphism of the pullbacks~\eqref{eq:IsoVertLiftleft} and~\eqref{eq:IsoVertLiftright} hence induces the isomorphism 
\begin{equation*}
V_2E \cong E \times_{V\!E} \VTan{2}E    
\end{equation*}
in $\catC$.
Lastly, it follows from Proposition \ref{prop:PullGBun} that this is the pullback in the category $\BunG$.
\end{proof}

\subsection{Vertical restriction of the module structure}

\label{sec:VertRestModStr}

Let $E$ be a differentiable $G$-bundle and $R$ a ring object of the tangent category $\catC$. Then,
\begin{equation*}
r_{V\!E} \circ \pr_2: R \times V\!E \to G_0
\end{equation*}
is a $G$-bundle with right $G$-action given by
\begin{equation*}
\id_R \times \beta_{V\!E}: R \times V\!E \times_{G_0} G_1 
\longrightarrow 
R \times V\!E \,.
\end{equation*}
Let $\kappa:R \times T \to T$ be an $R$-module structure on $\pi:T \to \Id$. It is a natural transformation with components $\kappa_X:R \times TX \to TX$ for all $X \in \catC$ (Section~\ref{sec:CatTanScal}). We will show that $\kappa$ restricts to a $G$-equivariant $R$-module structure on $\pi':V \to \Id$.

\begin{Proposition}
\label{prop:ModuleStrRest}
Let $G$ be a differentiable groupoid in a cartesian tangent category $\catC$ with an $R$-module structure $\kappa: R \times T \to T$. Then there is a natural transformation $\kappa': R \times V \to V$ of functors $\BunGdiff \to \BunG$, such that
\begin{equation}
\label{diag:NewOldKappa}
\begin{tikzcd}
R \times V\! E
\ar[r, "\kappa'_E"]
\ar[d, "\id_R \times i_{V\!E}"']
&
V\!E
\ar[d, "i_{V\!E}"]
\\
R \times TE
\ar[r, "\kappa_E"']
&
TE
\end{tikzcd}
\end{equation}
commutes in $\calC$ for all differentiable $G$-bundles $E$.
\end{Proposition}

\begin{proof}
Let $r:E \to G_0$ be a differentiable $G$-bundle. Consider the following diagram:
\begin{equation}
\label{diag:newkappa}
\begin{tikzcd}[column sep=3em]
R \times TE
\ar[r, "\id_R \times Tr"] 
\ar[d, "\kappa_E"']
&
R \times TG_0 
\ar[d, "\kappa_{G_0}"]
&
R \times G_0
\ar[l, "\id_R \times 0_{G_0}"']
\ar[d, "\pr_2"]
\\
TE \ar[r, "Tr"']
&
TG_0
&
G_0 
\ar[l, "0_{G_0}"]
\end{tikzcd}
\end{equation}
The left square commutes by the naturality of $\kappa$. The right square commutes since the scalar multiplication is linear in $T$, which implies that it maps the zero section to the zero section. By the commutativity of pullbacks and products, we have the following isomorphism for the limit of the top row,
\begin{equation*}
\begin{split}
  \phi: (R \times TE) \times_{R \times TG_0} (R \times G_0) &\xrightarrow{~\cong~} (R \times_R R) \times (TE \times_{TG_0} G_0)
  \\
  &\xrightarrow{~\cong~} R \times V\!E 
  \,.    
\end{split}
\end{equation*}
Diagram~\eqref{diag:newkappa} induces a unique map 
%from the limit of the top row to the limit of the bottom row
\begin{equation}
\label{eq:RModuleVE}
  \kappa'_E 
  := \big(\kappa_E \times_{\kappa_{G_0}} \pr_2\big) \circ \phi^{-1}: R \times V\!E \longrightarrow V\!E \,,
\end{equation}
such that Diagram \eqref{diag:NewOldKappa} and the diagram
\begin{equation*}
\begin{tikzcd}
R \times V\!E \ar[r, "{\kappa'_E}"] \ar[d, "{\id_R \times r_{V\!E}}"'] 
& 
V\!E \ar[d, "r_{V\!E}"]
\\
R \times G_0 \ar[r, "\pr_2"'] & G_0
\end{tikzcd}
\end{equation*}
commute. This shows that $\kappa'_E$ is a morphism of bundles over $G_0$. 

Next, we show that $\kappa$ is right $TG$-equivariant. From Lemma~\ref{lem:FiberProdMaps} with $F(\Empty) = R \times T(\Empty)$ and $G=T$ we obtain the following commutative square:
\begin{equation*}
\begin{tikzcd}[column sep=5em]
R \times T(E \times_{G_0} G_1)
\ar[r, "\kappa_{E \times_{G_0} G_1}"]
\ar[d]
&
T(E \times_{G_0} G_1)
\ar[d]
\\
(R \times TE) \times_{R \times TG_0} (R \times TG_1)
\ar[r, "\kappa_E \times_{\kappa_{G_0}} \kappa_{G_1}"']
&
TE \times_{TG_0} TG_1
\end{tikzcd}
\end{equation*}
For the domain of the bottom horizontal arrow, we have the isomorphism
\begin{equation*}
\begin{split}
  \psi: (R \times TE) \times_{R \times TG_0} (R \times TG_1) &\xrightarrow{~\cong~} (R \times_R R) \times (TE \times_{TG_0} TG_1)
  \\
  &\xrightarrow{~\cong~} R \times TE \times_{TG_0} TG_1 
  \,,
\end{split}
\end{equation*}
so that we obtain the commutative square
\begin{equation}
\begin{tikzcd}[column sep=7em]
\label{diag:ModuleStr02}
R \times T(E \times_{G_0} G_1)
\ar[r, "\kappa_{E \times_{G_0} G_1}"]
\ar[d, "\id_R \times \nu_{1,1}"']
&
T(E \times_{G_0} G_1)
\ar[d, "\nu_{1,1}"]
\\
R \times TE \times_{TG_0} TG_1
\ar[r, "(\kappa_E \times_{\kappa_{G_0}} \kappa_{G_1}) \, \circ \, \psi^{-1}"']
&
TE \times_{TG_0} TG_1
\end{tikzcd}
\end{equation}
Expressing $\psi$ explicitly in terms of projections, we can write the morphism represented by the bottom arrow as
\begin{equation*}
\begin{split}
 \tilde{\kappa} 
 &:= (\kappa_E \times_{\kappa_{G_0}} \kappa_{G_1}) \circ \psi^{-1}
 \\
 &=
 \bigl(\kappa_E \circ (\pr_1,\pr_2), 
 \kappa_{G_1} \circ (\pr_1,\pr_3)\bigr)   
 \,.
\end{split}
\end{equation*}
It can be viewed as the diagonal $R$-module structure or as the $R$-module structure of the pullback of $R$-modules. By the assumption that $E$ is differentiable, the vertical arrows of~\eqref{diag:ModuleStr02} are isomorphisms, so that we obtain the commutative diagram
\begin{equation*}
\begin{tikzcd}[column sep=3em]
R \times TE \times_{TG_0} TG_1
\ar[r, "\tilde{\kappa}"]
\ar[d, "\id_R \times \nu_{1,1}^{-1}"']
\ar[dd, "\id_R \times \beta_{TE}"', bend right=75]
&
TE \times_{TG_0} TG_1
\ar[d, "\nu_{1,1}^{-1}"]
\ar[dd, "\beta_{TE}", bend left=75]
\\
R \times T(E \times_{G_0} G_1)
\ar[r, "\kappa_{E \times_{G_0} G_1}"']
\ar[d, "\id_R \times T\beta_E"']
&
T(E \times_{G_0} G_1)
\ar[d, "T\beta_E"]
\\
R \times TE \ar[r, "\kappa_E"']
&
TE
\end{tikzcd}
\end{equation*}
which shows that the $R$-module structure and the $TG$-action on $TE$ commute. This diagram is the inner square of the following diagram:
\begin{equation}
\label{diag:kappaGEquiv}
\begin{tikzcd}%[row sep=3em]
R \times V\!E \times_{G_0} G_1
\ar[rrr, "{\kappa'_E \times_{G_0} \id_{G_1}}"] 
\ar[ddd, "\id_R \times \beta_{V\!E}"'] 
\ar[dr, "{\id_R \times i_{V\!E} \, \times_{0_{G_0}} 0_{G_1}}" near end]
&[-5em] &[+0em] &[-4em]
V\!E \times_{G_0} G_1
\ar[ddd, "\beta_{V\!E}"] 
\ar[dl, "{i_{V\!E} \, \times_{0_{G_0}} 0_{G_1}}"' near end]
\\
&
R \times TE \times_{TG_0} TG_1
\ar[d, "\id_R \times \beta_{TE}"']
%\ar[r]
\ar[r, "\tilde{\kappa}"]
&
TE \times_{TG_0} TG_1 \ar[d, "\beta_{TE}"]
&
\\
&
R \times TE \ar[r, "{\kappa_E}"'] 
&
TE
&
\\
R \times V\!E \ar[rrr, "{\kappa'_E}"']
\ar[ur, "{\id_R \times i_{V\!E}}"'] 
& & &
V\!E
\ar[ul, "{i_{V\!E}}"]
\end{tikzcd}    
\end{equation}
The right and left trapezoids of~\eqref{diag:kappaGEquiv} commute by the commutativity of Diagram~\eqref{eq:VE02}. The lower trapezoid is Diagram~\eqref{diag:NewOldKappa}, which we have already shown to commute. The top trapezoid commutes since 
\begin{equation*}
\begin{split}
&\tilde{\kappa} \circ 
(\id_R \times i_{V\!E} \times_{0_{G_0}} 0_{G_1}) 
\\
={}
&\big(\kappa_E \circ (\pr_1,\pr_2), \, \kappa_{G_1} \circ (\pr_1,\pr_3)\big) \circ (\id_R \times i_{V\!E} \times_{0_{G_0}} 0_{G_1}) 
\\
={}
%&\big(\kappa_E \circ (\pr_1, i_{V\!E} \circ \pr_2), \kappa_{G_1} \circ (\pr_1, 0_{G_1} \circ \pr_3) \big)
%\\
%={}
&\bigl(\kappa_E \circ (\id_R \times i_{V\!E}) \circ (\pr_1,\pr_2), \kappa_{G_1} \circ (\id_R \times 0_{G_1}) \circ (\pr_1,\pr_3) \bigr)
\\
={}
&\bigl((i_{V\!E} \circ \kappa'_E) \circ (\pr_1, \pr_2), (0_{G_1} \circ \pr_2) \circ (\pr_1, \pr_3) \bigr)
\\
={}
&\bigl((i_{V\!E} \circ \kappa'_E) \circ (\pr_1, \pr_2), (0_{G_1} \circ \pr_3) \bigr)
\\
={}
&(i_{V\!E} \circ \kappa'_E) \times_{0_{G_0}} 0_{G_1}
\\
={}
&\bigl(i_{V\!E} \times_{0_{G_0}} 0_{G_1} \bigr) \circ \bigl(\kappa'_E \times_{G_0} \id_{G_1} \bigr)
\,,
\end{split}
\end{equation*}
where we have used the commutativity of the bottom trapezoid, that $\kappa$ preserves the zero section, and the functoriality of pullbacks. Since the morphism $i_{V\!E}$ is a monomorphism, it follows from Lemma~\ref{lem:InnerOuterSquares} that the outer square commutes, which shows that $\kappa'$ is $G$-equivariant. Moreover, since $\kappa$ is a natural transformation and pullbacks are natural, $\kappa'_E := \big(\kappa_E \times_{\kappa_{G_0}} \pr_2\big) \circ \phi^{-1}$ is natural in $E \in \BunGdiff$.
\end{proof}

\begin{Remark}
The $R$-module structure $\kappa'_E$ of $V\!E \to E$ defined in~\eqref{eq:RModuleVE} is the module structure described in the proof of Proposition~\ref{prop:KerOfMorphR} for $\phi=Tr$ and $\ker \phi = V\!E$.
\end{Remark}

\begin{Corollary}
\label{cor:newproj}
The natural transformation $\pi':V \to \Id$ with neutral element $0':\Id \to V$ and addition $+':V_2 \to V$ is a bundle of abelian groups over $\Id$. If the tangent category $\catC$ has an $R$-module structure, then $\pi':V \to \Id$ is a bundle of $R$-modules.
\end{Corollary}

\begin{proof}
Let $r: E \to G_0$ be a differentiable $G$-bundle. Then, $Tr: TE \to TG_0$ is a morphism of bundles of abelian groups by the naturality of the tangent structure. It follows from Proposition~\ref{prop:KerOfMorph}~(i) that its kernel $V\!E = TE \times_{TG_0}^{Tr, \, 0_{G_0}} G_0$, together with the vertical restriction of the bundle projection
\begin{equation*}
\pi'_{E}: V\!E \xrightarrow{~i_{V\!E}~} TE \xrightarrow{~\pi_{E}~} E \,,
\end{equation*}
has the structure of a bundle of abelian groups. As explained in the proof of Proposition~\ref{prop:KerOfMorph}, the group structure is given by the vertical restrictions $+'_E$ of the addition $+_E$ and $0'_E$ of the zero section $0_E$.

The $R$-module structure is given by the unique map $\kappa'_E$, as defined in Equation~\eqref{eq:RModuleVE} in the proof of Proposition~\ref{prop:ModuleStrRest}.
\end{proof}

The commutative diagrams of associativity, unitality, and linearity of the $R$-module structure $\kappa$ restrict to commutative diagrams at the level of the vertical tangent bundle. This can be shown by similar arguments using Lemma~\ref{lem:InnerOuterSquares}. It follows that the natural morphism $\kappa'_E: R \times V\!E \to V\!E$ is an $R$-module structure on $\pi'_E: V\!E \to E$ for all differentiable $G$-bundles $E$.
\section{The Lie bracket of invariant vector fields}
\label{sec:LieBracketInv}

The goal of this section is to state and prove one of the main results of this paper that invariant vector fields are closed under the Lie bracket. We will first introduce vertical and invariant vector fields on differentiable $G$-bundles in the context of tangent categories in Section~\ref{sec:InvVecField}. After having all the ingredients ready, we state and prove the main theorem in Section~\ref{sec:ProofVerticalInvariant}.

\subsection{Invariant vector fields on groupoid bundles}
\label{sec:InvVecField}

\begin{Definition}
\label{def:InvVecField}
Let $E$ be a differentiable $G$-bundle in a tangent category $\catC$. A vector field $v: E \to TE$ is called \textbf{vertical} if it factors through $i_{V\!E}: V\!E \to TE$, that is, if there is a (necessarily unique) morphism $v': E \to V\!E$ such that 
\begin{equation}
\label{diag:vertical}
\begin{tikzcd}
&
V\!E \ar[d, "i_{V\!E}"]
\\
E \ar[ru, "v'"]
\ar[r, "v"']
&
TE
\end{tikzcd}
\end{equation}
commutes. The vector field $v$ is called \textbf{invariant} if it is vertical and if $v'$ is $G$-equivariant.
\end{Definition}

%\begin{Example}
%It follows from Proposition~\ref{prop:TangStrRest} that the zero section $0_E: E \to TE$ is invariant with vertical lift $0'_E:E \to V\!E$.
%\end{Example}

\begin{Lemma}
\label{lem:VertLiftBunMap}
The vertical lift $v'$ of a vertical vector field $v$ is a section of $\pi'_E:V\!E \to E$.
\end{Lemma}

\begin{proof}
Using the first commutative triangle in \eqref{diag:newNatTransf}, we have that 
\begin{equation*}
\pi'_E \circ v' = \pi_E \circ i_{V\!E} \circ v' = \pi_E \circ v = \id_E \,.
\end{equation*}
Moreover, $v'$ is a bundle morphism since
\begin{equation*}
\begin{split}
r_E = r_E \circ \id_E &= r_E \circ \pi_E \circ v
\\
&=
r_E \circ \pi_E \circ i_{V\!E} \circ v'
\\
&=
\pi_{G_0} \circ Tr_E \circ i_{V\!E} \circ v'
\\
&=
\pi_{G_0} \circ 0_{G_0} \circ r_{V\!E} \circ v'
\\
&=
\id_{G_0} \circ r_{V\!E} \circ v'
\\
&=
r_{V\!E} \circ v' \,,
\end{split}
\end{equation*}
where $r_E:E \to G_0$ and $r_{V\!E}:V\!E \to G_0$ are the bundle projections. In the calculation, we have used the naturality of $\pi$, the commutative square defining the pullback $V\!E$ and that $\pi \circ 0 = 1$. 
\end{proof}

Recall from Definition~\ref{def:GBunMor} that  the bundle morphism $v'$ is $G$-equivariant if the diagram
\begin{equation*}
\begin{tikzcd}[column sep=4em]
E \times_{G_0} G_1 \ar[r, "v' \times_{G_0} \id_{G_1}"] \ar[d] 
& V\!E \times_{G_0} G_1 \ar[d]
\\
E \ar[r, "v'"'] & V\!E
\end{tikzcd}
\end{equation*}
commutes, where the vertical arrows are the $G$-actions.

\begin{Remark}
Since the groupoid bundles that we consider in this paper are all right groupoid bundles (Terminology~\ref{term:right}), we do not include the adjective \textit{right} in the definition of invariant vector fields. In $\Mfld$, these are precisely the usual right invariant vector fields (vector fields that are invariant under all right translations). As a standard reference, see for instance \cite[Equation~(8.12)]{lee} for the identity satisfied by left invariant vector fields.
\end{Remark}

\begin{Definition}
Let $E$ be a $G$-bundle with right $G$-action $\beta_E$ in a category $\catC$; let $R$ be an object of $\catC$. A morphism $f:E \to R$ in $\catC$ is called \textbf{invariant} if the diagram
\begin{equation*}
\begin{tikzcd}
E \times_{G_0} G_1 \ar[r, "\beta_E"] \ar[d, "\pr_1"'] & E \ar[d, "f"]
\\
E \ar[r, "f"'] & R
\end{tikzcd}
\end{equation*}
commutes.
\end{Definition}

We will denote by
\begin{equation}
\label{eq:InvVecModule}
  \calX(E)^G
  := \{ v \in \Gamma(E, TE) ~|~ 
    \text{$v$ is invariant}\}
\end{equation}
the set of invariant vector fields on $E$ and by
\begin{equation*}
  \catC(E,R)^G
  := \{ f \in \catC(E, R) ~|~ 
    \text{$f$ is invariant}\}
\end{equation*}
the set of invariant morphisms from $E$ to $R$. If $R$ has a ring structure, then $\catC(E,R)^G \subset \catC(E,R)$ is a subring.

\begin{Proposition}
\label{prop:SubmoduleStr}
Let $E$ be a differentiable $G$-bundle in a tangent category $\catC$ with an $R$-module structure. Then, the set $\calX(E)^G$ of invariant vector fields on $E$ is a $\catC(E,R)^G$-submodule of $\Gamma(E,TE)$.
\end{Proposition}

\begin{proof}
Let $v, w: E \to TE$ be invariant vector fields with vertical lifts $v',w': E \to V\!E$, that is, $v = i_{V\!E} \circ v'$ and $w = i_{V\!E} \circ w'$. The addition of the vector fields and the addition of their lifts are defined by
\begin{equation*}
\begin{aligned}
  v+w &= +_E \circ (v,w)
  \\
  v'+w' &= +'_E \circ (v',w')
  \,.
\end{aligned}
\end{equation*}
Consider the following diagram:
\begin{equation*}
\begin{tikzcd}[column sep=large, row sep=large]
&
V_2E
\ar[d, "i_{V_2E}"]
\ar[r, "+'_{E}"] 
&
V\!E
\ar[d, "i_{V\!E}"]
\\
E
\ar[ru, "{(v',w')}"]
\ar[r, "{(v,w)}"']
&
T_2E
\ar[r, "+_{E}"']
&
TE
\end{tikzcd}    
\end{equation*}
The left triangle commutes by Diagram~\eqref{diag:vertical}. The right square is the third commutative diagram of~\eqref{diag:newNatTransf}. The commutativity of the outer pentagon shows that $v'+w'$ is the vertical lift of $v+w$. The morphisms $v'$ and $w'$ are $G$-equivariant by assumption. The sum $+'_{E}$ is $G$-equivariant by Proposition~\ref{prop:TangStrRest}. We conclude that the composition $v'+w' = +'_E \circ (v',w')$ is $G$-equivariant, which shows that $v+w$ is invariant. 

Composing a vector field with the inverse $\iota_E: TE \to TE$ of the addition $+_E$ yields the inverse with respect to the addition of vector fields, that is, 
\begin{equation*}
\begin{split}
  v + (\iota_E \circ v) 
  &=  +_E \circ (v, \iota_E \circ v) 
  \\
  &= +_E \circ (\id_{TE} \times_E \iota_E) \circ \Delta_{TE} \circ v
  \\
  &= 0_E \circ \pi_E \circ v
  \\
  &= 0_E
  \,,
\end{split}
\end{equation*}
where $\Delta_{TE} : TE \to TE \times_E TE$ is the diagonal morphism in $\calC \Comma E$. Since $\iota_X \circ 0_X = 0_X$ for all objects $X$ in $\catC$, so in particular for $X=G_0$ and $X=G_1$, it follows from Lemma~\ref{lem:BigTechLemma} that the inverse has a $G$-equivariant vertical restriction $\iota'_E: V\!E \to V\!E$ such that the diagram
\begin{equation*}
\begin{tikzcd}
V\!E
\ar[r, "\iota'_{E}"]
\ar[d, "i_{V\!E}"']
&
V\!E
\ar[d, "i_{V\!E}"]
\\
TE
\ar[r, "\iota_{E}"']
&
TE
\end{tikzcd}
\end{equation*}
commutes. This implies that $\iota'_E \circ v'$ is a $G$-equivariant vertical lift of $\iota_E \circ v$, so that $\iota_E \circ v$ is invariant. By Proposition~\ref{prop:TangStrRest}, the zero vector field $0_{E}: E \to TE$ is invariant with vertical lift $0'_{E}$. We conclude that $\calX(E)^{G}$ is an abelian subgroup of $\Gamma(E,TE)$.

Now, let $f:E \to R$ be an invariant morphism in $\catC$. The $\catC(E,R)^G$-module structure on $\Gamma(E,TE)$ is given by
\begin{equation*}
  fv = \kappa_E \circ (f,v) \,.
\end{equation*}
Consider the following diagram:
\begin{equation*}
\begin{tikzcd}[column sep=large, row sep=large]
&
R \times V\!E
\ar[d, "\id_R \times i_{V\!E}"]
\ar[r, "\kappa'_{E}"] 
&
V\!E
\ar[d, "i_{V\!E}"]
\\
E
\ar[ru, "{(f,v')}"]
\ar[r, "{(f,v)}"']
&
R \times TE
\ar[r, "\kappa_{E}"']
&
TE
\end{tikzcd}    
\end{equation*}
where $\kappa'_{E}$ is the natural map from Proposition~\ref{prop:ModuleStrRest}. The left triangle commutes since $v$ is, by assumption, vertical. The right square is the commutative diagram~\eqref{diag:NewOldKappa}. The commutativity of the outer pentagon shows that 
\begin{equation*}
fv' := \kappa'_E \circ (f,v')    
\end{equation*}
is the vertical lift of $fv$. The morphism $\kappa'_{E}$ is $G$-equivariant by Proposition~\ref{prop:ModuleStrRest}. By assumption, the morphism $f$ is invariant and  the morphism $v'$ is $G$-equivariant. This implies that $(f,v')$ is $G$-equivariant. We conclude that $fv' = \kappa'_E \circ (f,v')$ is $G$-equivariant, which shows that $fv$ is invariant. 

We have shown that $\calX(E)^G \subset \Gamma(E, TE)$ is an abelian subgroup that is closed under the $\catC(E,R)^G$-module structure.
\end{proof}

%\begin{Proposition}
%\label{prop:InvariantBracket}
%Let $E \to G_0$ be a differentiable $G$-bundle. If two vector fields $v,w: E \to TE$ are vertical (invariant), then so is their Lie bracket $[v,w]$.
%\end{Proposition}

\subsection{Invariant vector fields are closed under the Lie bracket}
\label{sec:ProofVerticalInvariant}

\begin{Theorem}
\label{thm:InvariantBracket}
Let $E$ be a differentiable $G$-bundle in a tangent category $\catC$. Then, the set $\calX(E)^G$ of invariant vector fields on $E$ is a Lie subalgebra of $\Gamma(E,TE)$.
\end{Theorem}

In the proof of Proposition~\ref{prop:SubmoduleStr} we have shown that $\calX(E)^G \subset \Gamma(E, TE)$ is an abelian subgroup. The rest of this section is devoted to the proof of Theorem~\ref{thm:InvariantBracket}, where we show that (vertical) invariant vector fields are closed under the Lie bracket.

\subsubsection{Proof that the Lie bracket of vertical vector fields is vertical}

Let $v$ and $w$ be vertical, but not necessarily invariant vector fields with vertical lifts $v', w': E \to V\!E$. Let $\delta(v,w):  E \to T^2 E$ be the morphism defined in Equation~\eqref{eq:DeltaOrig}. In a first step, we will show that $\delta(v,w)$ factors through $i_{\VTan{2}E}: \VTan{2}E \to T^2 E$. That is, we want to show that there is a morphism $\delta'(v',w')$, such that
\begin{equation}
\label{diag:VertDelta}
\begin{tikzcd}
&
\VTan{2}E \ar[d, "i_{\VTan{2}E}"]
\\
E \ar[ru, "{\delta'(v',w')}"]
\ar[r, "{\delta(v,w)}"']
&
T^2 E
\end{tikzcd}
\end{equation}
commutes.

%$r_{V\!E} = r_E \circ \pi_E \circ i_{V\!E}: V\!E \to G_0$
Let $r_E: E \to G_0$ and $r_{V\!E}: V\!E \to G_0$ denote the bundle projections. Firstly, we have that
\begin{equation}
\label{eq:SmallFact}
\begin{split}
Tr_E \circ v
&=
Tr_E \circ i_{V\!E} \circ v'
\\
&=
0_{G_0} \circ r_{V\!E} \circ v'
\\
&=
0_{G_0} \circ r_E
\end{split}
\end{equation}
by the commutative square defining the pullback $V\!E$ and since $v'$ is a bundle map (Lemma~\ref{lem:VertLiftBunMap}). For $\delta(v,w)$ we have the relation
\begin{equation*}
\begin{split}
&T^2 r_E \circ \delta(v, w)
\\
= \ 
&T^2 r_E \circ -_{TE} \circ (Tw \circ v, \tau_E \circ Tv \circ w) \quad \text{by Equation~\eqref{eq:DeltaOrig}}
\\
= \ 
&-_{TG_0} \circ \, (T^2 r_E \times_{Tr_E} T^2 r_E) \circ (Tw \circ v, \tau_E \circ Tv \circ w) \quad \text{by the naturality of $-$}
\\
= \ 
&-_{TG_0} \circ \, (T^2 r_E \circ Tw \circ v, T^2 r_E \circ \tau_E \circ Tv \circ w) \quad \text{by Lemma~\ref{lem:MapsinPar}~(ii)}
\\
= \ 
&-_{TG_0} \circ \, (T^2 r_E \circ Tw \circ v, \tau_{G_0} \circ T^2 r_E \circ Tv \circ w) \quad \text{by the naturality of $\tau$}
\\
= \ 
&-_{TG_0} \circ \, (T0_{G_0} \circ Tr_E \circ v, \tau_{G_0} \circ T0_{G_0} \circ Tr_E \circ w) \quad \text{by Equation~\eqref{eq:SmallFact}}
\\
= \ 
&-_{TG_0} \circ \, (T0_{G_0} \circ 0_{G_0} \circ r_E, \tau_{G_0} \circ T0_{G_0} \circ 0_{G_0} \circ r_E) \quad \text{by Equation~\eqref{eq:SmallFact}}
\\
= \ 
&-_{TG_0} \circ \, (\Zero{2}_{G_0} \circ r_E, \Zero{2}_{G_0} \circ r_E) \quad \text{by Equation~\eqref{eq:TauZero}}
\\
= \ 
&\Zero{2}_{G_0} \circ r_E \,,
\end{split}
\end{equation*}
%where we have used the Definition~\eqref{eq:DeltaOrig} of $\delta(v,w)$, the naturality of the subtraction $-_X$, the naturality of $\tau_X$, Equation~\eqref{eq:SmallFact} twice, Equation~\eqref{eq:TauZero}, and finally that
since the subtraction of zero by zero is zero, as is the case for any group object. Using this relation, we conclude from the universal property of the pullback defining $\VTan{2}E$, that there is a unique morphism $\delta'(v',w')$, such that
\begin{equation*}
\begin{tikzcd}
E 
\ar[ddr, "{\delta(v,w)}"', bend right] 
\ar[rrd, "r_E", bend left] 
\ar[dr, "{\exists! \, \delta'(v',w')}", dashed]
&&
\\
&
\VTan{2} E
\ar[r, "{r_{\VTan{2}E}}"]
\ar[d, "{i_{\VTan{2}E}}"]
&
G_0
\ar[d, "\Zero{2}_{G_0}"]
\\
&
T^{2}E
\ar[r, "T^{2} r_E"']
&
T^{2}G_0
\end{tikzcd}
\end{equation*}
commutes. Consider the relation
\begin{equation*}
\begin{split}
  i_{V\!E} \circ \pi'^{[1]}_E \circ \delta'(v',w')
  &=
  T\pi_E \circ i_{\VTan{2}E} \circ \delta'(v',w')
  \\
  &=
  T\pi_E \circ \delta(v,w)
  \\
  &=
  0_E
  \\
  &= i_{V\!E} \circ 0'_E
  \,,
\end{split}
\end{equation*}
where we have used the relation defining the vertical prolongation of $\pi'_E$, the Relation~\eqref{diag:VertDelta} defining $\delta'(v',w')$, and Equation~\eqref{eq:DeltaKerOrig}. Since $i_{V\!E}$ is a monomorphism, it follows that
\begin{equation*}
  \pi'^{[1]}_E \circ \delta'(v',w')
  = 0'_E
  \,.
\end{equation*}
We can summarize the situation in the following commutative diagram:
\begin{equation}
\label{eq:VertBracketDef}
\begin{tikzcd}[column sep=3em]
E 
\ar[r, "{[v',w']}"]
\ar[drr, "{\delta'(v',w')}", bend left=55]
\ar[dr, "{\exists! \, \phi}"', dashed]
\ar[ddr, "{\id_E}"', bend right=55]
& V\!E 
&
\\
&
V_2 E
\ar[u, "\pr_2"']
\ar[dr, phantom, "\lrcorner", very near start]
\ar[r, "\lambda'_{2,E}"] 
\ar[d, "\pi'_E \, \circ \, \pr_1"']
&
\VTan{2}E 
\ar[d, "{\pi'^{[1]}_E}"]
\\
&
E 
\ar[r, "{0'_E}"']
&
V\!E
\end{tikzcd}
\end{equation}
We have shown in Proposition~\ref{prop:NewVertLiftKernel2} that the bottom right square is a pullback. The unique dashed arrow exists due to the universal property of the pullback. It remains to show that its projection onto the second factor of $V_2 E$, which we will denote by $[v',w']$, is the vertical lift of $[v,w]$. For this, we consider the following diagram:
\begin{equation}
\label{diag:vert4}
\begin{tikzcd}[row sep=3ex, column sep=2em]
E
\ar[rrrr, "{[v',w']}", bend left=25]
\ar[rrr, "{\exists!}", dashed] 
\ar[ddd, "\id"'] 
\ar[dr, "\id"']
& &[0.7em] &
V_2 E
\ar[ddd, "i_{V_2E}"] 
\ar[dl, "\lambda'_{2,E}"]
\ar[r, "\pr_2"']
&
V\!E
\ar[ddd, "i_{V\!E}"]
\\
&
E 
\ar[d, "\id"'] 
\ar[r, "{\delta'(v',w')}"] 
&
\VTan{2}E \ar[d, "i_{\VTan{2}E}"]
&&
\\
&
E \ar[r, "{\delta(v,w)}"'] &
T^2E
&&
\\
E 
\ar[rrrr, "{[v,w]}"', bend right=25]
\ar[rrr, "{\exists!}"', dashed]
\ar[ur, "\id"] 
& & &
T_2E
\ar[ul, "\lambda_{2,E}"']
\ar[r, "\pr_2"]
&
TE
\end{tikzcd}
\end{equation}
The inner square is the commutative triangle~\eqref{diag:VertDelta}. The right trapezoid is the last commutative square in \eqref{diag:newNatTransf}. The commutativity of the upper and lower trapezoids is explained in the construction of the bracket as summarized in Diagrams~\eqref{eq:BracketDef} and \eqref{eq:VertBracketDef}. The left trapezoid commutes trivially. Since monomorphisms are stable under pullbacks and the zero section is a monomorphism, it follows from Diagram~\eqref{diag:VertLift} that the map $\lambda_{2,E}$ is a monomorphism. Thus, the left outer square commutes by Lemma~\ref{lem:InnerOuterSquares}. The right square commutes by the naturality of the projection. From the commutativity of the outer rectangle it follows that 
\begin{equation}
\label{eq:BrackVert}
i_{V\!E} \circ [v',w'] = [v,w] = [i_{V\!E} \circ v', i_{V\!E} \circ w'] \,.
\end{equation}
This shows that $[v,w]$ is vertical.

\subsubsection{Proof that the Lie bracket of invariant vector fields is invariant}

Assume now that the vertical lifts $v', w': E \to V\!E$ are $G$-equivariant. Consider the following diagram:
\begin{equation}
\label{diag:InnerSqDeltaGequiv1}
\begin{tikzcd}[column sep=4.8em]
E \times_{G_0} G_1
\ar[r, "{v' \times_{G_0} \id_{G_1}}"] 
\ar[d, "\beta_E"']
&
V\!E \times_{G_0} G_1 
\ar[r, "{w'^{[1]} \times_{G_0} \id_{G_1}}"]
\ar[d, "{\beta_{V\!E}}"]
&
\VTan{2}E \times_{G_0} G_1 
\ar[d, "{\beta_{\VTan{2}E}}"]
\\
E 
\ar[r, "v'"']
\ar[d, "\id"']
&
V\!E 
\ar[r, "{w'^{[1]}}"']
\ar[d, "{i_{V\!E}}"]
&
\VTan{2}E
\ar[d, "{i_{\VTan{2}E}}"]
\\
E
\ar[r, "v"']
&
TE
\ar[r, "Tw"']
&
T^2E
\end{tikzcd}
\end{equation}
The upper left square commutes since $v'$ is $G$-equivariant. The lower left square commutes because $v'$ is the vertical lift of $v$. The upper and lower right squares commute by Proposition~\ref{prop:Prolongation}, where $w'^{[1]}:V\!E \to \VTan{2}E$ is the first vertical prolongation of $w'$. We conclude that the outer square commutes, so that we obtain
\begin{equation}
\label{eq:InnerSqDeltaGequiv1}
\begin{split}
Tw \circ v \circ \beta_E
&=
i_{\VTan{2}E} \circ \beta_{\VTan{2}E} \circ \bigl(
(w'^{[1]} \circ v') \times_{G_0} \id_{G_1} \bigr)
\\
&=
\beta_{T^2 E} \circ \bigl( i_{\VTan{2}E} \times_{\Zero{2}_{G_0}} \Zero{2}_{G_1} \bigr) \circ 
\bigl( (w'^{[1]} \circ v') \times_{G_0} \id_{G_1} \bigr)
\\
&=
\beta_{T^2 E} \circ \bigl( (i_{\VTan{2}E} \circ w'^{[1]} \circ v') \times_{\Zero{2}_{G_0}} \Zero{2}_{G_1} \bigr)
\\
&=
\beta_{T^2 E} \circ \bigl(
(Tw \circ v) \times_{\Zero{2}_{G_0}} \Zero{2}_{G_1} \bigr)
\,,
\end{split}
\end{equation}
where we have used the Diagram~\eqref{eq:VE02} defining the vertical $G$-action, the functoriality of pullbacks, and the bottom rectangle of Diagram~\eqref{diag:InnerSqDeltaGequiv1}. The diagram 
\begin{equation}
\label{diag:InnerSqDeltaGequiv2}
\begin{tikzcd}[column sep=4.5em]
T^2E \times_{T^2G_0} T^2G_1
\ar[d, "\beta_{T^2E}"']
\ar[r, "{\tau_E \times_{\tau_{G_0}} \tau_{G_1}}"] 
&
T^2E \times_{T^2G_0} T^2G_1 \ar[d, "\beta_{T^2E}"]
\\
T^2E \ar[r, "{\tau_E}"'] 
&
T^2E
\end{tikzcd}    
\end{equation}
is the inner commutative square of~\eqref{diag:alphaGEquiv} for $\alpha = \tau$. Similarly, for the subtraction $\alpha = - :T_2 \to T$, we obtain the commutative diagram
\begin{equation}
\label{diag:betadiag}
\begin{tikzcd}[column sep=1.5em]
(T^2\!E \times_{T^2G_0} T^2G_1) \times_{TE \times_{TG_0} TG_1} (T^2\!E \times_{T^2G_0} T^2G_1)
\ar[dd, "{\beta_{T^2\!E} \, \times_{\beta_{T\!E}} \, \beta_{T^2\!E}}"']
\ar[r, "\cong"]
&
T_2T\!E \times_{T_2TG_0} T_2TG_1
\ar[d, "{-_{T\!E} \times_{-_{TG_0}} -_{TG_1}}"] 
\\
&
T^2\!E \times_{T^2G_0} T^2G_1
\ar[d, "\beta_{T^2\!E}"]
\\
T_2T\!E \ar[r, "{-_{T\!E}}"'] 
&
T^2\!E
\end{tikzcd}    
\end{equation}
where the isomorphism is a simple reordering of the factors since pullbacks commute with pullbacks. Using~\eqref{eq:InnerSqDeltaGequiv1} with $v$ and $w$ swapped and~\eqref{diag:InnerSqDeltaGequiv2}, we get
\begin{equation}
\label{eq:InnerSqDeltaGequiv2}
\begin{split}
\tau_E \circ Tv \circ w \circ \beta_E 
&=
\tau_E \circ \beta_{T^2 E} \circ \bigl( (Tv \circ w) \times_{\Zero{2}_{G_0}} \Zero{2}_{G_1} \bigr)
\\
&=
\beta_{T^2E} \circ 
\bigl( \tau_E \times_{\tau_{G_0}} \tau_{G_1} \bigl) \circ \bigl( (Tv \circ w) \times_{\Zero{2}_{G_0}} \Zero{2}_{G_1} \bigr)
\\
&=
\beta_{T^2E} \circ \bigl( (\tau_E \circ Tv \circ w) \times_{\Zero{2}_{G_0}} \Zero{2}_{G_1} \bigr) 
\,,
\end{split} 
\end{equation}
where in the last step we have used that $\tau$ takes the zero section to the zero section as expressed by Equation~\eqref{eq:TauZero}.

Now we can subtract~\eqref{eq:InnerSqDeltaGequiv1} and \eqref{eq:InnerSqDeltaGequiv2}, which yields
\begin{equation*}
\begin{split}
&\delta(v,w) \circ \beta_E
\\
={}
& -_{TE} \circ \, 
  (Tw \circ v, \, \tau_E \circ Tv \circ w) \circ \beta_E \quad \text{by Equation~\eqref{eq:DeltaOrig}}
\\
={}
& -_{TE} \circ \, (Tw \circ v \circ \beta_E, \, 
      \tau_E \circ Tv \circ w \circ \beta_E) \quad \text{by Lemma~\ref{lem:MapsinPar}~(i)}
\\
={}
& -_{TE} \circ\, \bigl\{
  \beta_{T^2E} \circ \bigl((Tw \circ v) \times_{\Zero{2}_{G_0}} \Zero{2}_{G_1} \bigr), \, \beta_{T^2E} \circ \bigl((\tau_E \circ Tv \circ w) \times_{\Zero{2}_{G_0}} \Zero{2}_{G_1} \bigr) \bigr\} 
\\ & \text{by Equations~\eqref{eq:InnerSqDeltaGequiv1} and~\eqref{eq:InnerSqDeltaGequiv2}}
\\
={}
& -_{TE} \circ\, 
  (\beta_{T^2E} \times_{\beta_{TE}} \beta_{T^2E} ) \circ 
        \bigl( (Tw \circ v) \times_{\Zero{2}_{G_0}} \Zero{2}_{G_1}, \, 
  (\tau_E \circ Tv \circ w) \times_{\Zero{2}_{G_0}} \Zero{2}_{G_1} \bigr)
\\
& \text{by Lemma~\ref{lem:MapsinPar}~(ii)}
\\
={}
& \beta_{T^2E} \, \circ (-_{TE} \times_{-_{TG_0}} -_{TG_1}) \, \circ 
  \bigl( (Tw \circ v, \, \tau_E \circ Tv \circ w) \times_{\big(\Zero{2}_{G_0}, \, \Zero{2}_{G_0}\big)} (\Zero{2}_{G_1}, \Zero{2}_{G_1}) \bigr)
\\
& \text{by Diagram~\eqref{diag:betadiag}}
\\
={}
& \, \beta_{T^2E} \, \circ \, \bigl(\delta(v,w) \times_{\Zero{2}_{G_0}} \Zero{2}_{G_1} \bigr) \,,
\end{split}
\end{equation*}
where in the last step we have used Equation~\eqref{eq:DeltaOrig} once more and that the subtraction of zero by zero is zero.
%where we have used Equations~\eqref{eq:DeltaOrig}, \eqref{eq:InnerSqDeltaGequiv1}, \eqref{eq:InnerSqDeltaGequiv2}, Diagram \eqref{diag:betadiag}, and that the subtraction of zero by zero is zero. 
This relation means that the inner square of the following diagram commutes:
\begin{equation*}
\begin{tikzcd}[row sep=3em]
E \times_{G_0} G_1
\ar[rrr, "{\delta'(v',w') \times_{G_0} \id}"] 
\ar[ddd, "\beta_E"'] 
\ar[dr, "{\id}"]
&[-2em] &[3em] &[-4em]
\VTan{2} E \times_{G_0} G_1
\ar[ddd, "\beta_{\VTan{2} E}"] 
\ar[dl, "{i_{\VTan{2} E} \times_{\Zero{2}_{G_0}} \Zero{2}_{G_1}}"', near end]
\\
&
E \times_{G_0} G_1
\ar[d, "\beta_E"']
\ar[r, "{\delta(v,w) \, \times_{\Zero{2}_{G_0}} \Zero{2}_{G_1}}"] 
&
T^{2} E \times_{T^{2} G_0} T^{2} G_1 
\ar[d, "\beta_{T^{2} E}"]
&
\\
&
E
\ar[r, "{\delta(v,w)}"']
&
T^{2} E
&
\\
E \ar[rrr, "{\delta'(v',w')}"']
\ar[ur, "{\id}"'] 
& & &
\VTan{2} E
\ar[ul, "{i_{\VTan{2} E}}"]
\end{tikzcd}
\end{equation*}
The left trapezoid commutes trivially. The bottom trapezoid commutes by~\eqref{diag:VertDelta}. The top trapezoid commutes by~\eqref{diag:VertDelta} and the functoriality of pullbacks. The right trapezoid commutes by~\eqref{eq:VE02}. Since $i_{\VTan{2}E}$ is a monomorphism, it follows from Lemma~\ref{lem:InnerOuterSquares} that the outer square commutes, which shows that $\delta'(v',w')$ is $G$-equivariant.

By Proposition \ref{prop:NewVertLiftKernel2}, the pullback square in Diagram~\eqref{eq:VertBracketDef} is a pullback in the category $\BunG$ of $G$-bundles and $G$-equivariant bundle morphisms. Since $\delta'(v',w')$ is a $G$-equivariant bundle morphism, it follows by the universal property of the pullback that the morphism $\phi: E \to V_2 E$ depicted by the dashed arrow in~\eqref{eq:VertBracketDef} is $G$-equivariant. Since $\pr_2: V_2 E \to V\!E$ is also $G$-equivariant, so is the composition $[v',w'] = \pr_2 \circ \phi$. We conclude that $[v',w']$ is $G$-equivariant, which shows that $[v,w]$ is invariant.

We have shown that the abelian subgroup $\calX(E)^{G} \subset \Gamma(E, TE)$ is closed under the Lie bracket, which finishes the proof of Theorem~\ref{thm:InvariantBracket}.
\qed
\section{The abstract Lie algebroid of a differentiable groupoid}

\label{sec:AbsractLieAlgd0}

This section is the core of this paper, where we describe the construction of the infinitesimal counterpart of a differentiable groupoid object. The first step is to generalize the notion of Lie algebroids in the category of smooth manifolds to the setting of tangent categories. This is the goal of Section~\ref{sec:AbsractLieAlgd}. In Section~\ref{sec:DiffProcGpdObj}, we set up the differentiation procedure, and finally state and prove the main theorems of this section.

\subsection{Abstract Lie algebroids in a tangent category}
\label{sec:AbsractLieAlgd}

\begin{Definition}
\label{def:AbstractLieAlgd}
Let $\calC$ be a cartesian tangent category with scalar $R$-mul\-ti\-pli\-ca\-tion. An \textbf{abstract Lie algebroid} in $\calC$ consists of a bundle of $R$-modules $A \to X$, a morphism $\rho: A \to TX$ of bundles of $R$-modules, called the \textbf{anchor}, and a Lie bracket $[\Empty , \Empty]$ on the abelian group $\Gamma(X,A)$, such that
\begin{align}
\label{eq:LieAlgdLeibniz}
  [a, fb] &= f[a,b] + \bigl((\rho \circ a) \cdot f \bigr) b
  \\
%\shortintertext{and}
\label{eq:anchorLieAlgMap}
  \rho \circ [a,b] &= [\rho \circ a , \rho \circ b]
\end{align}
for all sections $a$, $b$ of $A$ and all morphisms $f: X \to R$ in $\catC$.
\end{Definition}

Recall that the $\calC(X,R)$-module structure $fa$ is defined by~\eqref{eq:SectionsModuleStr}. Moreover, the action of vector fields on $R$-valued morphisms is defined by~\eqref{eq:VecFieldFunc}. Equation~\eqref{eq:LieAlgdLeibniz} is the Leibniz rule for Lie algebroids. By an analogous argument as in Corollary~\ref{cor:LieAlgebraBilinear}, we have that the Lie bracket on $\Gamma(X,A)$ is $\catC(*,R)$-linear. Observe that $\catC(*,R)$ is a ring in sets that can generally not be identified with the ring object $R$ in $\calC$. In $\Mfld$, where $R = \bbR$, we retrieve the usual $\bbR$-linearity of the Lie bracket.

In the definition of a usual Lie algebroid, the condition \eqref{eq:anchorLieAlgMap} that $\rho$ is a morphism of Lie algebras is redundant, since it can be proved using the Jacobi identity, bilinearity and antisymmetry of the Lie bracket, and the Leibniz rule. However, this proof relies on the identification of vector fields as derivations of the ring of functions on $X$ and the definition of the Lie bracket as the commutator. In our case, vector fields are sections of the tangent bundle and the Lie bracket is defined by Diagram~\eqref{eq:BracketDef}. In the smooth manifold setting, these coincide.

\begin{Example}
Lie algebroids are abstract Lie algebroids in the tangent category $\Mfld$.
\end{Example}

\subsection{The differentiation procedure}
\label{sec:DiffProcGpdObj}

Let $G$ be a groupoid in a tangent category $\catC$. The source bundle $s: G_1 \to G_0$ is equipped with the right $G$-action given by groupoid multiplication. If $G$ is differentiable in the sense of Definition~\ref{def:DiffGroupoid}, then this bundle is differentiable in the sense of Definition~\ref{def:DiffBunG}. Its vertical tangent bundle will be denoted by
\begin{equation*}
  VG_1 = TG_1 \times^{Ts, \, 0_{G_0}}_{TG_0} G_0
  \xrightarrow{~r_{VG_1}~}
  G_0
  \,.
\end{equation*}
It follows from Proposition~\ref{prop:VE} that $VG_1$, equipped with the right $G$-action defined in~\eqref{eq:betaVTanE} for $n=1$, is a $G$-bundle. Recall from Definition~\ref{def:InvVecField} that a vector field $v: G_1 \to TG_1$ is called invariant if it factors through a morphism $v': G_1 \to VG_1$ that is $G$-equivariant with respect to the right $G$-actions.

\begin{Proposition}
\label{prop:SubmoduleStrG0}
Let $G$ be a differentiable groupoid in a cartesian tangent category $\catC$ with scalar $R$-multiplication $\kappa:R \times T \to T$. Then, the set $\calX(G_1)^G$ of invariant vector fields is naturally a $\calC(G_0, R)$-module with the addition of vector fields and the module structure
\begin{equation}
\label{eq:ModStrTargetPullback}
  fv := (t^*\!f)v = \kappa_{G_1} \circ (t^*\! f,v)
\end{equation}
for all $v \in \calX(G_1)^G$ and $f \in \calC(G_0, R)$.
\end{Proposition}

\begin{proof}
\sloppy
It follows from Proposition~\ref{prop:SubmoduleStr} that $\calX(G_1)^G$ is a $\catC(G_1,R)^G$-submodule of $\Gamma(G_1, TG_1)$. The pullback
\begin{equation*}
\begin{aligned}
t^*: \catC(G_0,R) &\longrightarrow \catC(G_1,R)
\\
f &\longmapsto f \circ t \,,
\end{aligned}
\end{equation*}
\fussy
is a ring homomorphism by Lemma~\ref{lem:PullbackRingHom}. Its image is the ring $\catC(G_1,R)^G$ of invariant morphisms by Axiom~\eqref{diag:SourceTargetCond} of groupoids. Then, the assignment $fv = (t^*\!f)v$ equips $\calX(G_1)^G$ with a $\catC(G_0,R)$-module structure by Lemma~\ref{lem:PullbackRingHom0}.
\end{proof}

In the Definition~\ref{def:DiffGroupoid} of differentiable groupoids, we have assumed that the pullback~\eqref{diag:DiffGroupoidG02} exists, which can be written as the restriction of the vertical bundle to the identity bisection:
\begin{equation}
\label{eq:LABundle}
\begin{tikzcd}
\makebox[0pt][r]{$A \ := \ $}
G_0 \times_{G_1} VG_1
\ar[r, "i_A"] \ar[d, "p_A"']
\arrow[dr, phantom, "\lrcorner", very near start] 
&
VG_1
\ar[d, "\pi'_{G_1}"]
\\
G_0 \ar[r ,"1"']
&
G_1
\end{tikzcd}
\end{equation}
The vertical tangent bundle $\pi'_{G_1}: VG_1 \to G_1$ is a bundle of $R$-modules by Corollary~\ref{cor:newproj}. It follows from Proposition~\ref{prop:PullbackBundleR} that $p_A: A \to G_0$ is a bundle of $R$-modules. The composition
\begin{equation}
\label{eq:AnchorDef}
  \rho:
  A \xrightarrow{~i_A~} VG_1 
  \xrightarrow{~i_{VG_1}~} TG_1
  \xrightarrow{~Tt~} TG_0
  \,,
\end{equation}
is a morphism of bundles of $R$-modules, since all three arrows are. Indeed, $i_A$ and $i_{VG_1}$ are morphisms of bundles of $R$-modules by Propositions~\ref{prop:PullbackBundleR} and~\ref{prop:KerOfMorphR}~(ii) respectively. So is the map $Tt$ by the naturality of the tangent structure. 

\begin{Theorem}
\label{thm:SecAInvVec}
Let $G$ be a differentiable groupoid in a cartesian tangent category $\calC$ with scalar $R$-multiplication. Let $A \to G_0$ be as defined in~\eqref{eq:LABundle}. Then there is a natural isomorphism of $\calC(G_0, R)$-modules
\begin{equation*}
  \phi: \Gamma(G_0, A) \xrightarrow{~\cong~} \calX(G_1)^G
  \,.
\end{equation*}
\end{Theorem}

\begin{Theorem}
\label{thm:LieAlgdOfGroupoid}
Let $G$ be a differentiable groupoid in a cartesian tangent category $\catC$ with scalar $R$-multiplication. Then the bundle $A \to G_0$ defined in~\eqref{eq:LABundle}, with the anchor $\rho: A \to TG_0$ defined in~\eqref{eq:AnchorDef} and the Lie bracket of invariant vector fields on $\calX(G_1)^G \cong \Gamma(G_0, A)$ is an abstract Lie algebroid.
\end{Theorem}

The rest of this section is devoted to the proofs of Theorems~\ref{thm:SecAInvVec} and~\ref{thm:LieAlgdOfGroupoid}.

\subsubsection{Proof of \texorpdfstring{Theorem~\ref{thm:SecAInvVec}}{the Theorem}}

Let $\Gamma(G_1, VG_1)^G$ denote the set of $G$-equivariant sections of $\pi'_{G_1}:VG_1 \to G_1$. By the definition of invariant vector fields, the map
\begin{equation}
\begin{split}
\label{eq:GammaXGiso}
  (i_{VG_1})_*: \Gamma(G_1, VG_1)^G
  &\longrightarrow \calX(G_1)^G
  \\
  v' &\longmapsto i_{VG_1} \circ v'
\end{split}
\end{equation}
is surjective. Since $i_{VG_1}$ is a monomorphism, the map~\eqref{eq:GammaXGiso} is also injective. By Proposition~\ref{prop:SubmoduleStrG0}, $\calX(G_1)^G$ has a $\catC(G_0,R)$-module structure given by~\eqref{eq:ModStrTargetPullback}. Similarly, $\Gamma(G_1, VG_1)^G$ has a $\catC(G_0,R)$-module structure as follows. Let $f:G_0 \to R$ be a morphism in $\catC$ and let $v':G_1 \to VG_1$ be a $G$-equivariant section. Consider the composition
\begin{equation*}
fv': G_1 \xrightarrow{~(t^*\!f, \, v')~} R \times VG_1 \xrightarrow{~\kappa'_{G_1}~} VG_1 
\,.
\end{equation*}
The morphism $\kappa'_{G_1}$ is $G$-equivariant by Proposition~\ref{prop:ModuleStrRest}. By assumption, the section $v'$ is also $G$-equivariant. Moreover, the morphism $t^*\!f$ is invariant by the commutative diagram on the right of~\eqref{diag:SourceTargetCond}. This implies that $fv' = \kappa'_{G_1} \circ (t^*\!f, v')$ is $G$-equivariant. It follows from the $R$-linearity of $i_{VG_1}$, as expressed by Diagram~\eqref{diag:NewOldKappa}, that
\begin{equation*}
\begin{split}
(i_{VG_1})_* (fv')
&=
i_{VG_1} \circ \kappa'_{G_1} \circ (t^*\!f, v')
\\
&=
\kappa_{G_1} \circ (\id_R \times i_{VG_1}) \circ (t^*\!f, v')
\\
&=
\kappa_{G_1} \circ (t^*\!f, i_{VG_1} \circ v')
\\
&=
f\big((i_{VG_1})_* (v')\big) \,.
\end{split}
\end{equation*}
We conclude that~\eqref{eq:GammaXGiso} is an isomorphism of $\calC(G_0, R)$-modules. It remains to show that there is an isomorphism of $\catC(G_0,R)$-modules between $\Gamma(G_0, A)$ and $\Gamma(G_1, VG_1)^G$.

Let $a$ be a section of $A$. Consider the following morphism:
\begin{equation*}
%\label{eq:SecARightInv01}
  \phi'(a):
  G_1 
  \xrightarrow{~(a \circ t, \, \id_{G_1})~}
  A \times^{p_A,t}_{G_0} G_1
  \xrightarrow{~i_A \times_{G_0} \id_{G_1}~}
  VG_1 \times^{r_{VG_1},t}_{G_0} G_1 
  \xrightarrow{~\beta_{VG_1}~}
  VG_1
  \,.
\end{equation*}
Equationally,
\begin{equation}
\label{eq:SecARightInv01}
  \phi'(a) 
  = \beta_{VG_1} \circ (i_A \circ a \circ t, \id_{G_1})
  \,.    
\end{equation}
First, we show that $\phi'(a)$ is a section of the vertical tangent bundle by a straightforward calculation,
\begin{equation*}
\begin{split}
  &\pi'_{G_1} \circ \phi'(a) 
  \\
  = \
  &\pi'_{G_1} \circ \beta_{VG_1} \circ (i_A \circ a \circ t, \id_{G_1}) \quad \text{by Equation~\eqref{eq:SecARightInv01}}
  \\
  = \
  &\pi_{G_1} \circ i_{VG_1} \circ \beta_{VG_1} \circ (i_A \circ a \circ t, \id_{G_1}) \quad \text{by the first diagram in~\eqref{diag:newNatTransf}}
  \\
  = \
  &\pi_{G_1} \circ \beta_{TG_1} \circ (i_{VG_1} \times_{0_{G_0}} 0_{G_1}) \circ (i_A \circ a \circ t, \id_{G_1}) \quad \text{by Diagram~\eqref{eq:VE02}}
  \\
  = \ 
  &m \circ (\pi_{G_1} \times_{\pi_{G_0}} \pi_{G_1}) \circ (i_{VG_1} \times_{0_{G_0}} 0_{G_1}) \circ (i_A \circ a \circ t, \id_{G_1})
  \\
  &\text{by the inner square of~\eqref{diag:alphaGEquiv} for $\beta_{G_1}=m$ and $\alpha=\pi$}
  \\
  = \
  &m \circ \big((\pi_{G_1} \circ i_{VG_1}) \times_{G_0} \id_{G_1}\big) \circ (i_A \circ a \circ t, \id_{G_1})
  \\
  &\text{by the functoriality of pullbacks and that $\pi \circ 0 = \Id$}
  \\
  = \
  &m \circ (\pi'_{G_1} \circ i_A \circ a \circ t, \id_{G_1}) \quad \text{by Lemma~\ref{lem:MapsinPar}~(ii)}
  \\
  = \ 
  &m \circ (1 \circ p_A \circ a \circ t, \id_{G_1}) \quad \text{by Diagram~\eqref{eq:LABundle}}
  \\
  = \
  &m \circ (1 \circ t, \id_{G_1}) \quad \text{since $a$ is a section of $p_A:A \to G_0$}
  \\
  = \
  &\id_{G_1} \quad \text{by the unitality axiom~\eqref{eq:Unit1}}
\,.
\end{split}
\end{equation*}
%where we have used Equation~\eqref{eq:SecARightInv01}, the first diagram in~\eqref{diag:newNatTransf}, Diagram \eqref{eq:VE02}, the inner square of~\eqref{diag:alphaGEquiv} (for $\beta_{G_1}$ given by the group multiplication $m$ and $\alpha=\pi$), the functoriality of pullbacks, that $\pi \circ 0 = \Id$, Diagram~\eqref{eq:LABundle}, that $a$ is a section of $p_A:A \to G_0$, and the unitality axiom~\eqref{eq:Unit1}.
To show the $G$-equivariance of $\phi'(a)$, we consider the following diagram:
\begin{equation}
\label{eq:RightInv}
\begin{tikzcd}[column sep=5em, row sep=3em] 
G_1 \times_{G_0} G_1 
\ar[d, "{(t,\id_{G_1}) \times_{G_0} {\id_{G_1}}}"', "\cong"] \ar[r, "m"]
&
G_1 
\ar[d, "{(t,\id_{G_1})}", "\cong"'] 
\\
G_0 \times_{G_0}^{\id, t}
G_1 \times_{G_0} G_1
\ar[d, "{a \times_{G_0} \id_{G_2}}"'] 
\ar[r, "\id_{G_0} \times_{G_0} m"]
%\ar[ddd, bend left=45, "v_a \times_{G_0} \id"]
&
G_0 \times_{G_0}^{\id, t} G_1 
\ar[d, "{a \times_{G_0} \id_{G_1}}"]
%\ar[ddd, bend right=45, "v_a"']
\\
A \times_{G_0} G_1 \times_{G_0} G_1 
\ar[d, "{i_A \times_{G_0} \id_{G_2}}"'] 
\ar[r, "\id_A \times_{G_0} m"]
&
A \times_{G_0} G_1 
\ar[d, "i_A \times_{G_0} \id_{G_1}"]  
\\
VG_1 \times_{G_0} G_1 \times_{G_0} G_1 
\ar[d, "\beta_{VG_1} \times_{G_0} \id_{G_1}"'] 
\ar[r, "\id_{VG_1} \times_{G_0} m"]
&
VG_1 \times_{G_0} G_1
\ar[d, "\beta_{VG_1}"]
\\
VG_1 \times_{G_0} G_1 
\ar[r, "\beta_{VG_1}"']
&
VG_1 
\end{tikzcd}
\end{equation}
The commutativity of the top square follows from the pasting law of pullbacks and by the right diagram in~\eqref{diag:SourceTargetCond}. The second square commutes due to the functoriality of the fiber product,
\begin{equation*}
\begin{split}
  (a \times_{G_0} \id_{G_1}) \circ (\id_{G_0} \times_{G_0} m)
  &= (a \circ \id_{G_0}) \times_{G_0} (\id_{G_1} \circ m)
  \\
  &= (\id_{A} \circ a) \times_{G_0} (m \circ \id_{G_2})
  \\
  &=
  (\id_{A} \times_{G_0} m) \circ (a \times_{G_0} \id_{G_2})
  \,.
\end{split}
\end{equation*}
The commutativity of the third square follows from an analogous argument, where we replace $a$ with $i_A$. The commutativity of the bottom square expresses the associativity of the right action $\beta_{VG_1}$. It follows that the outer rectangle of Diagram~\eqref{eq:RightInv} commutes. The composition of the right vertical arrows is $\phi'(a)$. The composition of the left vertical arrows is $\phi'(a) \times_{G_0} \id_{G_1}$. This shows that $\phi'(a)$ is $G$-equivariant. 

So far, we have constructed a map
\begin{equation*}
  \phi': \Gamma(G_0, A) \longrightarrow
  \Gamma(G_1, VG_1)^G
  \,.
\end{equation*}
Next, we construct a map in the opposite direction. Given a section $v' \in \Gamma(G_1,VG_1)$, we get a section of $A$ by restriction to the identity bisection. More precisely, there is a unique morphism $\psi'(v'): G_0 \to A$, such that
\begin{equation}
\label{eq:psiprime01}
  i_A \circ \psi'(v')
  = v' \circ 1  
  \,.
\end{equation}
It follows from this equation that
\begin{equation*}
\begin{split}
  p_A \circ \psi'(v') 
  &=
  s \circ 1 \circ p_A \circ \psi'(v')
  \\
  &= 
  s \circ \pi'_{G_1} \circ i_A \circ \psi'(v')
  \\
  &= s \circ \pi'_{G_1} \circ v' \circ 1
  \\
  &= s \circ 1
  \\
  &= \id_{G_0}
  \,,
\end{split}
\end{equation*}
where we have used that $s \circ 1 = \id_{G_0}$, Diagram~\eqref{eq:LABundle}, and that $v'$ is a section of $\pi'_{G_1}:VG_1 \to G_1$. It follows that $\psi'(v')$ is a section of $p_A: A \to G_0$. By restricting to the $G$-equivariant sections, we obtain a map
\begin{equation*}
  \psi': \Gamma(G_1, VG_1)^G \longrightarrow 
  \Gamma(G_0, A)
  \,.
\end{equation*}

In the next step, we will show that $\phi'$ and $\psi'$ are mutually inverse. We have
\begin{equation*}
\begin{split}
  \phi'(a) \circ 1
  &= 
  \beta_{VG_1} \circ 
  (i_A \circ a \circ t, \id_{G_1}) \circ 1
  \\
  &=
  \beta_{VG_1} \circ (i_A \circ a, 1)
  \\
  &=
  i_A \circ a 
  \,,
\end{split}
\end{equation*}
where we have used the defining Equation~\eqref{eq:SecARightInv01} for $\phi'(a)$, that $t \circ 1 = \id_{G_0}$, and that the right action is unital. Comparing this equation with~\eqref{eq:psiprime01} for $v'=\phi'(a)$, we conclude that $a = \psi'\bigl(\phi'(a)\bigr)$. Conversely,
\begin{equation*}
\begin{split}
  \phi'\bigl(\psi'(v') \bigr)
  &=
  \beta_{VG_1} \circ 
  \bigl(i_A \circ \psi'(v') \circ t, \id_{G_1} \bigr)
  \\
  &=
  \beta_{VG_1} \circ 
  \bigl(v' \circ 1 \circ t, \id_{G_1} \bigr)  
  \\
  &=
  \beta_{VG_1} \circ 
  (v' \times_{G_0} \id_{G_1}) \circ (1 \circ t, \id_{G_1})
  \\
  &=
  v' \circ m \circ (1 \circ t, \id_{G_1})
  \\
  &=
  v'
  \,,
\end{split}
\end{equation*}
where we have used~\eqref{eq:SecARightInv01}, \eqref{eq:psiprime01}, that $v'$ is $G$-equivariant and Diagram~\eqref{eq:Unit1}. 
%that the groupoid multiplication is unital. 
We conclude that $\phi'$ and $\psi'$ are inverse to each other. 

Finally, we show that the map $\psi'$ is $\calC(G_0, R)$-linear. Denote by $\kappa_A:R \times A \to A$ the $R$-module structure on $A$. For $f \in \calC(G_0, R)$ and $v' \in \Gamma(G_1, VG_1)^G$ we have that
\begin{equation*}
\begin{split}
  i_A \circ f\psi'(v')
  &=
  i_A \circ \kappa_A \circ \bigl(f, \psi'(v')\bigr)
  \\
  &=
  \kappa'_{G_1} \circ (\id_R \times i_A) \circ
  \bigl(f, \psi'(v')\bigr) \quad \text{by the $R$-linearity of $i_A$}
  \\
  &=
  \kappa'_{G_1} \circ \bigl(f, i_A \circ \psi'(v')\bigr) \quad \text{by Lemma~\ref{lem:MapsinPar}~(ii)}
  \\
  &=
  \kappa'_{G_1} \circ \bigl(f, v' \circ 1) \quad \text{by Equation~\eqref{eq:psiprime01}}
  \\
  &=
  \kappa'_{G_1} \circ \bigl(f \circ t, v') \circ 1 \quad \text{since $t \circ 1 = \id_{G_0}$}
  \\
  &=
  \kappa'_{G_1} \circ \bigl(t^*\!f, v') \circ 1
  \\
  &=
  (t^*\!f) v' \circ 1
  \,,
\end{split}
\end{equation*}
%where we have used the definition of the module structure on $\Gamma(G_0, A)$, the $R$-linearity of the inclusion $i_A$, Equation~\eqref{eq:psiprime01}, that $t \circ 1 = \id_{G_0}$ 
where we have used the module structure~\eqref{eq:ModStrTargetPullback} on $\Gamma(G_1,VG_1)^G$. 
Comparing this equation with~\eqref{eq:psiprime01}, we conclude that 
\begin{equation*}
\psi'\big((t^*\!f)v'\big) = f\psi'(v')
\,,
\end{equation*}
which shows that $\psi'$ is $\catC(G_0,R)$-linear. It follows from Lemma~\ref{lem:RModMap} that its inverse $\phi'$ is also $\catC(G_0,R)$-linear. 

%\sloppy
Composing $\phi'$ with Isomorphism~\eqref{eq:GammaXGiso}, we obtain the isomorphism of $\catC(G_0,R)$-modules
\begin{equation}
\label{eq:ModIsoAlgd}
  \phi = (i_{VG_1})_* \circ \phi':
  \Gamma(G_0, A) \xrightarrow{~\cong~}
  \calX(G_1)^G
  \,,
\end{equation}
which concludes the proof.
\qed
%\fussy

\subsubsection{Proof of \texorpdfstring{Theorem~\ref{thm:LieAlgdOfGroupoid}}{the Theorem}}

It follows from Theorem~\ref{thm:InvariantBracket} that the Lie bracket of vector fields on $G_1$ restricts to a Lie bracket on the abelian group $\calX(G_1)^G$ of invariant vector fields. Let $a$ and $b$ be sections of $A \to G_0$; let $\phi(a)$ and $\phi(b)$ be the corresponding invariant vector fields given by Isomorphism~\eqref{eq:ModIsoAlgd}. The Lie bracket $[a,b]$ is defined by
\begin{equation}
\label{eq:LieAlgdBracketDef}
  \phi\bigl( [a, b] \bigr)
  = [\phi(a), \phi(b)]
  \,.
\end{equation}
The target morphism $t$ of the groupoid is invariant under the right $G$-multiplication, as given in the right of Diagram~\eqref{diag:SourceTargetCond}. This implies that $Tt$ is invariant under the right $TG$-action, so that the bottom square of the following diagram
\begin{equation*}
\begin{tikzcd}[column sep= large, row sep=large]
VG_1 \times_{G_0} G_1
\ar[r, "\beta_{VG_1}"] 
\ar[d, "{i_{VG_1} \times_{0_{G_0}} 0_{G_1}}"]
\ar[dd, "i_{VG_1} \circ \, \pr_1"', bend right=75]
&
VG_1
\ar[d, "i_{VG_1}"] 
\\
TG_1 \times_{TG_0} TG_1
\ar[r, "\beta_{TG_1}"']
\ar[d, "\pr_1"]
&
TG_1
\ar[d, "Tt"]
\\
TG_1
\ar[r, "Tt"']
&
TG_0
\end{tikzcd}
\end{equation*}
commutes. The top square commutes by the definition of the $G$-action on the vertical tangent bundle (Diagram~\ref{eq:VE02} for $n=1$ and $E=G_1$). It follows that the outer square is equivalent, so that we obtain the relation
\begin{equation}
\label{eq:TtInv01}
  Tt \circ i_{VG_1} \circ \beta_{VG_1} 
  =
  Tt \circ i_{VG_1} \circ \pr_1
  \,.
\end{equation}

To prove Equation~\eqref{eq:anchorLieAlgMap}, we first calculate that
\begin{equation}
\label{eq:Trel1}
\begin{split}
  Tt \circ \phi(a)
  &=
  Tt \circ i_{VG_1} \circ \phi'(a) \quad \text{by Isomorphism~\eqref{eq:ModIsoAlgd}}
  \\
  &=
  Tt \circ i_{VG_1} \circ 
  \beta_{VG_1} \circ (i_A \circ a \circ t, \id_{G_1}) \quad \text{by Equation~\eqref{eq:SecARightInv01}}
  \\
  &=
  Tt \circ i_{VG_1} \circ 
  \pr_1 \circ (i_A \circ a \circ t, \id_{G_1}) \quad \text{by Equation~\eqref{eq:TtInv01}}
  \\
  &=
  Tt \circ i_{VG_1} \circ i_A \circ a \circ t
  \\
  &=
  (\rho \circ a) \circ t
  \,,
\end{split}
\end{equation}
%where we have used the Isomorphism~\eqref{eq:ModIsoAlgd}, Equations~\eqref{eq:SecARightInv01} and~\eqref{eq:TtInv01},
where we have used the defining Equation~\eqref{eq:AnchorDef} of the anchor. This shows that the vector field $\phi(a)$ on $G_1$ is $t$-related to the vector field $\rho \circ a$ on $G_0$. It follows from Proposition~\ref{prop:ProjectVectorFields} that $[\phi(a), \phi(b)]$ is $t$-related to $[\rho \circ a, \rho \circ b]$, so that
\begin{equation*}
\begin{split}
  \rho \circ [a,b] \circ t
  &=
  Tt \circ \phi([a, b])
  =
  Tt \circ [\phi(a), \phi(b)]
  \\
  &=
  [\rho \circ a, \rho \circ b]
  \circ t
  \,,
\end{split}
\end{equation*}
where we have used~\eqref{eq:Trel1} for the section $[a,b]$ and Equation~\eqref{eq:LieAlgdBracketDef}. Since $t$ is an epimorphism, we can cancel $t$ on both sides of this equation, which yields Equation~\eqref{eq:anchorLieAlgMap}.

Let $f \in \calC(G_0, R)$. By definition, the function $t^*\!f:G_1 \to R$ is $t$-related to $f:G_0 \to R$. From~\eqref{eq:Trel1}, we know that the vector field $\phi(a)$ on $G_1$ is $t$-related to the vector field $\rho \circ a$ on $G_0$. It follows from Proposition~\ref{prop:ProjectFunctions}~(ii) that the function $\phi(a) \cdot (t^*\!f)$ is $t$-related to $(\rho \circ a) \cdot f$, that is,
\begin{equation}
\label{eq:MidStep}
\begin{split}
  \phi(a) \cdot (t^*\!f)
  &=
  \big((\rho \circ a) \cdot f \big) \circ t
  \\
  &=
  t^*\!\bigl( (\rho \circ a) \cdot f \bigr)
  \,.
\end{split}
\end{equation}
To prove the Leibniz rule, we calculate 
\begin{equation*}
\begin{split}
  \phi\bigl( [a, fb] \bigr)
  &=
  [\phi(a), \phi(fb)] \quad \text{by Equation~\eqref{eq:LieAlgdBracketDef}}
  \\
  &=
  [\phi(a), (t^*\!f)\phi(b)] \quad \text{by the $\calC(G_0, R)$-linearity of $\phi$}
  \\  
  &= \bigl(\phi(a) \cdot (t^*\!f)\bigr) \phi(b) + (t^*\!f)[\phi(a),\phi(b)] \quad \text{by Proposition~\ref{prop:LeibnizRule}}
  \\
  &= \bigl(t^*( (\rho\circ a) \cdot f) \bigr) \phi(b) 
  + (t^*\!f)[\phi(a),\phi(b)] \quad \text{by Equation~\eqref{eq:MidStep}}
  \\
  &= \phi\bigl( ( (\rho\circ a) \cdot f)b \bigr) 
  + (t^*\!f)\phi([a,b])
  \\
  &=
  \phi\bigl( ( (\rho\circ a) \cdot f)b + f[a,b] \bigr)
  \,,
\end{split}
\end{equation*}
%where we have used Equation~\eqref{eq:LieAlgdBracketDef}, the $\calC(G_0, R)$-linearity of $\phi$, Proposition~\ref{prop:LeibnizRule}, Equation~\eqref{eq:MidStep}, 
where we have once more used the $\calC(G_0, R)$-linearity of $\phi$ and Equation~\eqref{eq:LieAlgdBracketDef}. By Theorem~\ref{thm:SecAInvVec}, $\phi$ is an isomorphism, so that we can cancel $\phi$ on both sides of the equation, which yields Equation~\eqref{eq:LieAlgdLeibniz}.
\qed
\section{Examples and applications}
\label{sec:Examples}

This is the concluding section of the article. The goal is to provide classes of examples and possible applications of the differentiation procedure developed in this paper. The computation of explicit examples consists of three steps:
\begin{itemize}
\item[(1)] Identify a cartesian tangent category $\catC$ with scalar $R$-multiplication.
\item[(2)] Identify a differentiable groupoid object in $\catC$.
\item[(3)] Apply the differentiation procedure of Section~\ref{sec:DiffProcGpdObj} step by step.
\end{itemize}
By Theorem~\ref{thm:LieAlgdOfGroupoid}, this yields an abstract Lie algebroid in the tangent category.

\subsection{Lie groupoids and Lie algebroids}
\label{sec:ExampleLieGrpd}

The prototypical cartesian tangent category is the category $\Mfld$ of finite-dimensional smooth manifolds with the usual scalar $\bbR$-multiplication of tangent vectors.

%The namesake example for tangent categories is the category of smooth finite-dimensional manifolds with the the usual tangent functor and scalar $\bbR$-multiplication of tangent vectors.

\begin{Proposition}
\label{prop:LieGrpdDiff}
A groupoid in $\Mfld$ is differentiable if and only if it is a Lie groupoid.
\end{Proposition}

\begin{proof}
A Lie groupoid satisfies all conditions of a differentiable groupoid. For the converse direction, we consider a smooth map of manifolds $f: X \to Y$. Assume that the pullback $X \times_Y X$ exists and that the tangent functor commutes with the pullback, $T(X \times_Y X) \cong TX \times_{TY} TX$. This implies that the tangent fiber at $(x,x') \in X \times_Y X$ is given by
\begin{equation*}
  T_{(x,x')} (X \times_Y X)
  \cong
  T_x X \times_{T_y Y} T_{x'} X
  \,,
\end{equation*}
where $y = f(x) = f(x')$. Sard's theorem implies that there is at least one regular value $y \in Y$ of $f$. Let $x$ and $x'$ be elements of the fiber over $y$. Then the dimension of the tangent fiber over $(x,x')$ is $2\dim X - \dim Y$. At all other points, where either $T_x f$ or $T_{x'}f$ is not surjective, the dimension of the tangent fiber is strictly larger. Since the dimension of the tangent fibers of the manifold $X \times_Y X$ must be constant, it follows that both $T_x f$ and $T_{x'}f$ are surjective at all points $(x,x') \in X \times_Y X$. In particular, for every $x \in X$, the pair $(x,x)$ is an element of $X \times_Y X$, so that $T_x f$ is surjective. We conclude that $f$ is a submersion.

For every groupoid object $G$, we have an isomorphism
\begin{equation*}
  G_1 \times_{G_0}^{s,t} G_1 
  \cong G_1 \times_{G_0}^{s,s} G_1
  \,,
\end{equation*}
obtained by applying the inverse to the second factor of the pullback. If $G$ is differentiable, the tangent functor commutes with the pullback on the right hand side. As we have just proved, this implies that the source map $s$ is a submersion. We conclude that a differentiable groupoid in manifolds is a Lie groupoid.
\end{proof}

Our generalized differentiation procedure recovers the usual construction of the Lie algebroid of a Lie groupoid.

\begin{Proposition}
\label{prop:LieAlgdDiff}
Abstract algebroids in $\Mfld$ are the same as Lie algebroids.
\end{Proposition}
\begin{proof}
In Example~\ref{ex:VecBund}, we have seen that bundles of $\bbR$-modules are vector bundles. All other axioms are the same for abstract and non-abstract Lie algebroids in $\Mfld$.
\end{proof}

\subsection{Elastic diffeological groups}

Diffeological spaces are concrete sheaves on the site of Euclidean spaces with the usual open covers \cite[Definition~3.4]{blohmann2024}. Elastic diffeological spaces are diffeological spaces with certain \textit{elasticity} axioms \cite[Definition~4.1]{blohmann2024}. These axioms ensure that the left Kan extension of the tangent structure on Euclidean spaces defines a tangent structure on elastic diffeological spaces. Moreover, the left Kan extension of the scalar $\bbR$-multiplication on Euclidean spaces is a scalar $\bbR$-multiplication on elastic diffeological spaces in the sense of Definition~\ref{def:RScalar}. The upshot is that the full subcategory of elastic diffeological spaces is a cartesian tangent structure with scalar $\bbR$-multiplication \cite[Theorem~4.2]{blohmann2024}.

\subsubsection{Diffeomorphism groups}

Given a smooth (not necessarily compact) manifold, its group $\Diff(M)$ of diffeomorphisms is not a finite-dimensional Lie group. It is commonly viewed as an infinite-dimensional Lie group, in the sense that it is locally modeled on compactly supported vector fields \cite{KrieglMichor:1997, Neeb:2006}. $\Diff(M)$ can alternatively be viewed as a diffeological space equipped with the functional diffeology \cite[Example~3.8~(e)]{blohmann2024}. Since $\Diff(M)$ is an elastic diffeological group \cite[Example~5.4]{blohmann2024}, we can differentiate it using our method.

It is well-known that the Lie algebra of the diffeomorphism group of a compact smooth manifold viewed as a Fr{\'e}chet Lie group is given by its space of vector fields with the (opposite) Lie bracket. While our construction recovers this result, it has the following advantages:
\begin{enumerate}

\item
Our approach is based entirely on universal categorical constructions, sidestepping all technicalities of functional analysis and Fr{\'e}chet manifolds.

\item
As a consequence, our construction applies without modification to the diffeomorphism group of a non-compact manifold. We posit that its Lie algebra is the Lie algebra of \textit{all} vector fields, without any conditions on their support or behavior at infinity \cite{BlohmannMiyamoto}.

\end{enumerate}

\subsubsection{Gauge transformations}
Given a smooth manifold $M$ and a Lie group $G$ with Lie algebra $\mathfrak{g}$, the diffeological mapping space ${\intDflg}(M,G)$ is elastic with tangent space
\begin{equation*}
\begin{split}
T {\intDflg}(M,G) &\cong {\intDflg} (M, TG)
\\
&\cong {\intDflg} (M, G \times \mathfrak{g})
\,,
\end{split} 
\end{equation*}
where we have used \cite[Corollary~5.12]{blohmann2024}. For a recap on diffeological mapping spaces, see \cite[Example~3.8~(e)]{blohmann2024}. This has numerous applications in the setting of gauge theory \cite[Section~9.3]{blohmannLFT}, where the fields are connections on a principal $G$-bundle $P \to M$. Here, $\intAut(P) \cong \Gamma(M, P \times_{\Ad}G)$ is the diffeological gauge group of a principal $G$-bundle $P \to M$. Its Lie algebra is $\Gamma(M, P \times_{\Ad}\frakg)$.

\subsubsection{Bisection groups of Lie groups}
The last two classes of examples belong to a larger class: bisection groups of Lie groupoids. A \textbf{bisection} of a Lie groupoid $G_1 \rightrightarrows G_0$ is a section $\sigma:G_0 \to G_1$ of the source map $s$ such that $t \circ \sigma:G_0 \to G_0$ is a diffeomorphism. The set of bisections on a Lie groupoid has a natural group structure, and is denoted by $\Bis(G)$ (e.g. \cite[Proposition~1.4.2]{mackenzie2005}).

Given a smooth manifold $M$, its diffeomorphism group is the bisection group of the pair groupoid $M \times M \rightrightarrows M$. For a Lie group $G$, the symmetry group ${\intDflg}(M,G)$ of global sections of the trivial principal $G$-bundle over $M$ is the bisection group of the trivial Lie group bundle\footnote{A \textbf{Lie group bundle} is a Lie groupoid where the source and target maps coincide. In this case, the source (or target) fibers have the structure of a Lie group.} 
$G \times M \to M$.

It is known that the Lie algebra of the bisection group $\Bis(G)$ of a Lie groupoid $G$ with compact base is isomorphic to the Lie algebra of sections of the Lie algebroid of $G$.
In \cite{SchmedingWockel:2015}, the authors compute the Lie algebra of $\Bis(G)$ by using sophisticated machinery from functional analysis \cite{Milnor:1984, Neeb:2006}, where $\Bis(G)$ has a natural locally convex Lie group structure. Another approach is given in \cite[Section~15]{models}, where the authors differentiate $\Bis(G)$ using paths, choosing a splitting of the tangent bundle over $G_0$. This can, for instance, be done with a Riemannian metric. The heuristic approach of \cite{models} is in essence diffeological.

As explained for the case of diffeomorphism groups, the advantage of our approach is that it does not involve choices of norms or metrics, but only universal constructions, so that it encompasses the case when $M$ is not compact. We conjecture that bisection groups of Lie groupoids are elastic and that their Lie algebras are the Lie algebras of sections of the Lie algebroid.

\subsection{Elastic diffeological groupoids}

We now present classes of possible applications of elastic diffeological groupoids. We start with action groupoids induced by actions of elastic diffeological groups on elastic diffeological spaces. Concrete examples of such action groupoids appear naturally in classical field theory.

\subsubsection{Actions of diffeomorphism groups}
Consider the natural action of the diffeomorphism group $\Diff(M)$ of a smooth manifold $M$ on the sections of a natural fiber bundle
%\footnote{In a natural bundle, the diffeomorphisms between open subsets of the base lift functorially to diffeomorphisms between local sections (e.g. \cite[Remark~8.3.13]{blohmannLFT}). Examples of natural bundles include the (co)tangent bundle of a manifold, tensor bundles, and finite order jet bundles. For an exhaustive characterization see \cite{EpsteinThurston:1997}. Taking sections of these bundles we obtain vector fields, differential forms, Riemannian metrics, etc.} 
$p:F \to M$. As we have seen above, $\Diff(M)$ is elastic. Moreover, $\Gamma(M,F)$ is elastic \cite[Corollary~5.12]{blohmann2024} with tangent space
\begin{equation*}
T \Gamma(M,F) \cong \Gamma(M, V\!F)
\,,
\end{equation*}
where $V\!F = \ker Tp = TF \times_{TM}^{Tp, 0_M} M$ is the vertical tangent bundle. The induced action diffeological groupoid
\begin{equation*}
\Diff(M) \times \Gamma(M,F) \rightrightarrows \Gamma(M,F)
\end{equation*}
is elastic, which follows from the fact that the finite product of elastic spaces is elastic \cite[Proposition~4.9]{blohmann2024}. This observation has applications in Lagrangian field theory \cite{blohmannLFT}.

\subsubsection{Actions of gauge transformations}
Similarly, given a principal $G$-bundle $P \to M$ and an open subset $U \subset M$, the diffeological mapping space ${\intDflg}(U,G)$ acts on the diffeological space of local sections $\Gamma(U,P)$. This yields a sheaf of elastic diffeological groupoids
\begin{equation*}
{\intDflg}(U,G) \times \Gamma(U,P) \rightrightarrows \Gamma(U,P) \,.
\end{equation*}
By differentiation, we obtain a sheaf of abstract Lie algebroids in diffeological spaces.

\subsubsection{Elastic diffeological action groupoids}
The last two classes of examples belong to a larger class: diffeological action groupoids. An action of an elastic diffeological group $G$ on an elastic diffeological space $M$ induces an action diffeological groupoid $G \times M \rightrightarrows M$. As mentioned above, since the finite product of elastic spaces is elastic \cite[Proposition~4.9]{blohmann2024}, this action groupoid is elastic.

\subsubsection{Groupoid symmetry of general relativity}

In \cite{BlohmannFernandesWeinstein:2013}, the authors embarked on a project to understand the symmetry structure of the initial value problem of general relativity \cite{BlohmannWeinstein:2024, BlohmannSchiavinaWeinstein}. The main result was the construction of a diffeological groupoid whose diffeological Lie algebroid has the same bracket as the somewhat mysterious Poisson bracket of the Gau\ss-Codazzi constraints for Ricci flat metrics. Without a theory of diffeological groupoids available, the differentiation of the diffeological groupoid of \cite{BlohmannFernandesWeinstein:2013} was carried out in an ad-hoc manner. The reason why this worked is that the authors had constructed a differentiable groupoid object in the tangent category of elastic diffeological spaces.

\subsubsection{Reduction of action Lie groupoids by subgroups}

The construction of the groupoid symmetry of general relativity from \cite{BlohmannFernandesWeinstein:2013} is a particular case of a more general construction: reductions of action Lie groupoids by subgroups.

Let $\varphi: G \to \Diff(M)$ be a smooth Lie group action on a smooth manifold. Let $H \subseteq G$ be a closed (not necessarily normal) Lie subgroup, such that $\varphi|_{H}$ is free and proper. There is a unique Lie groupoid structure on $H \setminus G \times_H M$ over $M/H$, where $H \setminus G \times_H M = (G \times M) / (H \times H)$ with respect to the action given by $(h_1,h_2) \cdot (g,m) = \big(h_1gh_2^{-1}, \varphi_{h_2}(m)\big)$. For the notation and statement of the main proposition, see \cite[Section~8.1]{BlohmannWeinstein:2024}. In particular, when $H$ is not normal, this reduction is not an action Lie groupoid. The computation of the Lie algebroid of this quotient groupoid in the setting of smooth finite-dimensional manifolds is given in \cite[Section~8.5]{BlohmannWeinstein:2024}. To see how this construction applies to general relativity, we refer to \cite[Example~8.13]{BlohmannWeinstein:2024}.

\subsection{Diffeological integration of Lie algebroids} 

It is known that Lie's third theorem fails in the setting of Lie groupoids and Lie algebroids \cite{CrainicFernandes:2003}. In \cite{Villatoro}, Villatoro addresses the question of integrability of Lie algebroids via diffeological groupoids. In his construction, he introduces a class of diffeological spaces, called \textit{quasi-{\'e}tale} diffeological spaces (QUED), which are locally quotients of smooth manifolds by a well-behaved equivalence relation. He then defines a Lie functor from \textit{singular} Lie groupoids (groupoid objects in QUED with source and target QUED-\textit{submersions}, and the space of objects given by a smooth manifold) to Lie algebroids. 

The crucial part of the integration procedure in \cite{Villatoro} is to show that the Weinstein groupoid $\Pi_1(A) \rightrightarrows M$ of a Lie algebroid $A \to M$ (e.g. \cite[Section~2]{CrainicFernandes:2003}) is a singular Lie groupoid. In an attempt to compare Villatoro's construction to ours, there is one subtlety. The Weinstein groupoid (at least of a non-transitive Lie algebroid) does not seem to be elastic and so its differentiation by our methods seems to be more involved.
%a comparison between our methods and theirs would involve overcoming one difficulty: it is not known whether the Weinstein groupoid is elastic in general.
%In future work, we aim to apply our differentiation method to the diffeological groupoids arising from Villatoro's integration.

\subsection{Holonomy groupoids of singular foliations}

The question of diffeological integrability has also been considered in the context of \textit{holonomy groupoids}. Associated to a singular foliation, there is a diffeological groupoid, called its holonomy groupoid \cite[Definition~3.5]{AndroulidakisSkandalis:2009}. More generally, associated to a singular subalgebroid of an integrable Lie algebroid, there is a holonomy groupoid \cite{Zambon:2022}. (The relation between the two is that the singular subalgebroids of $A=TM$ are precisely the singular foliations on $M$ \cite[Example~1.2]{Zambon:2022}).

In \cite{AndroulidakisZambon:2023}, Androulidakis and Zambon propose an integration method for singular subalgebroids via diffeological groupoids. In the first step of establishing the differentiation procedure, the authors consider a certain class of diffeological groupoids, which behave like the holonomy groupoid of a singular subalgebroid. These diffeological spaces have a smooth manifold as their space of objects, and come together with a morphism to a Lie groupoid, satisfying additional properties. Using families of global bisections, the authors differentiate such a holonomy-like diffeological groupoid to a singular subalgebroid \cite[Theorem~3.2]{AndroulidakisZambon:2023}.

Under the assumptions on the diffeological groupoids in \cite{AndroulidakisZambon:2023}, the tangent structure on smooth manifolds is sufficient and the question of a rigorous definition of tangent spaces to diffeological spaces can be circumvented. In an attempt to compare their construction to our differentiation procedure in the tangent category of elastic diffeological spaces, the following questions arise: Are holonomy groupoids of singular foliations (or more generally of singular subalgebroids) elastic? What are the minimal conditions that make them elastic? These are interesting questions for future work.

\subsection{Differentiable stacks}

A stack is a category fibered in groupoids that satisfies a descent condition, which is analogous to the $2$-cocycle condition of the transition functions of the local trivializations of a principal fiber bundle. A stack is differentiable if it is equipped with an atlas, that is, with a representable surjective submersion with domain a smooth manifold. The twofold fiber product of this morphism is a Lie groupoid, called a presentation of the stack. A morphism of differentiable stacks presented by Lie groupoids $G$ and $H$ gives rise to a generalized morphism from $G$ to $H$ given by a $G$-$H$ groupoid bibundle that is principal on the codomain \cite[Section~5.4]{moerdijk2003}. 

The composition of such bibundles is not associative on the nose, but only up to biequivariant isomorphism, so that we obtain a bicategory. As explained in \cite{Blohmann2008}, the $(2,1)$-category of Lie groupoids, principal bibundles, and biequivariant isomorphisms is equivalent to the bicategory of differentiable stacks \cite{BehrendXu:2011, Carchedi:2011}.

The advantage of working in the category of presentations of differentiable stacks is that the usual methods of differentiation and integration on smooth manifolds are at our disposal. In fact, constructing the convolution algebra of a Lie groupoid, which can be viewed as the noncommutative geometry of the stack, uses the Haar measures on the groupoid fibers. To our best knowledge, there is no known method that would construct the convolution directly on the stack. For the differentiation, the situation is similar. 

By applying the tangent functor to a Lie groupoid $G$, we obtain the tangent groupoid $TG$. By applying the tangent functor to a $G$-$H$ bibundle $E$, we obtain a $TG$-$TH$ bibundle $TE$. The bibundle $E$ is $H$-principal if and only if the endpoint map of the $H$-action, $E \times_{H_0} H_0 \to E \times_{G_0} E$ is an isomorphism. Since $T$ preserves isomorphisms by virtue of being a functor, $T$ preserves principality. A similar argument applies to the composition of one-sided principal bibundle and the associator of the bicategory. The upshot is that $T$ gives rise to an endofunctor of the presentation category of differentiable stacks. The same is true for powers and fiber products of $T$. Moreover, the natural transformations of the tangent structure induce natural transformations of the endofunctors of the bicategory involved in the tangent structure. This leads us to the conclusion that the homotopy category of this bicategory, that is, the category with Lie groupoids as objects and isomorphism classes of bibundles as morphisms is naturally equipped with a cartesian tangent structure with scalar $\bbR$-multiplication.

For a more refined approach, we have to consider the higher categorical aspects of the tangent structure, that is, the question of coherence relations between the associator and unit constraints of the bicategory on one side and the tangent structure on the other side. Fortunately, the composition of natural transformations of a bicategory is associative on the nose, so that the axioms of a tangent category expressed by the commutative diagrams of natural transformations should remain the same. The notion of tangent $\infty$-categories is being developed by Kristine Bauer, Matthew Burke, and Michael Ching; the category of differentiable stacks appears to be an example (private communication by Michael Ching).

\subsection{Geometric deformation theory}

Deformation theory is typically concerned with the study of moduli spaces of structures of a certain type, such as the multiplication map of an algebra, a complex structure, or a Riemannian metric, modulo isomorphisms. Deformations of a structure $X$ are \textit{smooth} paths $\varepsilon \mapsto X_{\varepsilon}$ in the corresponding moduli space $\CMcal{M}$ through $X = X_0$. Differentiating it at $\varepsilon=0$ yields an infinitesimal deformation
\begin{equation*}
\frac{d}{d\varepsilon} X_{\varepsilon}
\Bigr|_{\varepsilon=0}
\in  T_X{\CMcal{M}} 
\,.
\end{equation*}
Since moduli spaces typically have singularities, the differentiation often requires a derived approach in which the tangent space is replaced with a differential complex.

It has long been conjectured that deformation problems are governed by differential graded Lie algebras (dgLa) or $L_{\infty}$-algebras via Maurer-Cartan elements. Prototypical examples are the deformation theory of associative algebras (Hochschild cohomology) \cite{Gerstenhaber:1964}, Lie algebras (Chevalley-Eilenberg cohomology) \cite{Nijenhuis:1967}, and complex structures (Kodaira-Spencer cohomology) \cite{kodaira1986}. This conjecture is rigorously proved by Lurie \cite{lurie} and independently Pridham \cite{pridham} using the language of higher category theory. However, as promising as this sounds, constructing a computable model of the dgLa controlling a specific deformation problem turns out to be very difficult. As Kontsevich states: ``It is an `art' to discover these objects for a general deformation theory'' \cite[p.~6]{kontsevich}. 

For instance, it is surprising that while there is consensus that the deformation complex of a Lie group is given by its group cohomology with values in the adjoint representation on its Lie algebra, the Lie brackets of a dgLa or $L_{\infty}$-structure are not known. Crainic, Mestre, and Struchiner have generalized this to the deformation complex of Lie group\emph{oids} \cite{CrainicMestreStruchiner:2020}. As for Lie groups, the compatible Lie bracket on this complex has not yet been found.

Our approach to tackle the issue of singularities of $\CMcal{M}$ is the observation that many moduli spaces are equipped with a natural notion of smooth families parametrized by real parameters, which constitute a diffeology. For the moduli spaces mentioned above, it is the functional diffeology. Geometrically, $\CMcal{M}$ can then be viewed as a stack presented by a diffeological groupoid $\mathcal{G}_{\CMcal{M}}$. The infinitesimal first order deformations should then be given by the abstract Lie algebroid of $\mathcal{G}_{\CMcal{M}}$. To make these heuristic ideas rigorous, we have to address the following questions:
\begin{enumerate}

\item
How do tangent structures generalize to the higher categorical setting of the $(2,1)$-category of diffeological stacks?

\item
Is the diffeological groupoid $\mathcal{G}_{\CMcal{M}}$ presenting the stack $\CMcal{M}$ elastic?

\end{enumerate}
Our ultimate goal is to describe the deformation theory of geometric structures in terms of the jets of diffeological groupoids.

\subsection{Equivariant Lie groupoids}

Let $G$ be a Lie groupoid. A smooth right $G$-bundle is a submersion $r_E: E \to G_0$ with a right $G$-action $E \times_{G_0}^{r_E, t} G_1 \to E$, as studied in detail in Section~\ref{sec:Gbun}. The map that sends $E$ to the vertical tangent bundle $V\!E$ with its right $G$-action is a tangent functor.

In this tangent category, a vector field is a $G$-equivariant section $v: E \to V\!E$ of the vertical tangent bundle $V\!E \to E$, that is, an invariant vector field on $E$. The Lie bracket is the bracket of invariant vector fields (Theorem~\ref{thm:InvariantBracket}). The product of two $G$-bundles is given by the fiber product $E \times_{G_0} E'$ with the diagonal right $G$-action. The vertical tangent bundle satisfies $V(E \times_{G_0} E') \cong V\!E \times_{G_0} V\!E'$, which shows that the tangent category is cartesian. 

Every smooth manifold $M$ can be viewed as $G$-bundle $\pr_2: M \times G_0 \to G_0$ with right $G$-action $\big(m,t(g)\big) \cdot g = \big(m, s(g)\big)$. A smooth map $f: M \to N$ gives rise to the morphism $f \times \id_{G_0}: M \times G_0 \to N \times G_0$ of $G$-bundles. This defines a functor from the category of smooth manifolds to the category of $G$-bundles. The functor preserves finite products, which implies that the ring object $\bbR$ in $\Mfld$ is mapped to a ring object $R = (\bbR \times G_0 \to G_0)$ in the category of $G$-bundles. By identifying $(\bbR \times G_0) \times_{G_0} E \cong \bbR \times E$, the $\bbR$-multiplication of vertical tangent vectors can be viewed as a scalar multiplication by the ring object $R$.

Differentiable groupoid objects $E$ and their abstract Lie algebroids in the tangent category of $G$-bundles look like a promising subject for future work. For now, we can identify a few relevant special cases: If $G$ is a group and $E = P \times_X P \to X$ is the fiber groupoid of a principal $G$-bundle $P \to X$ with the diagonal $G$-action, then we obtain an equivariant description of the gauge groupoid. If $G$ is a group and the bundle is $E = G$ with the regular right action, then we get the notion of a Lie group with a compatible groupoid structure. This should be closely related to strict Lie 2-groups or, equivalently, smooth crossed modules. If $E = G$ is a groupoid with the regular right action, we have a notion of a groupoid with a compatible groupoid structure. This might be related to the notion of strict Lie 2-groupoids. If $E = G_0$ with the usual right action of the groupoid on its base, we obtain a groupoid with a singular multiplicative foliation. It would be very interesting to carry out the differentiation procedure in all these cases.

\subsection{Affine schemes}
\label{sec:AffSchemes}

The category $\Aff_k \cong \CAlg_k^\op$ of affine schemes over a field $k$ is equipped with a tangent functor given by the K\"ahler differentials on the algebras \cite[Section~4]{CruttwellLemay:2023}. More precisely, for any $k$-algebra $A$, let $\Omega^1_A$ denote the $A$-module of K\"ahler differentials, that is, the $A$-module generated by the symbols $da$ for all $a \in A$ subject to the relations $d 1 = 0$, $d(a+b) = da + db$, and $d(ab) = da\,b + a\, db$. The map $d: A \to \Omega^1_A$, $a \mapsto da$ is the universal derivation of $A$. Any other derivation of $A$ with values in some $A$-module $M$ is given by composing $d$ with a unique $A$-linear map $\Omega^1_A \to M$. In other words, the map $d^*: \Mod_A(\Omega^1_A, M) \to \Der_k(A, M)$, $\phi \mapsto \phi \circ d$ is a natural isomorphism of $A$-modules.

The tangent object of $A \in \CAlg_k^\op$ is $\Sym_A( \Omega^1_A)$, the free commutative unital $A$-algebra generated by the $A$-module $\Omega^1_A$. Since $\Sym_A$ is the left adjoint to the forgetful functor $\CAlg_A \to \Mod_A$, we have
\begin{equation}
\label{eq:AffDerRepr}
  \CAlg_A\bigl( \Sym_A(\Omega^1_A), B \bigr)
  \cong
  \Der_k(A, B)
\end{equation}
for any $A$-algebra $B$. The tangent functor on an affine scheme is then defined by
\begin{equation*}
  T\Spec(A)
  := \Spec\bigl( \Sym_A(\Omega^1_A) \bigr)
  \,.
\end{equation*}
The natural transformations of the tangent structure are rather obvious. The bundle projection $T\Spec A \to \Spec A$ is given by the unit $\pi_A: A \hookrightarrow \Sym_A(\Omega^1_A)$; the zero section is given by the augmentation map $0_A:\Sym_A(\Omega^1_A) \to A$; the fiberwise addition $+_A: \Sym_A(\Omega^1_A) \to \Sym_A(\Omega^1_A) \otimes_A \Sym_A(\Omega^1_A)$ is the $A$-linear map defined by $da \mapsto da \otimes 1 + 1 \otimes da$. For the symmetric structure and the vertical lift, we refer to \cite[Section~4]{CruttwellLemay:2023}.

In $\Aff_k$, the ring of scalars is the affine line $\bbA_k = \Spec( k[x])$. The scalar multiplication $\bbA_k \times T\Spec(A)  \to T\Spec(A)$ is the $A$-linear $k[x]$-coaction
\begin{align*}
  \kappa_A: \Sym_A(\Omega^1_A) 
  &\longrightarrow k[x] \otimes \Sym_A(\Omega^1_A)
  \\
  a 
  &\longmapsto 1 \otimes a 
  \\
  da 
  &\longmapsto x \otimes a
  \,.
\end{align*}
It is straightforward to check that this defines a scalar $\bbA_k$-multiplication in the sense of Definition~\ref{def:RScalar}. The ring of $\bbA_k$-valued functions on $\Spec(A)$ is the ring itself,
\begin{equation*}
  \Aff_k\bigl( \Spec(A), \bbA_k \bigr) \cong A
  \,.
\end{equation*}
The Lie algebra of vector fields on $\Spec A$ are the derivations of $A$,
\begin{equation*}
  \calX\bigl( \Spec(A) \bigr) \cong \Der_k(A)
  \,,
\end{equation*}
with the usual commutator bracket.

\begin{Proposition}
An abstract Lie algebroid in the tangent category $\Aff_k$ with scalar $\bbA_k$-multiplication is given as follows:
\begin{itemize}

\item[(i)] The bundle of $\bbA_k$-modules $\Spec(B) \to \Spec(A)$ is given by a commutative $k$-algebra $A$ and a commutative $A$-algebra $B \in \CAlg_A$ with an $A$-linear coaction by $k[x]$.

\item[(ii)] The space of sections of the bundle is given by $\CAlg_A(B, A)$. The action by $\Aff_k(\Spec(A), \bbA_k) \cong A$ is given by pointwise $A$-multiplication, $(a \phi)(b) = a\, \phi(b)$ for all $\phi \in \CAlg_A(B, A)$.

\item[(iii)] The anchor is given by a derivation $\rho \in \Der_k(A,B)$. The pullback
\begin{align*}
  \rho^*\colon \CAlg_A(B, A) 
  &\longrightarrow \Der_k(A)
  \\
  \phi &\longmapsto \phi \circ \rho
\end{align*}
is the induced map from sections of the bundle to vector fields on its base.

\end{itemize}
The Lie algebroid bracket is a $k$-linear Lie bracket on $\CAlg_A(B, A)$, satisfying
\begin{align*}
  [\phi,a\psi]
  &= a[\phi, \psi] + \bigl( (\rho^*\phi)(a)\bigr) \psi
  \\
  \rho^*[\phi, \psi]
  &= [\rho^*\phi, \rho^*\psi]
  \,,
\end{align*}
for all $\phi,\psi \in \CAlg_A(B, A)$ and $a \in A$.
\end{Proposition}

\begin{proof}[Sketch of proof]
A morphism $\Spec(B) \to \Spec(A)$ in $\Aff_k$ is a homomorphism $i: A \to B$ in $\CAlg_k$, which is the same as equipping $B$ with the structure of an commutative $A$-algebra given by $a\cdot b = i(a)\,b$. The anchor is a bundle morphism $\Spec(B) \to T\Spec(A)$, which is an $A$-linear homomorphism $\Sym(\Omega^1_A) \to B$, which in turn can be identified with a derivation $\rho \in \Der_k(A,B)$. The precomposition of $\phi \in \CAlg_A(B, A)$ with $\rho$ is a $k$-linear map
\begin{equation*}
  \rho^*: 
  \CAlg_A(B, A)
  \longrightarrow \Der_k(A)
  \,.
\end{equation*}
The other structure of a Lie algebroid is straightforward to identify.
\end{proof}

\begin{Example}
For a Lie algebroid $\pi: E \to M$ in smooth manifolds, we set $A := C^\infty(M)$ and $B := C^\infty(E)$. The homomorphism $i := \pi^*: A \to B$ equips $B$ with the structure of a $A$-algebra. The space of sections is isomorphic to $\Gamma(M,E) = \CAlg_A(B,A)$. Vector fields on $M$ are derivations of $C^\infty(M)$. In this way, we recover the usual notion of a Lie algebroid.
\end{Example}

When we apply our Lie functor to an affine group scheme, we recover the usual Lie algebra of invariant derivations of the commutative Hopf algebra \cite[Chapter~12]{Waterhouse:1971}. An affine groupoid scheme is given by a commutative Hopf algebroid. We posit that if the source (and therefore the target) map is a smooth morphism of affine schemes, then the groupoid is differentiable in the sense of Definition~\ref{def:DiffGroupoid}. We further conjecture that this statement generalizes to arbitrary groupoid schemes. We do not know whether having smooth source and target morphisms is necessary for a groupoid scheme to be differentiable. Computing and studying the abstract Lie algebroids of groupoid schemes is a promising topic for future research, with potential applications to arithmetic geometry.

\subsection{Non-examples}

\subsubsection{Abelian groups}

The category of abelian groups has an abstract tangent functor given by
\begin{equation*}
  TA := A \times A
\end{equation*}
on objects and $Tf := f \times f$ on morphisms $f: A \to B$. The fiberwise addition, zero section, symmetric structure, and vertical lift are defined exactly as for Euclidean spaces (Section~\ref{sec:TangentEuclidean}). A vector field is given by a morphism of abelian groups $v: A \to A$. The bracket is given by the commutator $[v, w] = w \circ v - v \circ w$. In other words the Lie algebra of vector fields on $A$ is given by $\Gamma(A, TA) = \mathrm{gl}(A)$, the Lie algebra of $\bbZ$-linear maps.

By construction, the tangent category of abelian groups is cartesian. One might guess that there is a scalar $\bbZ$-multiplication. However, $\bbZ$ is not a ring object in the category of abelian groups with the categorical product. (The multiplication $\bbZ \times \bbZ \to \bbZ$ is bilinear, not linear.) To fix this, we have to use the tensor product of $\bbZ$-modules, which would require substantial generalization of the concepts used in this paper.

\subsubsection{\texorpdfstring{$R$}{R}-modules}

The same observations as for abelian groups apply, analogously, to the category of $R$-modules by replacing $\bbZ$ with $R$. The upshot is that in the tangent category of $R$-modules the Lie algebra of vector fields is given by the commutator Lie algebra of the ring of endomorphisms. However, we do not have a natural scalar multiplication. Therefore, the general results of this paper do not apply.

\subsubsection{Categories with biproducts}
\label{sec:ExBiprod}

The last two examples generalize to any category $\calC$ with finite biproducts. Since the terminal object is the initial object, $* = \emptyset = 0$, the unit of addition and the unit of multiplication of any ring object in $\calC$ are the same. This shows that $\calC$ only contains the zero ring. As a consequence, the only $R$-module that exists in $\calC$ is the zero module, which implies that $TX = X \oplus 0 \cong X$. We conclude that the only possible tangent structure with scalar multiplication on $\calC$ is the trivial tangent structure.\footnote{We thank the anonymous referee for this observation.}

\subsubsection{Commutative rings}
\label{ex:CRing}

One of the original examples of Rosick\'y was the category of commutative rings, where the tangent functor maps a ring $A$ to its square zero extension $TA := A[\epsilon]/(\epsilon^2)$. However, there are no interesting differential objects in this tangent category \cite[Corollary~3.12]{CruttwellLemay:2023}, so that it cannot be equipped with a scalar multiplication.
\appendix
\section{Useful categorical lemmas}
\label{app:Cat}

For the convenience of the reader, we fix some notation and gather some categorical facts that we use extensively throughout the paper. For detailed proofs of the well-known results in this appendix, we refer the reader to \cite[Appendix~A.4]{Aintablian:2024}. Whenever we talk about certain limits in a category, we assume their existence. Let $\catC$ and $\catD$ be categories.

\subsection{Pullbacks and the paranthesis notation}
\label{sec:Pasting}

Let $A: \catI \to \catC$, $i \mapsto A_i$ and $B: \catI \to \catC$, $i \mapsto B_i$ be two diagrams indexed by a small category $\catI$. Let $f:A \to B$ be a natural transformation with components $\{f_i:A_i \to B_i\}_{i \in \catI}$. Then, by the universal property of limits, we obtain a unique morphism in $\catC$ denoted by
\begin{equation*}
\Lim_{i \in \catI} {f_i}: \Lim_{i \in \catI} A_i \longrightarrow \Lim_{i \in \catI} B_i \,,
\end{equation*} 
making all diagrams induced by the limiting cones commute. 

If $\catI = \{0 \xrightarrow{p} 2 \xleftarrow{q} 1\}$ so that diagrams indexed by $\catI$ are pullback diagrams, $f$ is given by a commutative diagram
\begin{equation*}
%\label{eq:PullbacksDiag}
\begin{tikzcd}
A_0 \ar[r, "A(p)"] \ar[d, "f_0"'] &[1em] A_2 \ar[d, "f_2"'] &[1em] A_1 \ar[l, "A(q)"'] \ar[d, "f_1"]
\\
B_0 \ar[r, "B(p)"'] & B_2 & B_1 \ar[l, "B(q)"]
\end{tikzcd}
\end{equation*}
The induced unique morphism between the pullbacks will be denoted by 
\begin{equation}
\label{eq:PullbackMap}
f_0 \times_{f_2} f_1: A_0 \times_{A_2}^{A(p), \, A(q)} A_1 \longrightarrow B_0 \times_{B_2}^{B(p), \, B(q)} B_1 \,.
\end{equation}
It makes the following squares 
\begin{equation*}
\begin{tikzcd}
A_0 \ar[d, "f_0"'] &[1em] A_0 \times_{A_2} A_1 \ar[d, "f_0 \times_{f_2} f_1"'] 
\ar[r, "\pr_2"] \ar[l, "\pr_1"'] &[1em] A_1 \ar[d, "f_1"]
\\
B_0 & B_0 \times_{B_2} B_1 \ar[r, "\pr_2"'] \ar[l, "\pr_1"] & B_1
\end{tikzcd}
\end{equation*}
commute. If $B_2=A_2$ and $f_2 = \id_{A_2}: A_2 \to A_2$, the morphism~\eqref{eq:PullbackMap} will be often written as 
\begin{equation*}
f_0 \times_{A_2} f_1: A_0 \times_{A_2} A_1 \longrightarrow B_0 \times_{A_2} B_1 \,.
\end{equation*}

Now, let $A$, $B$ and $X$ be objects of $\catC$ and let $f:X \to A$ and $g:X \to B$ be morphisms in $\catC$. Then, the unique morphism induced by the universal property of the product $A \times B$ will be denoted by $(f,g): X \to A \times B$. It makes the following diagram 
\begin{equation*}
\begin{tikzcd}[row sep=2em]
& X \ar[d, dashed, "{\exists !}"', "{(f,g)}"] \ar[dr, "g"] \ar[dl, "f"'] &
\\
A & A \times B \ar[r, "\pr_2"'] \ar[l, "\pr_1"] & B
\end{tikzcd}
\end{equation*}
commutative.

A special case of this is when we consider two morphisms in the overcategory $\catC \Comma C$, given by the commutative diagram:
\begin{equation*}
%\label{diag:NotationParanthesis}
\begin{tikzcd}
& X \ar[d, "h"] \ar[dr, "g"] \ar[dl, "f"'] &
\\
A \ar[r, "\alpha"'] & C & B \ar[l, "\beta"]
\end{tikzcd}
\end{equation*}
Then, the unique morphism $(f,g)$ induced by the universal property of the pullback $A \times^{\alpha,\beta}_C B$ makes the following diagram 
\begin{equation}
\label{diag:(fg)}
\begin{tikzcd}
X 
\ar[ddr, "f"', bend right] 
\ar[rrd, "g", bend left] 
\ar[dr, "{\exists! \, (f,g)}", dashed]
&&
\\
&
A \times_C B
\ar[r, "\pr_2"]
\ar[d, "\pr_1"']
&
B
\ar[d, "\beta"]
\\
&
A
\ar[r, "\alpha"']
&
C
\end{tikzcd}
\end{equation}
commutative. Note that products in the overcategory $\catC \Comma C$ are pullbacks over $C$.

\begin{Caution}
We use the same notation $(f,g)$ for the unique morphism from $X$ to the product $A \times B$ or to the pullback $A \times_C B$. It will be clear from the context whether we are in the category $\catC$ or $\catC \Comma C$.
\end{Caution}

\begin{Lemma}
\label{lem:IncPullback}
Let $f,g:A \to B$ be morphisms in $\catC$. Then, the natural morphism $A \times_B^{f,g} A \to A \times A$ is a monomorphism.
\end{Lemma}

\begin{Lemma}
\label{lem:MapsinPar}
Consider the following commutative diagram:
\begin{equation*}
\begin{tikzcd}
& X \ar[d, "h"] \ar[dr, "g"] \ar[dl, "f"'] &
\\
A \ar[r, "\alpha"'] \ar[d, "\varphi"'] & C \ar[d, "\id_C"'] & B \ar[l, "\beta"] \ar[d, "\psi"]
\\
D \ar[r, "r"'] & C & E \ar[l, "s"]
\end{tikzcd}
\end{equation*}
Let $k:Y \to X$ be a morphism in $\catC$ and let $(f,g):X \to A \times_C B$ be the unique morphism from Diagram~\eqref{diag:(fg)}. Then, the following equations hold:
\begin{itemize}
    \item[(i)] $(f,g) \circ k = (f \circ k, g \circ k)$,
    \item[(ii)] $(\varphi \times_C \psi) \circ (f,g) = (\varphi \circ f, \psi \circ g)$.
\end{itemize}
\end{Lemma}

\begin{Lemma}
\label{lem:FiberProdMaps}
Let $F,G: \catC \to \catD$ be functors and $\alpha: F \to G$ a natural transformation. Let $A \rightarrow C \leftarrow B$ be a pullback diagram in $\catC$ such that $FA \rightarrow FC \leftarrow FB$ and $GA \rightarrow GC \leftarrow GB$ are pullback diagrams in $\catD$. Then, the following diagram
\begin{equation*}
\begin{tikzcd}[column sep=large, row sep=large]
F(A \times_C B) \ar[r, "\alpha_{A \times_C B}"] \ar[d, "{(F\pr_1, \, F\pr_2)}"'] & G(A \times_C B) \ar[d, "{(G\pr_1, \, G\pr_2)}"]
\\
FA \times_{FC} FB \ar[r, "\alpha_A \times_{\alpha_C} \alpha_B"'] & GA \times_{GC} GB
\end{tikzcd}
\end{equation*}
commutes, where $\pr_1:A \times_C B \to A$ and $\pr_2:A \times_C B \to B$ are the projections onto the first and second factor respectively.
\end{Lemma}

\subsection{Nested commutative squares}
\label{sec:OuterInner}

The following lemma provides a simple, yet powerful technique to infer the commutativity of a square that is a ``restriction'' of another commutative square. The proof of the commutativity of numerous diagrams in this paper is based on this lemma.

\begin{Lemma}
\label{lem:InnerOuterSquares}
Consider the following diagram:
\begin{equation*}
\begin{tikzcd}[column sep=1.5em, row sep=1.5em]
A 
\ar[rrr, "f"] 
\ar[ddd, "g"'] 
\ar[dr, "\alpha"]
& & &
B \ar[ddd, "h"] 
\ar[dl, "\beta"']
\\
&
A' 
\ar[d, "g'"'] 
\ar[r, "f'"] 
&
B' \ar[d, "h'"]
&
\\
&
C' \ar[r, "k'"'] &
D'
&
\\
C \ar[rrr, "k"']
\ar[ur, "\gamma"'] 
& & &
D 
\ar[ul, "\delta"]
\end{tikzcd}
\end{equation*}
Assume that the four outer trapezoids commute. If the inner square commutes and $\delta$ is a monomorphism, then the outer square commutes.
\end{Lemma}

\begin{proof}
From the commutativity of the right and the top trapezoids it follows that
\begin{equation}
\label{eq:innerouter1}
\begin{split}
\delta \circ h \circ f
&=
h' \circ \beta \circ f
\\
&= h' \circ f' \circ \alpha
\,.
\end{split}
\end{equation}
Similarly, from the commutativity of the left and the bottom trapezoids it follows  that
\begin{equation}
\label{eq:innerouter2}
\begin{split}
\delta \circ k \circ g
&=
k' \circ \gamma \circ g
\\
&= k' \circ g' \circ \alpha
\,.
\end{split}
\end{equation}
By assumption, the inner square commutes, that is, $h' \circ f' = k' \circ g'$. Comparing Equations~\eqref{eq:innerouter1} and~\eqref{eq:innerouter2}, we get that
\begin{equation*}
\delta \circ h \circ f
=
\delta \circ k \circ g
\,.
\end{equation*}
Since, by assumption, $\delta$ is a monomorphism, it follows that $h \circ f = k \circ g$. In other words, the outer square commutes. 
\end{proof}

\subsection{Module objects}
\label{sec:GrpRngMod}

Assume that the category $\catC$ has finite products. Denote its terminal object by $*$.

\begin{Definition}
\label{def:RModDef}
Let $(R, \Radd, \Rzero, \Rmult, \Runit)$ be a ring object in $\catC$. An \textbf{$R$-module object} in $\catC$ is an abelian group object $(M,+,0)$ in $\catC$ together with a morphism $\kappa: R \times M \to M$, called the \textbf{action}, such that the following diagrams commute:
\begin{itemize}
\item[(i)] Associativity:
\begin{equation*}
\begin{tikzcd}[column sep={large}]
R \times R \times M \ar[r, "\id_R \times \kappa"] 
\ar[d, "\Rmult \times \id_{M}"'] &
R \times M \ar[d, "\kappa"]
\\
R \times M \ar[r, "\kappa"'] & M
\end{tikzcd}
\end{equation*}

\item[(ii)] Unitality:
\begin{equation*}
\begin{tikzcd}[column sep=3em]
* \times M \ar[r, "\Runit \times \id_{M}"] \ar[dr, "\cong", "\pr_2"'] & 
R \times M \ar[d, "\kappa"]
\\
& M
\end{tikzcd}    
\end{equation*}

\item[(iii)] Linearity in $R$:
\begin{equation*}
\begin{tikzcd}[column sep={large}, row sep=3em]
R \times R \times M
\ar[r, "\Radd \, \times \, \id_{M}"] 
\ar[d, "{\big( \kappa \circ (\pr_1, \pr_3), \, \kappa \circ (\pr_2, \pr_3) \big)}"'] &
R \times M 
\ar[d, "\kappa"]
\\
M \times M 
\ar[r, "+"'] 
&
M
\end{tikzcd}
\end{equation*}

\item[(iv)] Linearity in $M$:
\begin{equation*}
\begin{tikzcd}[column sep={large}, row sep=3em]
R \times M \times M
\ar[r, "\id_R \, \times \, +"] 
\ar[d, "{\bigl( \kappa \circ (\pr_1, \pr_2), \, \kappa \circ (\pr_1, \pr_3) \bigr)}"'] &
R \times M 
\ar[d, "\kappa"]
\\
M \times M 
\ar[r, "+"'] 
&
M
\end{tikzcd}
\end{equation*}
\end{itemize}
\end{Definition}

Linearity in $R$, expressed by~(iii), implies that the action by $\Rzero: * \to R$ is zero, that is, the diagram
\begin{equation*}
%\label{diag:RzeroPreserve}
\begin{tikzcd}[column sep=4em,row sep=2em]
* \times M \ar[r, "{\Rzero \times \id_{M}}"] \ar[d, "{\pr_1}"'] & R \times M \ar[d, "{\kappa}"] \\
* \ar[r, "0"'] & M
\end{tikzcd}   
\end{equation*}        
is commutative. Similarly, linearity in $M$, expressed by~(iv), implies that the action sends zero to zero, that is, the diagram
\begin{equation*}
%\label{diag:RzeroPreserve2}
\begin{tikzcd}[column sep=4em,row sep=2em]
R \times * \ar[r, "{\id_R \times 0}"] \ar[d, "{\pr_2}"'] & R \times M \ar[d, "{\kappa}"] 
\\
* \ar[r, "0"'] & M
\end{tikzcd}   
\end{equation*}
commutes.

\begin{Definition}
A \textbf{morphism between two $R$-module objects} $M$ and $M'$ is a morphism $\varphi: M \to M'$ of abelian group objects in $\catC$ such that the following diagram
\begin{equation*}
\begin{tikzcd}
R \times M 
\ar[d, "\kappa"']
\ar[r, "\id_R \times \varphi"]
&[1em] 
R \times M'
\ar[d, "\kappa'"]
\\
M 
\ar[r, "\varphi"']
&
M'
\end{tikzcd}
\end{equation*}
commutes, where $\kappa$ and $\kappa'$ are the actions of $R$ on $M$ and $M'$ respectively.
\end{Definition}

\begin{Lemma}
\label{lem:PullbackRingHom0}
Let $\varphi:R' \to R$ be a morphism of ring objects in $\catC$. Let $M$ be an $R$-module object in $\catC$ with action $\kappa:R \times M \to M$. Then, the morphism $\kappa' := \kappa \circ (\varphi \times \id_M): R' \times M \to M$ equips $M$ with the structure of an $R'$-module.
\end{Lemma}

\begin{Remark}
Given a ring object $(R,\Radd, \Rzero, \Rmult, \Runit)$ in $\catC$ and an object $A$ of $\catC$, the set $\catC(A,R)$ of $R$-valued morphisms has the structure of a ring with addition and multiplication
\begin{equation*}
  f + g := \Radd \circ (f,g)
  \,,\quad
  fg := \Rmult \circ (f,g)
  \,,
\end{equation*}
for all $f, g \in \calC(A,R)$, with zero $A \xrightarrow{!_A} * \xrightarrow{\Rzero} R$, and unit $A \xrightarrow{!_A} * \xrightarrow{\Runit} R$. Here,  $!_A:A \to *$ denotes the unique morphism to the terminal object.
\end{Remark}

\begin{Lemma}
\label{lem:PullbackRingHom}
Let $(R,\Radd, \Rzero, \Rmult, \Runit)$ be a ring object in $\catC$ and $\alpha:B \to A$ a morphism in $\catC$. Then, the pullback
\begin{equation*}
\begin{aligned}
\alpha^*: \catC(A,R) 
&\longrightarrow
\catC(B,R)
\\
f
&\longmapsto
f \circ \alpha
\end{aligned}
\end{equation*}
is a morphism of rings.
\end{Lemma}

\begin{Lemma}
\label{lem:RModMap}
Let $R \in \catC$ be a ring object and let $M,N \in \catC$ be $R$-module objects. Let $f:M \to N$ and $g:N \to M$ be inverse morphisms in $\catC$, i.e.~$f \circ g = \id_N$ and $g \circ f = \id_M$. If $f$ is a morphism of $R$-modules, then so is $g$.
\end{Lemma}

\bibliographystyle{alpha}
%\bibliography{ClassicalFields,Groupoids}
\bibliography{Groupoids}

\end{document}